\documentclass[12pt,reqno]{article}
\usepackage{cite,epic,eepic,euscript,verbatim,amsmath,amssymb,amsthm,amsfonts,afterpage,float,bm,
paralist,graphicx,epsf,graphics,hyperref,cancel,color,fancyhdr,makeidx,epsfig,stmaryrd}

\usepackage{lipsum}
\usepackage[titletoc]{appendix}

\newcommand{\reff}[1]{{\rm (\ref{#1})}}

\allowdisplaybreaks

\newtheorem{theorem}{Theorem}[section]
\newtheorem{lemma}{Lemma}[section]

\newtheorem{definition}{Definition}[section]

\numberwithin{equation}{section}

\def\XXint#1#2#3{{\setbox0=\hbox{$#1{#2#3}{\int}$}
\vcenter{\hbox{$#2#3$}}\kern-.51\wd0}}

\newcommand{\R}{\mathbb{R}}            
\newcommand{\ve}{\varepsilon}          

\newcommand{\calN}{{\mathcal N}}

\newcommand{\calX}{\mathcal X}
\newcommand{\calV}{\mathcal V}

\begin{document}

\title{ 
The Calculus of Boundary Variations and 
the Dielectric Boundary Force in the 
Poisson--Boltzmann Theory 
for Molecular Solvation
} 

\author{
Bo Li
\thanks{
Department of Mathematics,
University of California, San Diego,
9500 Gilman Drive, MC 0112, La Jolla, CA 92093-0112, USA.
Email: bli@math.ucsd.edu}
\and
Zhengfang Zhang
\thanks{
College of Science, 
Hangzhou Dianzi University, 
Hangzhou, Zhejiang 310018, China. 
Email: zhengfangzhang@hdu.edu.cn
}
\and
Shenggao Zhou
\thanks{Department of Mathematics and Mathematical Center for Interdiscipline Research,
Soochow University, 1 Shizi Street, Suzhou 215006, Jiangsu, China. 
 Email: sgzhou@suda.edu.cn}
}

\date{\today}

\maketitle

\begin{abstract}

\noindent
In a continuum model of the solvation of charged molecules in an aqueous solvent, the classical 
Poisson--Boltzmann (PB) theory is generalized to 
include the solute point charges and the dielectric boundary 
that separates the high-dielectric solvent from the low-dielectric solutes. 
With such a setting, we construct an effective electrostatic free-energy functional of 
ionic concentrations, where the solute point charges are regularized by a reaction field. 
We prove that such a functional admits a unique minimizer in a class of admissible 
ionic concentrations and that the corresponding
electrostatic potential is the unique solution to the boundary-value problem of 
the dielectric-boundary PB equation.
The negative first variation of this minimum free energy with respect 
to variations of the dielectric boundary defines the normal component of 
the dielectric boundary force. 
Together with the solute-solvent interfacial tension
and van der Waals interaction forces, 
such boundary force drives an underlying charged molecular system to a stable equilibrium, 
as described by a variational implicit-solvent model.  
We develop an $L^2$-theory for the continuity and differentiability of solutions to 
elliptic interface problems with respect to boundary variations,  
and derive an explicit formula of the dielectric boundary force. 
With a continuum description, our result of the dielectric boundary force
confirms a molecular-level prediction that
the electrostatic force points from the high-dielectric and polarizable
aqueous solvent to the charged molecules. 
Our method of analysis is general as it does not rely on any variational principles.



\bigskip

\end{abstract}


\tableofcontents

\medskip

\noindent
{\bf Acknowledgements}
\hfill
{\bf 48}

\medskip

\noindent
{\bf References}
\hfill
{\bf 48}





{\allowdisplaybreaks

\section{Introduction}
\label{s:Introduction}

The classic Poisson--Boltzmann (PB) theory provides a continuum description 
of electrostatic interactions in an ionic solution 
through the PB equation
\cite{Andelman_Handbook95,
Chapman1913,
DebyeHuckel1923,
Fixman_JCP79,
Gouy1910} 
\begin{equation}
\label{ClassicPBE}
\nabla \cdot \ve \nabla \psi - B'(\psi) = - \rho \qquad \mbox{in } \Omega_0,   
\end{equation}
where $\Omega_0 \subseteq \R^3 $ is the region of the ionic solution, 
$\ve$ is the dielectric coefficient,  
$\rho: \Omega_0 \to \R $ represents the density of fixed charges, and 
$\psi:\Omega_0\to \R$ is the electrostatic potential. 
In \reff{ClassicPBE}, the function $B: \R \to \R$ is defined by 
\begin{equation}
\label{B}
B(s) = \beta^{-1} \sum_{j=1}^M c_j^\infty \left( e^{-\beta q_j s} - 1 \right) 
\qquad \forall s \in \R,   
\end{equation}
where $\beta = (k_{\rm B} T)^{-1}$ with 
$k_{\rm B}$ the Boltzmann constant and $T$ the temperature, 
$M$ is the total number of ionic species, 
$c_j^\infty$ is the bulk ionic concentration of the $j$th ionic species, 
and $q_j = z_j e$ is the charge of an ion of the $j$th species
with $z_j$ the valence of such an ion and $e$ the elementary charge.   
The PB equation \reff{ClassicPBE} is a combination of Poisson's equation 
\begin{equation*}
\nabla \cdot \ve \nabla \psi = - \left( \rho  + \sum_{j=1}^M q_j c_j \right)
\qquad \mbox{in } \Omega_0, 
\end{equation*}
where $c_j: \Omega_0 \to [0, \infty)$ is 
the ionic concentration of the $j$th ionic species,  
and the Boltzmann distributions for the equilibrium ionic concentrations
\begin{equation*}
c_j(x) = c_j^\infty e^{-\beta q_j \psi(x)}, \qquad x\in \Omega_0, \ j = 1, \dots, M.    
\end{equation*}

In modeling charged molecules (such as proteins) in an aqueous solvent (i.e., water or salted water)
within an implicit-solvent (i.e., continuum-solvent) framework, 
the PB theory is generalized to include the point charges of 
the charged molecules and a dielectric boundary that separates the high-dielectric solvent 
region from the low-dielectric solute region 
\cite{CDLM_JPCB08,
CramerTruhlar_ChemRev99,
TomasiPersico_ChemRev94,
DavisMcCammon_ChemRev90,
Li_SIMA09,
SharpHonig_Rev90}.   
To be more specific, let us assume that the entire solvation system occupies 
a region $\Omega \subseteq \R^3.$  It is the union of three disjoint parts: 
the region of solutes (i.e., charged molecules) $\Omega_-;$ 
the region of aqueous solvent $\Omega_+;$ 
and the solute-solvent interface or dielectric boundary $\Gamma$, which 
is a closed surface with possibly multiple components, that 
separates $\Omega_-$ and $\Omega_+;$ cf.\ Figure~1. 
We denote by $n$ the unit normal to the boundary $\Gamma$ pointing from $\Omega_-$ 
to $\Omega_+$ and also the exterior unit normal to $\partial \Omega,$ 
the boundary of $\Omega.$  The solute region $\Omega_-$ contains all the solute 
atoms that are located at $x_1, \dots, x_N$ and that carry partial charges $Q_1, \dots, Q_N$, 
respectively, where $N \ge 1$ is a given integer.  The solvent region $\Omega_+$ 
is the region of ionic solution, similar to $\Omega_0$ in \reff{ClassicPBE}. 
As before, we assume that there are $M$ species of ions in 
the solvent region $\Omega_+$ with the valence $z_j$, charge $q_j = z_j e$, 
bulk concentration $c_j^\infty$, and the local concentration $c_j: \Omega_+ \to [0, \infty)$
  for the $j$ ionic species $(j = 1, \dots, M).$ 
The dielectric coefficients in the solute region $\Omega_-$ and solvent region $\Omega_+$ are 
denoted by $\ve_-$ and $\ve_+$, respectively. 
Typically, $\ve_- = 1$ and $\ve_+ = 76 \sim 80$ in the unit of vacuum permittivity.
Note that the density of fixed charges is now 
 given by $ \rho = \sum_{i=1}^N Q_i \delta_{x_i},$ 
where $\delta_{x_i}$ is the Dirac delta function at $x_i.$

\begin{figure}[bpht]
\begin{center}
\includegraphics[scale=0.4]{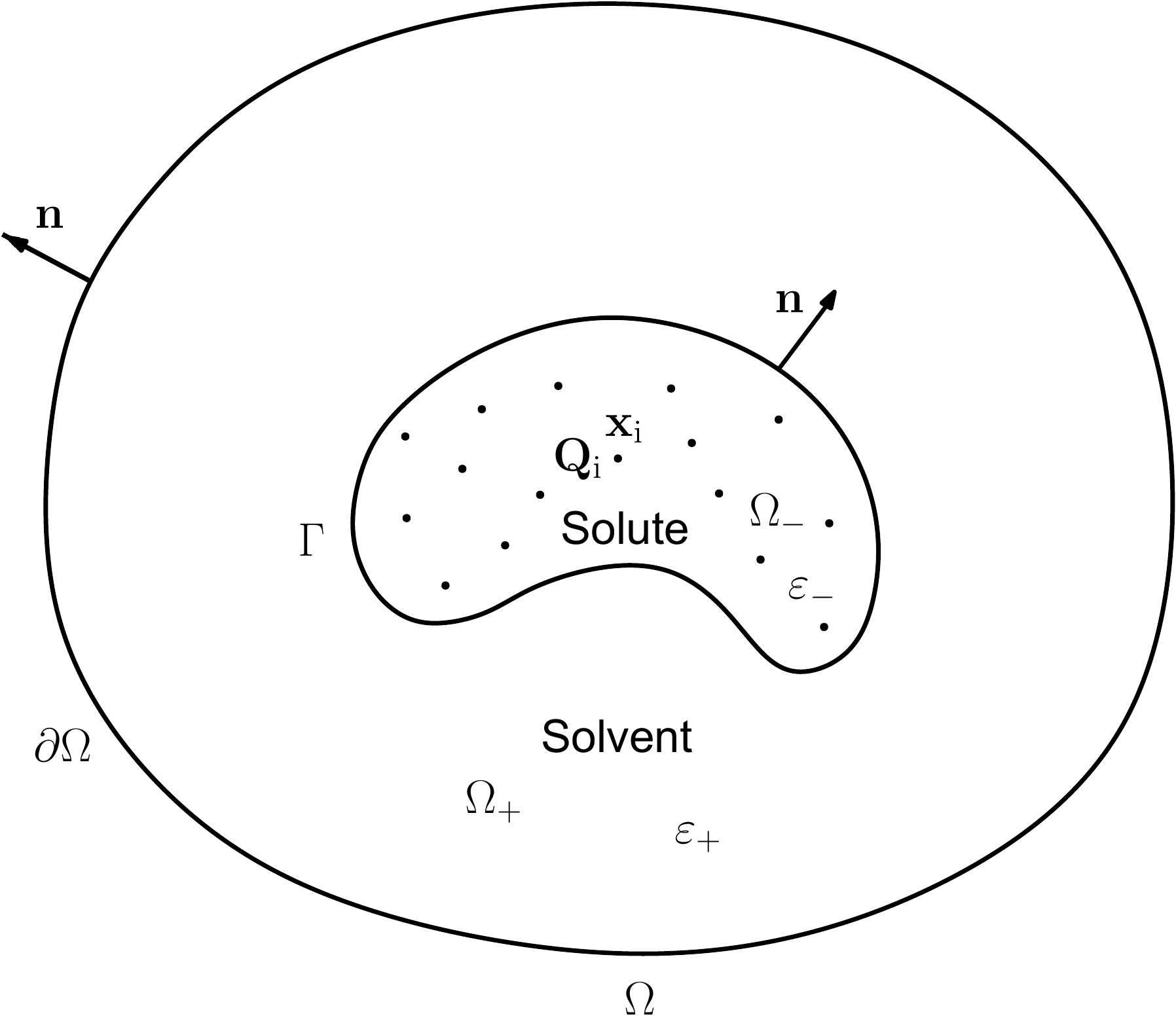}
\end{center}

\vspace{-6 mm}

\caption{
A schematic description of a solvation system with an implicit solvent. 
}
\label{f:cartoon}
\end{figure}

We introduce the dielectric-boundary, electrostatic free-energy functional of 
the ionic concentrations $c = (c_1, \dots, c_M)$ in the solvent region $\Omega_+$ 
\cite{ReinerRadke90,CDLM_JPCB08, Li_SIMA09, FogolariBriggs97}
\begin{align}
\label{FGammac}
F_\Gamma [c] &= \frac12  \sum_{i=1}^N Q_i ( \psi - \hat{\phi}_{\rm C} ) (x_i) 
+ \frac12 \int_{\Omega_+}  \left( \sum_{j=1}^M q_j c_j \right) \psi \, dx \nonumber \\
&\qquad +\beta^{-1} \sum_{j=1}^M \int_{\Omega_+} \left\{ c_j \left[ \log (\Lambda^3 c_j )
- 1\right] + c_j^{\infty}\right\}dx  - \sum_{j=1}^M \int_{\Omega_+} \mu_j  c_j   \, dx,    
\end{align}
where $\Lambda $ is the thermal de Broglie wavelength, 
$\mu_j $ is the chemical potential for ions of the $j$th species, and 
$c_j^\infty  = \Lambda^{-3} e^{\beta \mu_j}$ $(j = 1, \dots, M).$ 
In \reff{FGammac}, $\psi: \Omega \to \R$ is the electrostatic potential. It is  
the unique weak solution to the boundary-value problem of Poisson's equation 
\begin{equation}
\label{PoissonNew}
\left\{
\begin{aligned}
& \nabla \cdot \ve_\Gamma \nabla \psi = - \left( \sum_{i=1}^N Q_i \delta_{x_i} + 
\chi_+ \sum_{j=1}^M q_j c_j \right) &  & \quad \mbox{in } \Omega, 
\\
& \psi = \phi_\infty  &  & \quad \mbox{on } \partial \Omega,  
\end{aligned}
\right.
\end{equation}
where the dielectric coefficient $\ve_\Gamma: \Omega \to \R$ is defined by 
\begin{equation}
\label{veGamma}
\ve_\Gamma (x)  = 
\left\{
\begin{aligned}
& \ve_-  \quad & \mbox{if } x \in \Omega_-, \\
& \ve_+   \quad & \mbox{if } x \in \Omega_+,
\end{aligned}
\right.
\end{equation}
$\chi_+ = \chi_{\Omega_+}$ is the characteristic function of $\Omega_+,$ and 
$\phi_\infty$ is a given function on the boundary $\partial \Omega.$ 
The function $\hat{\phi}_{\rm C}$ in \reff{FGammac}
is the Coulomb potential arising from the point charges $Q_i$ at $x_i $ $(i = 1, \dots, N)$ 
in the medium with the dielectric coefficient $\ve_-$, serving as a reference field.  
It is given by 
\begin{equation}
\label{psi0}
\hat{\phi}_{\rm C} (x) = \sum_{i=1}^N \frac{Q_i}{ 4 \pi \ve_- |x - x_i|}
\qquad \forall x \in \R^3 \setminus \{ x_1, \dots, x_N \}.  
\end{equation}


We prove that the functional $F_\Gamma [c]$ has a unique minimizer $c_\Gamma 
= (c_{\Gamma,1}, \dots, c_{\Gamma,M})$ in a class of admissible concentrations, 
and derive the equilibrium conditions $\delta_{c_j} F_\Gamma [c_\Gamma] = 0$ 
$(j = 1, \dots, M),$ which lead
to the (modified) Botlzmann distributions 
\[
c_{\Gamma, j} = c_j^\infty e^{-\beta q_j ( \psi_\Gamma - \phi_{\Gamma, \infty}/2)}
\qquad \mbox{in } \Omega_+,  \ j = 1, \dots, M, 
\]
where $\psi_\Gamma$ 
is the corresponding electrostatic potential
and $\phi_{\Gamma,\infty}: \Omega \to \R$ is
the unique weak solution to the boundary-value problem
\begin{equation}
\label{hatpsiinfty}
\left\{
\begin{aligned}
&\nabla \cdot \ve_\Gamma \nabla \phi_{\Gamma,\infty} = 0 & & \quad  \mbox{in } \Omega, 
\\
& \phi_{\Gamma,\infty} = \phi_\infty & & \quad  \mbox{on } \partial \Omega.   
\end{aligned}
\right.
\end{equation}
We also prove that $\psi_\Gamma$ is the unique solution to the boundary-value problem of 
the dielectric-boundary PB equation  
\begin{equation}
\label{PBE}
\left\{
\begin{aligned}
&\nabla \cdot \ve_\Gamma \nabla \psi - \chi_+ B'\left( \psi - \frac{\phi_{\Gamma,\infty}}{2} \right)
= - \sum_{i=1}^N Q_i \delta_{x_i}   & & \quad \mbox{in } \Omega, 
\\
& \psi = \phi_\infty & & \quad \mbox{on } \partial\Omega,    
\end{aligned}
\right. 
\end{equation}
where $B$ is given in \reff{B};  
cf.\ Theorem~\ref{t:DBPBE} and Theorem~\ref{t:minFa}. 
With the Boltzmann distributions, the minimum free energy 
$ \min F_\Gamma[\cdot] = F_\Gamma[c_\Gamma ]$
can be expressed via the electrostatic potential $\psi_\Gamma.$  
We shall denote this minimum free energy by 
$E[\Gamma],$ as ultimately it  depends on the dielectric boundary $\Gamma$.

We define the (normal component of the) dielectric boundary force to be $-\delta_\Gamma E[\Gamma],$ the negative
first variation of the functional $E[\Gamma]$ with respect to the variation of the boundary $\Gamma$. 
The boundary variation is defined via a smooth vector field. Specifically, 
let $ V: \R^3 \to \R^3$ be a smooth map vanishing 
outside a small neighborhood of the dielectric boundary $\Gamma$. Let 
$x = x(t,X)$ be the solution map of the dynamical system defined by 
\cite{DelfourZolesio_Book87,KPTZ_LecNotes00,BucurButtazzo_Book05,Zolesio_Book92}
\[
\frac{dx(t,X)}{dt} = V(x(t,X)) \quad \forall  \, t \in \R
\qquad \mbox{and} \qquad  x(0, X) = X \quad \forall X \in \R^3. 
\]
Such solution maps define a family of transformations $T_t: \R^3 \to \R^3$
$(t \in \R)$ by $T_t(X) = x(t, X)$ for any $X\in \R^3.$  
The variational derivative (i.e., 
the shape derivative) of the functional $E[\Gamma]$ in the direction of $ V: \R^3 \to \R^3$ 
is defined to be 
\[
\delta_{\Gamma, V}E[\Gamma] = \frac{d}{dt} E [\Gamma_t(V)] \biggr|_{t = 0}, 
\]
if it exists, where $\Gamma_t(V) = \{ x(t,X): X \in \Gamma  \}$ 
and $E[\Gamma_t(V)]$ is defined similarly using $\Gamma_t(V)$ instead of $\Gamma.$

We prove that $\delta_{\Gamma, V} E[\Gamma]$ exists, and is
an integral over $\Gamma$ of the product of $V \cdot n$ and some function that is independent of $V$, 
where $n$ is the unit normal along $\Gamma,$ pointing from $\Omega_-$ to $\Omega_+$. 
This function on $\Gamma$ is identified
as the variational derivative of $E[\Gamma]$ and is denoted by $\delta_\Gamma E[\Gamma]$. 
We obtain an explicit formula for $\delta_\Gamma E[\Gamma]$.  
If the boundary value $\phi_\infty = 0$ on $\Gamma$, then
\begin{equation}
\label{deltaJ}
\delta_{\Gamma} E[\Gamma] =  - \frac{1}{2} \left( \frac{1}{\ve_+} - \frac{1}{\ve_-} \right)  
\left| \varepsilon_\Gamma \partial_n \psi_\Gamma \right|^2 + \frac{1}{2} ( \ve_+ - \ve_-)  
\left| \nabla_\Gamma \psi_\Gamma\right|^2 + B\left(\psi_\Gamma \right),      
\end{equation}
where $\psi_\Gamma$ is the unique solution to \reff{PBE}, 
$ \varepsilon_\Gamma \partial_n \psi_\Gamma$ is the common value from 
both sides of $\Gamma$, and $\nabla_\Gamma = ( I - n \otimes n) \nabla$ (with $I$ 
the $3\times 3$ identity matrix) is the tangential derivative along $\Gamma$. 
Additional terms arise from a general, inhomogeneous boundary value $\phi_\infty;$
cf.\ Theorem~\ref{t:DBF}. 

To describe the electrostatic free energy with point charges and to 
prove the main theorem, Theorem~\ref{t:DBF}, we introduce various auxiliary
functions that are weak solutions to the boundary-value problems of 
the operator $-\Delta$ or $- \nabla \cdot \ve_\Gamma \nabla,$ 
with or without the point charges $\sum_{i=1}^N Q_i \delta_{x_i}$ and with homogeneous
or inhomogeneous Dirichlet boundary conditions. 
We also prove several lemmas, Lemmas \ref{l:TForce}--\ref{l:w_t}, 
on the calculus of boundary variations.  Lemma~\ref{l:TForce} is of its own interest. 
It states that if the vector field $V$ satisfies $V \cdot n = 0$ on the boundary $\Gamma$, 
where $n$ is the unit normal along $\Gamma$, then for $|t| \ll 1$ the set $\Gamma_t
= \Gamma_t(V) $ is within an $O(t^2)$-neighborhood of the boundary $\Gamma.$   
Lemmas \ref{l:phi0}--\ref{l:w_t} are on the continuity and differentiability of 
those functions with respect to boundary variations. 
Lemma~\ref{l:pGinfty} states that the ``$\Gamma$-derivative" of the 
function $\phi_{\Gamma, \infty}$ which is defined in \reff{hatpsiinfty} is the unique
weak solution $\zeta_{\Gamma, V} \in H_0^1(\Omega)$ to 
the elliptic interface problem $-\nabla \cdot \ve_\Gamma \nabla \zeta_{\Gamma, V} = f$
in $\Omega$, where $f$ involves $\phi_{\Gamma, V}$ and $V$. Moreover, 
\[
\frac{ {\phi}_{\Gamma_t(V), \infty} \circ T_t - \phi_{\Gamma, \infty}}{t} \to \zeta_{\Gamma, V}
\quad \mbox{in } H^1(\Omega) \quad \mbox{as } t \to 0. 
\]
Lemma~\ref{l:xi} and Lemma~\ref{l:w_t} generalize the result to other $\Gamma$-dependent
functions, including the electrostatic potential $\psi_\Gamma$ that is the unique 
soltuion to the boundary-value problem of the dielectric-boundary PB equation \reff{PBE}.

We now make several remarks on our results.
In our model, we use an inhomogeneous Dirichlet boundary condition 
for the electrostatic potential (cf.\ \reff{PoissonNew} and 
\reff{PBE}) that is common in modeling and analysis 
\cite{CDLM_JPCB08,Li_SIMA09,VISMPB_JCTC14}.  
The nonzero Dirichlet boundary value leads to an extra term 
${\phi}_{\Gamma, \infty}/2$ in the Boltzmann distribution and hence in the PB equation \reff{PBE}. 
If there are surface charges on the boundary $\partial \Omega$, then 
one can also use the Neumann boundary condition for the electrostatic potential on $\partial\Omega.$ 
In that case, the electrostatic energy should include a boundary integral term involving
the surface charge density; cf.\ \cite{Lu_SIAP18,LiLiuXuZhou_Nonlinearity2013}.  

If we use the homogeneous Dirichlet boundary condition 
$\phi_\infty = 0$ for the electrostatic potential, 
then the dielectric boundary force points from the high dielectric solvent region 
$\Omega_+$ to the low dielectric solute region $\Omega_-$; cf.\ \reff{deltaJ}. 
Such prediction of a macroscopic property is consistent with a microscopic picture of 
molecular forces that charged solute molecules polarize the surrounding aqueous solvent, 
which is otherwise electrically neutral, 
generating an additional electric filed that attracts the solvent to the solutes \cite{Chu67}.  
In the limiting case where the region $\Omega_+$ is conducting, i.e., the dielectric
coefficient in $\Omega_+$ is infinity, then it is expected that no bounded 
region $\Omega_-$ will minimize the sum of the electrostatic energy
and the surface energy \cite{LuOtto_CPAM14}. 
If a small, high-dielectric solvent region is surrounded by the low-dielectric
solute molecules (such as a few water molecules buried in a protein), then 
the competition between the solute-solvent interfacial tension force
and the dielectric boundary force results an equilibrium solute-solvent interface
which is however unstable with long-wave perturbations, as shown in 
the stability analysis in \cite{ChengLiWhiteZhou_SIAP13}; cf.\ also  
\cite{LiSunZhou_SIAP15}. Such analysis explains partially why water molecules
in proteins are metastable \cite{YinHummerRasaiah_JACS07, Rasaiah_JPCB10}. 
It remains open to confirm if the dielectric boundary force still points from 
the high-dielectric solvent region to the low-dielectric solute region
for a general inhomogeneous Dirichlet boundary value $\phi_{\infty}$.  

In \cite{CaiLuo2011,XiaoLuo_JCP2013,Luo_PCCP12}, the authors use the Maxwell stress tensor to 
define and derive the dielectric boundary force given an electrostatic potential
that is determined by the dielectric-boundary PB equation.  
The existence of such a stress tensor in the presence of dielectric boundary is implicitly assumed. 
The shape derivative approach seems first introduced 
in \cite{LiChengZhang_SIAP11} to define and derive the 
dielectric boundary force. However, approximations of point charges by smooth functions are made
there, and the derivation of the boundary force utilizes heavily on 
the underlying variational principle that the electrostatic potential extremizes the 
dielectric-boundary PB free-energy functional. 
This approach is applied to the electrostatic force acting on membranes \cite{MikuckiZhou_SIAP14}. 
Here, we use the direct calculations to derive the boundary force, which is a more general approach. 

Our study of the dielectric boundary force is closely related to the development of 
a variational implicit-solvent model (VISM) for biomolecules  
\cite{DSM_PRL06,DSM_JCP06} (cf.\ also 
\cite{CDML_JCP07, 
CXDMCL_JCTC09,
CWSDLM_JCP09,
WangEtal_VISMCFA_JCTC12,
VISMPB_JCTC14,
Zhou_PNAS19}). 
Central in the VISM is an effective free-energy functional of all possible dielectric boundaries 
that consists mainly of the surface energy of solute molecules, 
solute-solvent van der Waals interaction energy, and continuum electrostatic free energy. 
Minimization of the free-energy functional with respect to the dielectric boundary 
yields optimal solute-solvent interfaces, as well as the solvation free energy. 
Numerical implementation of such minimization requires a formula of the first variation 
of the VISM function, particularly, the dielectric boundary force. 
In \cite{LiLiu_SIAP15}, the authors use the matched asymptotic analysis to derive
the  sharp-interface limit of a phase-field VISM \cite{Sun_PFVISM_JCP15}. 
In \cite{DaiLiLu_ARMA18}, 
the authors prove the convergence of the free energy and force
in the phase-field VISM to their sharp-interface counterparts. 
In particular, they prove the general result that the variation of the 
van der Waals--Cahn--Hilliard functional converges to the mean curvature which 
is the variation of surface area. 
The recent work \cite{GinsterGladbach_ARMA2020} is a detailed analysis
of the electrostatics in molecular solvation through different scaling regimes 
arising from the large-number limit of solute particles.

The rest of the paper is organized as follows: 
In Section~\ref{s:PB}, we first state our assumptions and introduce some 
auxiliary functions. We then prove the existence, uniqueness, and bounds
for the solution to the boundary-value problem of the dielectric-boundary PB equation. 
We finally   study the electrostatic free-energy functionals
of ionic concentrations and electrostatic potentials, respectively,  
with a given set of point charges and a dielectric boundary. 
In Section~\ref{s:DBF}, we reformulate the minimum electrostatic free energy, 
define the dielectric boundary force, and present the main formula for such force. 
In Section~\ref{s:Lemmas}, we prove several lemmas on the calculus of 
boundary variations. These lemmas are needed for the proof
of the main theorem on the dielectric boundary force. 
Finally, in Section~\ref{s:Proof}, we prove the main theorem
(Theorem~\ref{t:DBF}) of the dielectric boundary force. 


\section{The Poisson--Boltzmann Equation and Free-Energy Functional}
\label{s:PB} 

\subsection{Assumptions and Auxiliary Functions}
\label{ss:Assumptions}

Unless otherwise stated, we assume the following throughout the rest of the paper: 
\begin{compactenum}
\item[A1.]
The set $\Omega \subset \R^3$ is non-empty, bounded, open, and connected. 
The sets $\Omega_- \subset \R^3$ and $\Omega_+ \subset \R^3$ are
non-empty, bounded, and open, and satisfy that $\overline{\Omega_-}
\subset \Omega$ and $\Omega_+ = \Omega\setminus \overline{\Omega_-}.$  
The interface $\Gamma = \partial \Omega_- = \overline{\Omega_-} \cap \overline{\Omega_+} $ 
and the boundary $\partial \Omega$
are of the class $C^3$ and $C^2$, respectively. 
The unit normal vector at the boundary $\Gamma$ exterior to $\Omega_-$ 
and that at $\partial \Omega$ exterior to $\Omega$ are both denoted by $n.$ 
The $N$ points $x_1, \dots, x_N$ for some integer $N \ge 1$ belong to $\Omega_-$;  
cf.~Figure~\ref{f:cartoon}.   
Moreover, there exists a constant $s_0 > 0$ such that 
\begin{equation}
\label{s0}
\mbox{dist}\,(\Gamma, \partial\Omega) \ge s_0;
\end{equation}
\item[A2.]
All the integer $M \ge 2$, and real numbers $\beta > 0$, $\Lambda > 0$, 
$Q_i \in \R$ $(1 \le i \le N)$, $q_j \ne 0$ and 
$\mu_j \in \R$ $(1 \le j \le M)$, and $\ve_- > 0$
and $\ve_+ > 0$ are given. Moreover, $\ve_- \ne  \ve_+.$  
The parameter $c_j^\infty$ is defined by 
$c_j^\infty = \Lambda^{-3} e^{\beta \mu_j}$ $(j = 1, \dots, M).$ 
The parameters $q_j $  and $c_j^\infty $ $(1 \le j \le M)$ satisfy the 
condition of charge neutrality 
\begin{equation}
\label{neutrality}
\sum_{j=1}^M q_j c_j^\infty = 0;     
\end{equation}
\item[A3.]
The function $B: \R \to \R$ is defined in \reff{B}. 
The function $\ve_\Gamma \in L^\infty(\Omega) $ is defined in \reff{veGamma}. 
The boundary data $\phi_\infty$ is the trace
of a given function, also denoted by $\phi_\infty$, in $C^2 (\overline{\Omega})$. 
(We use the standard notation for Sobolev spaces and other function spaces; 
cf.\ \cite{Adams75,EvansBook2010,GilbargTrudinger98}.)
\end{compactenum}
Note that the function $B$ defined in \reff{B}  satisfies that $B \in C^\infty(\R)$.  
Since 
\[
B'(s) = -\sum_{j=1}^M q_j c_j^\infty e^{-\beta q_j s}
\quad \mbox{and} \quad 
B''(s) = \sum_{j=1}^M \beta q_j^2 c_j^\infty e^{-\beta q_j s} > 0, 
\quad \forall s \in \R,
\]
the function $B$ is strictly convex, and 
 the charge neutrality \reff{neutrality} implies that $B'(0) = 0$.  
Hence, $ s = 0$ is the unique minimum point for $B$ with 
$B(s) > B(0) = 0$ for all $s \ne 0$. 
By the fact that $M \ge 2$ and the charge neutrality \reff{neutrality}, 
there exist some $q_j > 0$ and some $q_k < 0$. Hence, $B(\pm \infty)  = \infty.$ 
Similar arguments show that $B'(\infty) = \infty$ and $B'(-\infty) = - \infty.$

We now introduce several auxiliary functions to treat the point-charge singularities, the dielectric 
discontinuity $\Gamma$, and the inhomogeneous boundary data $\phi_\infty$ on $\partial\Omega.$   
We first recall that the Coulomb field $\hat{\phi}_{\rm C}$ is defined in \reff{psi0}. 
Let $\hat{\phi} \in \hat{\phi}_{\rm C} + H^1(\Omega) $ be defined by 
\begin{equation}
\label{Wpsi0}
\int_{\Omega} \ve_- \nabla \hat{\phi} \cdot  \nabla \eta  \, dx
= \sum_{i=1}^N Q_i \eta(x_i) \qquad \forall \eta \in C_{\rm c}^1 (\Omega), 
\end{equation}
where  $C_{\rm c}^1(\Omega)$ denotes the class of $C^1(\Omega)$-functions that are 
compactly supported in $\Omega.$ Clearly, we can modify the value of 
$\hat{\phi}$ on a set of zero Lebesgue measure, if necessary, so that 
 $\hat{\phi}$ is a $C^\infty$-function in $\Omega \setminus \{ x_1, \dots, x_N \}.$ Moreover, 
$\Delta \hat{\phi} = 0$ in $ \Omega \setminus \{ x_1, \dots, x_N \}$ and 
$ \Delta(\hat{\phi}-\hat{\phi}_{\rm C})   = 0 $ in  $\Omega.$ 
There are infinitely many such functions.  We will only use
three of them. One of them is the Coulomb field $\hat{\phi} = \hat{\phi}_{\rm C}$. 
The other two are $\hat{\phi} = \hat{\phi}_0$ and 
$\hat{\phi} = \hat{\phi}_\infty$. They are uniquely determined by the boundary conditions
\begin{equation}
\label{hatphiBC}
\hat{\phi}_0 = 0  \quad \mbox{on } \partial \Omega \qquad \mbox{and} \qquad  
\hat{\phi}_\infty = \phi_\infty \quad \mbox{on } \partial \Omega, 
\end{equation}
respectively. Since $\partial \Omega$ is $C^2$ and $\phi_\infty \in C^2(\overline{\Omega})$, we
have $\hat{\phi}-\hat{\phi}_{\rm C} \in H^2(\Omega);$
cf.\ Chapter 8 in \cite{GilbargTrudinger98}. Therefore, 
all these three functions belong to 
$\hat{\phi}_{\rm C} + H^2(\Omega)\cap C^\infty(\Omega) \subset W^{1,1}(\Omega). $ 

We remark that $\eta \in C_{\rm c}^1(\Omega)$ in \reff{Wpsi0}
can be replaced by $\eta \in H_0^1(\Omega) $ with $\eta|_{\Omega_-} \in C^1(\Omega_-).$ 
To see this, we first note that \reff{Wpsi0} holds true if $\hat{\phi}$ is 
replaced by $\hat{\phi}_{\rm C}$ (cf.\ \reff{psi0}).  Thus, 
\[
\int_\Omega \ve_- \nabla ( \hat{\phi} - \hat{\phi}_{\rm C} )
\cdot \nabla \eta \, dx = 0 \qquad \forall \eta \in H_0^1(\Omega),
\]
as $ \hat{\phi} - \hat{\phi}_{\rm C} \in H^1(\Omega)$ and $C^1_{\rm c}(\Omega)$ is 
dense in $H^1(\Omega).$ If $\eta \in H_0^1(\Omega) $ also satisfies 
$\eta|_{\Omega_-} \in C^1(\Omega_-),$ then $\nabla \hat{\phi}_{\rm C}
\cdot \nabla \eta$, hence $\nabla \hat{\phi} \cdot\nabla \eta,$
is integrable in $\Omega.$ Moreover, 
\[
\int_\Omega \ve_- \nabla  \hat{\phi} \cdot \nabla \eta \, dx
= \int_\Omega \ve_- \nabla  \hat{\phi}_{\rm C} \cdot \nabla \eta \, dx
= \sum_{i=1}^N Q_i \eta (x_i),  
\]
where the second equality follows from straight forward calculations using the 
definition of $\hat{\phi}_{\rm C}$ (cf.\ \reff{psi0}). 


We recall that the function 
$\phi_{\Gamma, \infty} \in H^1(\Omega)$ is 
the unique weak solution to the boundary-value problem \reff{hatpsiinfty}, defined by 
${\phi}_{\Gamma,\infty}  = \phi_\infty$ on $ \partial\Omega$ and   
\begin{equation}
\label{hatpsiinftyweak}
\int_\Omega \ve_\Gamma \nabla {\phi}_{\Gamma,\infty} \cdot \nabla \eta \, dx  
= 0  \qquad   \forall \eta \in H^1_0(\Omega).  
\end{equation}
By the regularity theory, we have, after modifying possibly the value
of $\phi_{\Gamma, \infty}$ on a set of zero Lebesgue measure, that 
\begin{equation}
\label{phiGireg} 
\phi_{\Gamma, \infty} \in C(\overline{\Omega}) \cap W^{1,\infty}(\Omega) 
\quad \mbox{and} \quad \phi_{\Gamma,\infty}|_{\Omega_{\rm s}} 
\in C^\infty(\Omega_{\rm s}) \cap H^2(\Omega_{\rm s}) \quad \mbox{for } s = -, +. 
\end{equation}
Moreover, there exists a constant $C = C(\Omega, \ve_+, \ve_-, \phi_\infty) > 0$, 
independent of $\Gamma$, such that 
\begin{equation}
\label{phiGammabounds}
\|\phi_{\Gamma, \infty} \|_{W^{1, \infty}(\Omega)}  \le C. 
\end{equation}
See \cite{GilbargTrudinger98} (Theorem~8.16) and \cite{LiVogelius_ARMA2000} 
(Theorem~1.1 and the beginning part of proof of Theorem~1.1) 
(also \cite{Chipot_ARMA86}) for the global $C(\overline{\Omega})$
and $W^{1,\infty}$ regularities, and the $W^{1,\infty}(\Omega)$ estimate, 
and \cite{Lady_EllipticBook1968} (Section~16 of Chapter~3)
and \cite{HuangZou_JDE2002,HuangZou_DiscCont2007} for the 
piecewise $H^2$-regularity.  By \reff{hatpsiinftyweak}, we have 
\begin{equation}
\label{DphiGi0} 
 \Delta \phi_{\Gamma, \infty} = 0 \qquad \mbox{in } \Omega_- \cup \Omega_+.  
\end{equation}
This implies the piecewise $C^\infty$-regularity in \reff{phiGireg}. 
Moreover, since $\phi_{\Gamma, \infty} \in H_0^1(\Omega)$, 
routine calculations by \reff{hatpsiinftyweak} and 
the Divergence Theorem imply that \cite{Li_SIMA09}
\begin{equation}
\label{phiGiJumps}
 \llbracket \phi_{\Gamma, \infty} \rrbracket_{\Gamma} = 0  \qquad \mbox{and} \qquad  
\llbracket \ve_\Gamma \partial_n \phi_{\Gamma, \infty} \rrbracket_{\Gamma} = 0. 
\end{equation}
Throughout, for any function $u$ on $\Omega$ that has trace on $\Gamma$, we denote 
\begin{equation}
\label{JumpSign}
u^+ = u|_{\Omega_+}, \qquad u^- = u|_{\Omega_-},  \qquad \mbox{and} \qquad  
\llbracket u \rrbracket_{\Gamma} = u^+ - u^- \qquad  \mbox{on } \Gamma. 
\end{equation}

Let $\hat{\phi}_{\Gamma, \infty} \in \hat{\phi}_{\rm C} + H^1(\Omega)$
be the unique function such that 
$ \hat{\phi}_{\Gamma, \infty} = \phi_\infty $ on $\partial \Omega$ and 
\begin{equation}
\label{phiGammaWeak}
 \int_\Omega \ve_\Gamma \nabla \hat{\phi}_{\Gamma,\infty}  \cdot \nabla \eta \, dx  
=  \sum_{i=1}^N Q_i \eta(x_i)   \qquad    \forall \eta \in C^1_{\rm c} (\Omega);    
\end{equation}
cf.\ \cite{LSW63,Elschner_IFB2007}. 
If $\hat{\phi} = \hat{\phi}_{\rm C}$,  or $\hat{\phi}_0$, or $\hat{\phi}_\infty$, then
\reff{phiGammaWeak} is equivalent to 
\begin{align}
\label{equivweak}
\int_\Omega \ve_\Gamma \nabla ( \hat{\phi}_{\Gamma, \infty} - \hat{\phi})  \cdot \nabla \eta \, dx  
&= - (\ve_+ - \ve_-) \int_{\Omega_+} \nabla \hat{\phi} \cdot \nabla \eta \, dx
\nonumber \\
& = (\ve_+-\ve_-) \int_\Gamma \partial_n \hat{\phi} \,  \eta \, dS
\qquad  \forall \eta \in H^1_0 (\Omega),    
\end{align}
where the unit normal $n$ at $\Gamma$ points from $\Omega_-$ to $\Omega_+$.  
If $\eta \in H_0^1(\Omega)$ satisfies $\eta|_{\Omega_-} \in C^1(\Omega_-),$ 
then it follows from \reff{equivweak} that 
\begin{align*}
\int_\Omega \ve_\Gamma \nabla  \hat{\phi}_{\Gamma, \infty} \cdot \nabla \eta \, dx
&= \int_\Omega \ve_\Gamma \nabla \hat{\phi}  \cdot \nabla \eta \, dx  
 - (\ve_+ - \ve_-) \int_{\Omega_+} \nabla \hat{\phi} \cdot \nabla \eta \, dx
\\
&= \int_\Omega \ve_- \nabla \hat{\phi} \cdot \nabla \eta \, dx 
\\
&=  \sum_{i=1}^N Q_i \eta(x_i). 
\end{align*}
Therefore, we can replace $\eta \in C_{\rm c}^1(\Omega)$ in \reff{phiGammaWeak} 
by $\eta \in H_0^1(\Omega)$ that satisfies $\eta|_{\Omega_-} \in C^1(\Omega_-).$ 

By \reff{phiGammaWeak} and \reff{equivweak}, we have,  
after possibly modifying the value of $\hat{\phi}_{\Gamma, \infty}$ on a 
set of zero Lebesgue measure, that 
\begin{align}
\label{DhatphiGi}
&\Delta ( \hat{\phi}_{\Gamma, \infty}  - \hat{\phi} )  = 0 \quad \mbox{in } \Omega_-
\quad \mbox{and} \quad \Delta \hat{\phi}_{\Gamma, \infty} = 0 \quad \mbox{in }
(\Omega_-\setminus \{ x_1, \dots, x_N \} ) \cup  \Omega_+, 
\\
\label{hatphiJumps}
&\llbracket \hat{\phi}_{\Gamma, \infty} \rrbracket_\Gamma = 0 
\qquad \mbox{and} \qquad 
\llbracket \ve_\Gamma \partial_n \hat{\phi}_{\Gamma, \infty} \rrbracket_\Gamma = 0
\qquad \mbox{on } \Gamma. 
\end{align}
Moreover, it follows from the elliptic regularity theory 
\cite{
Lady_EllipticBook1968,
GilbargTrudinger98,
LiVogelius_ARMA2000,
LSW63,
HuangZou_DiscCont2007, 
HuangZou_JDE2002, 
Chipot_ARMA86, 
Elschner_IFB2007}
that 
\begin{equation}
\label{hatphi_reg}
 \hat{\phi}_{\Gamma, \infty} - \hat{\phi}  \in C(\overline{\Omega})\cap 
W^{1,\infty}(\Omega), \quad 
(\hat{\phi}_{\Gamma, \infty}-\hat{\phi} )^-\in C^\infty(\Omega_-)\cap H^2 (\Omega_-), 
\quad \hat{\phi}_{\Gamma,\infty}^+\in C^\infty(\Omega_+)\cap H^2(\Omega_+). 
\end{equation}
Further, then there exists a constant 
$C > 0$ that may depend on $\Omega$, $x_i$ and $Q_i$ $(1\le i \le N)$, 
$ \ve_+,$ $\ve_-,$ $\phi_\infty,$ and $\hat{\phi}$, but does not depend on  $\Gamma$, such that 
\begin{equation}
\label{hatphiGibounds}
 \|\hat{\phi}_{\Gamma, \infty}  - \hat{\phi} \|_{W^{1, \infty}(\Omega)} \le C. 
\end{equation}
These results \reff{hatphi_reg} and \reff{hatphiGibounds}
follow from the same arguments used above (cf.\ the description below
\reff{phiGammabounds}) applied to \reff{phiGammaWeak} 
with $\eta \in C_{\rm c}^1(\Omega)$ so chosen 
that the support of $\eta$ is in a neighborhood of $\Gamma$ that excluding the
sigularities $x_i$ $(i = 1, \dots, N).$ 
For any $g\in H^{-1}(\Omega)$, let $L_\Gamma g \in H^1_0(\Omega)$  be the unique weak solution 
(defined using test functions in $H^1_0(\Omega)$) to the boundary-value problem
\begin{equation}
\label{Lf}
\nabla \cdot \ve_\Gamma \nabla L_\Gamma g  = - g \quad \mbox{in } \Omega 
\quad \mbox{and} \quad  L_\Gamma g = 0  \quad \mbox{on } \partial \Omega. 
\end{equation}
This defines a linear, continuous, and self-adjoint 
operator $L_\Gamma: H^{-1}(\Omega) \to H^1_0(\Omega).$  The map 
\begin{equation}
\label{NormLGamma}
g \mapsto \| g \|_{L_\Gamma} := \sqrt{ \langle g, L_\Gamma g \rangle_{H^{-1}(\Omega), H^1_0
(\Omega)} } = \left[ \int_\Omega \ve_\Gamma | \nabla ( L_\Gamma g ) |^2 dx \right]^{1/2} 
\end{equation}
defines a norm on $H^{-1}(\Omega)$ which is equivalent to the $H^{-1}(\Omega)$-norm.  
If $g\in L^1(\Omega)$, then we define $g\in L^1(\Omega) \cap H^{-1}(\Omega)$ if 
\begin{equation}
\label{sup}
\sup \left\{ \int_\Omega g u\, dx: u \in H^1_0(\Omega)\cap L^\infty(\Omega)
\mbox{ and }  \| u \|_{H^1(\Omega)}=1 \right\} < \infty.
\end{equation}
In this case, the action of $g$ on $H^1_0(\Omega)$ is defined first for any 
$u \in H^1_0(\Omega)\cap L^\infty(\Omega)$ by the integral of $g u $ over $\Omega$ and 
 then extended for any $u \in H^1_0(\Omega)$ by \reff{sup} and 
the fact that $H^1_0(\Omega)\cap L^\infty(\Omega)$ is dense in $H^1_0(\Omega)$. 


\medskip

\subsection{The Poisson--Boltzmann Equation}
\label{ss:PBE}

We now study the well-posedness of the boundary-value problem of the 
Poisson--Boltzmann (PB) equation \reff{PBE} with a dielectric boundary and point charges. 

\begin{definition}
\label{d:PBsolution}
A function $\psi \in \hat{\phi}_{\rm C} + H^1(\Omega)$ 
is a weak solution to the boundary-value problem of the dielectric-boundary
PB equation \reff{PBE}, if $\psi = \phi_\infty$ on $\partial \Omega$, 
$\chi_+  B'(\psi-\phi_{\Gamma, \infty}/2 ) \in L^1(\Omega)\cap H^{-1}(\Omega),$ and 
\begin{equation}
\label{weakPBE}
\int_\Omega \left[ \ve_\Gamma \nabla \psi \cdot \nabla \eta + 
\chi_+ B'\left( \psi- \frac{\phi_{\Gamma, \infty}}{2} \right) 
\eta \right] dx = \sum_{i=1}^N Q_i \eta (x_i) \qquad \forall \eta \in C^1_{\rm c}(\Omega).     
\end{equation}
\end{definition}

Note that we can replace $\eta\in C_{\rm c}^1(\Omega)$ in \reff{weakPBE}
by $\eta \in H^1_0(\Omega)$ that satisfies $\eta^- \in C^1(\Omega_-)$; cf.\ the remark
below \reff{equivweak}.  
The theorem below provides the existence and uniqueness of the solution to the boundary-value
problem of the dielectric-boundary PB equation, and an equivalent formulation 
of such a boundary-value problem.  These results are essentially proved in 
\cite{LiChengZhang_SIAP11}. 
Here we sketch the proof and add some points that are not included in the previous
proof due to some minor differences between the current and previous statements. 
Note that $\hat{\phi}_{\rm C} + H^1(\Omega) = \hat{\phi}_{\Gamma, \infty} + H^1(\Omega)$.
So, we can replace $\hat{\phi}_{\rm C} $ by $\hat{\phi}_{\Gamma, \infty}$ in the above definition. 
Note also that there is a variational principle for the PB equation; cf.\ Theorem~\ref{t:PBvariation}. 



\begin{theorem}
\label{t:DBPBE}
\begin{compactenum}
\item[\rm (1)] 
There exists a unique weak solution $\psi_\Gamma \in \hat{\phi}_{\Gamma, \infty} +H^1_0(\Omega)$ 
of the boundary-value problem of the dielectric-boundary PB equation \reff{PBE}. Moreover, 
after a possible modification of $\psi_\Gamma $ on a set of zero Lebesgue measure, 
$\psi_\Gamma - \hat{\phi}_{\Gamma, \infty} \in C(\overline{\Omega})\cap W^{1,\infty}(\Omega), $ 
$( \psi_\Gamma -\hat{\phi}_{\Gamma,\infty})^- \in C^\infty (\Omega_-)\cap H^2(\Omega_-),$
and $\psi_\Gamma^+ \in C^\infty (\Omega_+)\cap H^2(\Omega_+).$ Further, 
there exists a constant $C > 0$ that may depend on $\Omega$, $x_i$ and $Q_i$ $(1\le i \le N)$, 
$ \ve_+,$ $\ve_-,$ $\phi_\infty,$ and $B$, but does not depend on  $\Gamma$, such that 
\begin{align}
\label{PBsolnbounds}
& \| \psi_\Gamma - \hat{\phi}_{\Gamma, \infty} \|_{W^{1,\infty}(\Omega)} 
\le C. 
\end{align}
\item[\rm (2)]
A function $\psi \in \hat{\phi}_{\Gamma, \infty } + H^1(\Omega)$ with 
$\chi_+  B'(\psi-\phi_{\Gamma, \infty}/2 ) \in L^1(\Omega)\cap H^{-1}(\Omega)$ is the weak solution
to the boundary-value problem of the dielectric-boundary PB equation \reff{PBE} if
and only if it is the unique solution to the following elliptic interface problem: 
\begin{equation}
\label{interfaceform}
\left\{
\begin{aligned}
&\Delta (\psi - \hat{\phi}_{\Gamma, \infty}) = 0 &  \quad & \mbox{in } \Omega_-, \\
& \ve_+ \Delta \psi - B'\left( \psi -\frac{\phi_{\Gamma, \infty}}{2} \right) =   0 
& \quad & \mbox{in } \Omega_+, \\
&\llbracket  \psi \rrbracket_\Gamma  = 0\qquad \mbox{and} \qquad
 \llbracket  \varepsilon_\Gamma  \partial_n \psi
\rrbracket_\Gamma   = 0 & \quad  & \mbox{on } \Gamma, \\
& \psi = \psi_\infty  & \quad & \mbox{on } \partial \Omega. 
\end{aligned}
\right. 
\end{equation}
\end{compactenum}
\end{theorem}

\begin{proof}
(1) With $u = \psi - \hat{\phi}_{\Gamma,\infty}$ and by  \reff{phiGammaWeak} and \reff{weakPBE}, 
it is equivalent to show that there exists a unique 
$u_{\Gamma }\in H_0^1(\Omega)$ such that 
$\chi_+ B'(u_{\Gamma} + \hat{\phi}_{\Gamma, \infty} - \phi_{\Gamma, \infty}/2 ) 
\in L^1(\Omega) \cap H^{-1}(\Omega)$, and 
\begin{equation}
\label{weakuGamma} 
\int_\Omega \left[ \ve_\Gamma \nabla u_{\Gamma} \cdot \nabla \eta 
+ \chi_+ B'\left( u_{\Gamma} + \hat{\phi}_{\Gamma, \infty}  
- \frac{\phi_{\Gamma, \infty}}{2} \right) \eta \right] dx 
= 0 \qquad \forall \eta \in H^1_0 (\Omega).     
\end{equation}
Define 
\[
I[u] = \int_\Omega \left[ \frac{\ve_\Gamma}{2} | \nabla u |^2 + \chi_+ 
B\left(u  + \hat{\phi}_{\Gamma, \infty} -\frac{\phi_{\Gamma, \infty}}{2} \right) \right] dx 
\qquad \forall u\in H_0^1(\Omega).  
\]
Since $B \ge 0$ and $B$ is convex, we can use the direct method in the calculus of variations to 
obtain a unique minimizer $u_{\Gamma} \in H_0^1(\Omega) $ 
of the functional $I: H_0^1(\Omega) \to [0, \infty].$
Moreover, comparing the values $I[u_{\Gamma}]$ and $I[u_{{\Gamma}, \lambda}]$ for any constant 
$\lambda > 0$ large enough,  where $u_{{\Gamma}, \lambda} = u_{{\Gamma}}$ if 
$|u_{\Gamma}| \le \lambda$ and $u_{{\Gamma}, \lambda} = \lambda\, \mbox{sign}\,(u_{\Gamma})$
otherwise, we have by the convexity of $B$ that $u_{\Gamma} = u_{{\Gamma}, \lambda}$
a.e.\ $\Omega$ for some $\lambda$ independent on $\Gamma.$  
Hence, $u_{\Gamma} \in L^\infty(\Omega),$  
and $\| u_\Gamma \|_{L^\infty(\Omega)} \le C$ for some constant $C > 0$ independent of $\Gamma;$ 
cf.\ \cite{LiChengZhang_SIAP11}. 
This allows the use of the Lebesgue Dominated Convergence Theorem in the routine calculations
of $(d/dt)|_{t=0} I [u_\Gamma + t \eta ]= 0$ for any $\eta \in C_{\rm c}^1(\Omega)$
to obtain the equation in \reff{weakuGamma}. Since $C_{\rm c}^1(\Omega)$ is dense in $H_0^1(\Omega)$, 
\reff{weakuGamma} holds true.  
The convexity of $B$ now imiplies that $u_{\Gamma}$ is the unique solution as desired. 

The regularity of the soluton $\psi_\Gamma$ follows from the elliptic regularity theory
\cite{
Lady_EllipticBook1968,
Chipot_ARMA86, 
GilbargTrudinger98,
LiVogelius_ARMA2000, 
HuangZou_JDE2002, 
HuangZou_DiscCont2007, 
Elschner_IFB2007}, 
with the same argument above for the regularity of the function $\hat{\phi}_{\Gamma, \infty};$
cf.\  \reff{hatphi_reg} and \reff{hatphiGibounds}. 
Note that the piecewise $C^\infty$ smoothness follows from a usual bootstrapping method. 


(2) This part of the proof is the same as that given in \cite{Li_SIMA09}. 
\end{proof}

\subsection{Electrostatic Free-Energy Functional of Ionic Concentrations}
\label{ss:FreeEnergyFunctional}


We define
\begin{align*}
&{\cal X}  = \biggl\{ (c_1, \dots, c_M)\in L^{1}(\Omega, \R^M): 
c_j = 0 \mbox{ a.e. } \Omega_- \mbox{ for } j = 1, \dots, M 
\mbox{ and } \sum_{j=1}^M q_j c_j \in H^{-1}(\Omega) \biggr\},  
\\
&{\cal X}_+ = \biggl\{ (c_1, \dots, c_M)\in {\cal X}: c_j \ge 0 \mbox{ a.e. } \Omega_+  
\mbox{ for } j = 1, \dots, M \biggr\}. 
\end{align*}
Here, for any $g \in L^1(\Omega)$, we define $g \in L^1(\Omega) \cap H^{-1}(\Omega)$ by \reff{sup}. 
The space ${\cal X}$ is a Banach space equipped with the norm 
\[
\| c \|_{\cal X} = \sum_{j=1}^M \| c_j \|_{L^1(\Omega)} 
+ \left\| \sum_{j=1}^M q_j c_j \right\|_{H^{-1}(\Omega)}
\qquad \forall c = (c_1, \dots, c_M) \in {\cal X}. 
\]
Moreover, ${\cal X}_+$ is a convex and closed subset of ${\cal X}.$ 
For any $c = (c_1, \dots, c_M) \in {\cal X},$  standard arguments 
(cf.\ \cite{Elschner_IFB2007, EvansBook2010,GilbargTrudinger98,LSW63}) 
imply that there exists a unique weak solution $\psi$ 
to the boundary-value problem \reff{PoissonNew}, defined by $\psi \in \hat{\phi}_{\rm C} + H^1(\Omega),$
$\psi = \phi_\infty$ on $\partial \Omega$, and 
\begin{equation}
\label{weakpsi}
\int_\Omega \ve_\Gamma \nabla \psi \cdot \nabla \eta \, dx = \sum_{i=1}^N Q_i \eta (x_i) 
+ \int_{\Omega_+} \left( \sum_{j=1}^M q_j c_j \right) \eta \, dx
\qquad \forall \eta \in C^1_{\rm c}(\Omega),    
\end{equation}
Equivalently, if $\hat{\phi} \in \hat{\phi}_{\rm C} + H^1(\Omega)$ satisfies \reff{Wpsi0}, then
\begin{align*}
\int_\Omega \ve_\Gamma \nabla (\psi - \hat{\phi} ) \cdot \nabla \eta \, dx 
& = \int_{\Omega_+} \left[ (\ve_- - \ve_+) \nabla \hat{\phi} \cdot \nabla \eta 
+  \left( \sum_{j=1}^M q_j c_j \right) \eta \right] dx
\qquad \forall \eta \in H^1_0(\Omega).   
\end{align*}
Clearly, $\psi-\hat{\phi}$ is harmonic in $\Omega_-.$ 
Moreover, it follows from the definition of $\hat{\phi}_{\Gamma, \infty}$ (cf.~\reff{phiGammaWeak}) 
and $L_\Gamma$ (cf.~\reff{Lf}) that 
\begin{equation}
\label{psippLGamma}
\psi = \hat{\phi}_{\Gamma,\infty} + L_\Gamma \left( \sum_{j=1}^M q_j c_j \right). 
\end{equation}

Since the function $s\mapsto s \log s $ $(s \ge 0)$ is bounded below and $\Omega$ is bounded, 
$F_\Gamma[c] > -\infty$ for any $c \in {\cal X}_+,$ where $F_\Gamma[c]$ is defined in \reff{FGammac}.

\begin{theorem}
\label{t:minFa}
Let $\psi_\Gamma$ be the unique weak solution to the dielectric-boundary PB equation \reff{PBE}. 
For each $j \in \{ 1, \dots, M\}$, define $c_{\Gamma, j}: \Omega \to [0, \infty)$ by 
\begin{equation}
\label{Cmin}
c_{\Gamma, j}(x)= 
\left\{
\begin{aligned}
& 0 \quad & \mbox{if } x \in \overline{\Omega_-}, 
\\
& c_j^{\infty} e^{-\beta q_j \left[ \psi_\Gamma (x) 
-\phi_{\Gamma,\infty}(x)/2 \right]} \quad & \mbox{if } x \in \Omega_+.    
\end{aligned}
\right.
\end{equation}
Then $c_\Gamma := (c_{\Gamma,1}, \dots, c_{\Gamma, M})\in {\cal X}_+$ and 
$\psi_\Gamma$ is the electrostatic potential corresponding 
to $c_{\Gamma}$, i.e., the unique weak solution
to \reff{PoissonNew} with $c_j$ replaced by $c_{\Gamma, j}$ $(j=1, \dots, M)$. 
Moreover,  $c_{\Gamma}$ is the 
unique minimizer of the functional $F_\Gamma: {\cal X}_+ \to (-\infty, \infty]$ defined
in \reff{FGammac}, and 
\begin{align}
\label{Fpsi}
F_\Gamma [c_\Gamma]  & = \frac12  \sum_{i=1}^N Q_i (\psi_\Gamma - \hat{\phi}_{\rm C} ) (x_i) 
\nonumber \\
& \qquad 
+  \int_{\Omega_+} \left[  \frac12\left( \psi_\Gamma - \phi_{\Gamma,\infty}\right) 
B'\left( \psi_\Gamma -\frac{\phi_{\Gamma, \infty}}{2}  \right) 
- B\left( \psi_\Gamma -\frac{\phi_{\Gamma, \infty} }{2} \right) \right] dx. 
\end{align}
\end{theorem}

\begin{proof}
By the properties of $\psi_\Gamma$ (cf.\ Theorem~\ref{t:DBPBE}) and $\phi_{\Gamma, \infty}$
(cf.\ \reff{phiGammabounds}), we have $c_\Gamma \in \calX_+$. 
If we replace $c_j$ in \reff{PoissonNew} by $c_{\Gamma, j}$ defined in \reff{Cmin}
 and note the definition of $B$ in \reff{B}, we get exactly the PB equation \reff{PBE}. 
Therefore, the unique solution $\psi_\Gamma$ to the boundary-value problem of the PB equation
\reff{PBE} is also the unique solution to the boundary-value problem of Poisson's equation
\reff{PoissonNew} corresponding to $c_{\Gamma}$. 

We now prove that $c_\Gamma$ is the unique minimizer of $F_\Gamma: \calX_+\to (-\infty,  \infty]$. 
To do so, we first re-write the functional $F_\Gamma$. 
Let $c = (c_1, \dots, c_M) \in {\cal X}_+ $ and let $\psi \in \hat{\phi}_{\rm C} +H^1(\Omega)$ 
be the corresponding electrostatic potential, i.e., 
the weak solution to \reff{PoissonNew} defined in \reff{weakpsi}. 
Denote $f = \sum_{j=1}^M q_j c_j.$ Since $f = 0$ a.e.\ in $\Omega_-$, we have 
by the definition of $L_\Gamma$ (cf.\ \reff{Lf}) that $L_\Gamma f$ is harmonic in $\Omega_-$.  Moreover, 
\[
\sum_{i=1}^N Q_i \left( L_\Gamma f \right) (x_i) = \int_{\Omega_+} 
( \hat{\phi}_{\Gamma, \infty} - {\phi}_{\Gamma, \infty} )  f \, dx;   
\]
cf.\ Lemma~3.2 in \cite{Li_SIMA09} (where $L/(4\pi) $ and $G/(4 \pi) $ are  our 
$L_\Gamma$ and $ \hat{\phi}_{\Gamma, \infty} - {\phi}_{\Gamma, \infty}$ here, respectively). 
This, together with \reff{psippLGamma} and the fact that all $\psi - \hat{\phi}_{\rm C},$ 
$\hat{\phi}_{\Gamma, \infty}-\hat{\phi}$, and $L_\Gamma f $ 
are harmonic in $\Omega_-$, implies that 
\begin{align*}
\sum_{i=1}^N Q_i (\psi-\hat{\phi}_{\rm C})(x_i) & = \sum_{i=1}^N Q_i (L_\Gamma f )(x_i) +  
\sum_{i=1}^N Q_i ( \hat{\phi}_{\Gamma, \infty} - \hat{\phi}_{\rm C} ) (x_i) 
\\
& = \int_{\Omega_+} ( \hat{\phi}_{\Gamma, \infty} - {\phi}_{\Gamma, \infty} ) 
\left( \sum_{j=1}^M q_j c_j \right)  dx +  
\sum_{i=1}^N  Q_i  (\hat{\phi}_{\Gamma, \infty} - \hat{\phi}_{\rm C} )  (x_i).
\end{align*}
With this and \reff{psippLGamma}, we can rewrite $F_\Gamma [c]$ \reff{FGammac} as 
\begin{equation}
\label{ReformFGammac}
F_\Gamma [c] = \int_{\Omega_+ } \left[ \frac12 \left(\sum_{j=1}^M q_j c_j \right)
L_\Gamma \left(\sum_{j=1}^M q_j c_j \right)+ 
\sum_{j=1}^M \left( \beta^{-1} c_j \log c_j  + \alpha_j c_j \right) \right] dx + E_{0, \Gamma},
\end{equation}
where all $\alpha_j = \alpha_j(x) $ $(j =1 , \dots, M)$ and $E_{0, \Gamma}$ 
are independent of $c$, given by
\begin{align}
\label{alphaj}
&\alpha_j(x) =  q_j  \left[ \hat{\phi}_{\Gamma, \infty} (x) - \frac12 \phi_{\Gamma, \infty}(x)\right]
+  \beta^{-1} \left( 3 \log \Lambda - 1 \right) 
- \mu_j  \quad \forall x \in \Omega, \ j = 1, \dots, M, \\
&E_{0, \Gamma} = \frac12 \sum_{i=1}^N Q_i ( \hat{\phi}_{\Gamma, \infty} 
- \hat{\phi}_{\rm C} ) (x_i)    + \beta^{-1} | \Omega_+ | \sum_{j=1}^M c_j^\infty. 
\nonumber 
\end{align}
Here and below, we denote by $|A|$ the Lebesgue measure of $A$ when no confusion arises.


We now compare $F_\Gamma[c]$ and $F_\Gamma[c_\Gamma].$ 
By Taylor's expansion, we have for any $s, t \in (0, \infty)$ that 
\[
s \log s - t \log t =  ( 1 + \log t ) (s - t) + \frac{1}{2 r } (s - t)^2 \ge  
 ( 1 + \log t ) (s - t),
\]
where $r$ is in between $s$ and $t$. 
Consequently, by \reff{ReformFGammac} and the fact that $L_\Gamma$ is self-adjoint, we have 
\begin{align*}
F_\Gamma [c] - F_\Gamma[c_\Gamma] & =  
\int_\Omega \frac12 \left( \sum_{j=1}^M q_j (c_j - c_{\Gamma, j}) \right)  
L_\Gamma  \left( \sum_{j=1}^M q_j (c_j - c_{\Gamma, j}) \right)  dx
\\
&\qquad +   \int_\Omega  \left( \sum_{j=1}^M q_j (c_j - c_{\Gamma, j}) \right)  
L_\Gamma  \left( \sum_{k=1}^M q_k  c_{\Gamma, k} \right)  dx
\\
& \qquad +  \beta^{-1} \sum_{j=1}^M 
\int_\Omega \left( c_j \log c_j - c_{\Gamma, j} \log c_{\Gamma, j} \right) \, dx
+ \sum_{j=1}^M \int_\Omega  (c_j - c_{\Gamma, j}) \alpha_j \, dx
\\
&\ge \sum_{j=1}^M \int_\Omega (c_j - c_{\Gamma, j} )
\left[ q_j L_\Gamma  \left( \sum_{k=1}^M q_k  c_{\Gamma, k} \right)
+ \beta^{-1} \left( 1 + \log c_{\Gamma, j} \right) + \alpha_j \right] dx. 
\end{align*}
It follows from the fact that $c_j^\infty = \Lambda^{-3 } e^{\beta \mu_j}$ 
(cf.\ the assumption (A2)), \reff{psippLGamma}, \reff{Cmin}, 
and \reff{alphaj} that the quantity inside the brackets in the 
above integral vanishes. Thus, $F[c]\ge F[c_\Gamma]$. Hence, $c_\Gamma$ is a minimizer
of $F_\Gamma: \calX_+ \to (-\infty, \infty]$. Since $F_\Gamma$ is convex, and in particular, 
$s\mapsto s \log s$ is strictly convex on $(0, \infty)$, the minimizer of $F_\Gamma$ is unique;  
cf.\ \cite{Li_SIMA09}.  

Finally, we obtain \reff{Fpsi} from \reff{FGammac} with $\psi_\Gamma$ and $c_\Gamma$ replacing
$\psi$ and $c$, respectively, \reff{B}, and \reff{Cmin}. 
\end{proof}

\section{Dielectric Boundary Force}
\label{s:DBF}


\subsection{Electrostatic Free Energy of a Dielectric Boundary}
\label{ss:FGamma}

Given any dielectric boundary $\Gamma$, we denote by 
\begin{equation}
\label{DefineJGamma}
 E[\Gamma]  = \min_{c\in {\cal X}_+} F_\Gamma[c], 
\end{equation}
the minimum electrostatic free energy given in Theorem~\ref{t:minFa} (cf.\ \reff{Fpsi}). 
We reformulate $E[\Gamma]$ to convert the discrete part of the 
energy into volume integrals that will be useful when we calculate the variation of 
$E[\Gamma]$ with respect to the boundary variation of $\Gamma.$ 


\begin{lemma}
\label{l:EGammaReform}
Let $\Gamma$ be a dielectric boundary satisfying 
the part of the assumption A1 on 
$\Gamma$ in Subsection~\ref{ss:Assumptions}. Let $\psi_\Gamma \in \hat{\phi}_{\rm C}
+ H^1(\Omega)$ be the corresponding solution to the boundary-value 
problem of PB equation \reff{PBE}.  We have 
\begin{align}
\label{newEGamma}
E[\Gamma] & =
- \int_\Omega \frac{\ve_\Gamma}{2} | \nabla (\psi_\Gamma - \hat{\phi}_{\Gamma, \infty}) |^2 dx
- \int_{\Omega_+} B\left( \psi_\Gamma -\frac{\phi_{\Gamma, \infty} }{2} \right) dx
\nonumber
\\
& \quad \quad
+ \frac{\ve_- - \ve_+}{2} \int_{\Omega_+}
 \nabla  \hat{\phi}_{\Gamma, \infty} \cdot \nabla  \hat{\phi}_0 \, dx
+ \frac12 \sum_{i=1}^N Q_i (\hat{\phi}_\infty - \hat{\phi}_{\rm C} ) (x_i),
\end{align}
where all the functions $\hat{\phi}_{\rm C}$, $\hat{\phi}_0$, $\hat{\phi}_\infty$,
$\phi_{\Gamma, \infty}$, and $\hat{\phi}_{\Gamma, \infty}$
are defined in Subsection~\ref{ss:Assumptions}. 
\end{lemma}

\begin{proof}
We first prove an elementary identity.  Let $u \in C^2(\Omega_-) \cap
C^1(\overline {\Omega_-})$ be such that $\Delta u = 0$ in
$\Omega_-.$ Let $v \in \hat{\phi}_{\rm C} + H^1(\Omega_-)
\cap C(\overline{\Omega_-})$,
in particular, $v = \hat{\phi}_{\rm C}$, $\hat{\phi}_0$, $\hat{\phi}_\infty$,
or $\hat{\phi}_{\Gamma, \infty}$ (restricted onto $\Omega_-$).
Denote $B_\alpha = \cup_{i=1}^N B(x_i, \alpha)$ for $0 < \alpha \ll 1$
and $\nu$ the  unit normal at $\partial B(\alpha) = \cup_{i=1}^N \partial
B(x_i, \alpha) $, pointing toward $x_i$ $(i = 1, \dots, N)$.
Since the unit normal $n$ at $\Gamma$ points from $\Omega_-$ to $\Omega_+$, 
and since $v = \hat{\phi}_{\rm C} + \hat{v}$ for some $\hat{v} \in 
H^1(\Omega_-)\cap C(\overline{\Omega_-})$ 
and $\hat{\phi}_{\rm C} $ is given in \reff{psi0}, we have
\begin{align}
\label{fact}
\int_{\Omega_-} \nabla u  \cdot \nabla v \, dx
& = \lim_{\alpha \to 0^+} \int_{\Omega_-\setminus B_\alpha} \nabla u \cdot\nabla v \, dx
\nonumber \\
& = \int_\Gamma \partial_n u \, v \, dS
+ \lim_{\alpha \to 0^+}  \sum_{i=1}^N \int_{\partial B(x_i,\alpha)} \partial_\nu u \, v \, dS
\nonumber \\
& =  \int_\Gamma \partial_n  u \, v \, dS.
\end{align}
Denoting now $ W = (1/2) \sum_{i=1}^N Q_i (\hat{\phi}_\infty - \hat{\phi}_{\rm C} ) (x_i), $
we have by \reff{DefineJGamma} and \reff{Fpsi} that 
\begin{align}
\label{W+}
E[\Gamma] & =  
\frac12 \sum_{i=1}^N Q_i ( \hat{\phi}_{\Gamma, \infty} - \hat{\phi}_\infty ) (x_i) 
+ \frac12 \sum_{i=1}^N Q_i (\psi_\Gamma - \hat{\phi}_{\Gamma, \infty} ) (x_i) 
\nonumber \\
& \qquad 
+  \int_{\Omega_+} \left[  \frac12 ( \psi_\Gamma - \phi_{\Gamma,\infty})
B'\left( \psi_\Gamma -\frac{\phi_{\Gamma, \infty}}{2}  \right) 
- B\left( \psi_\Gamma -\frac{\phi_{\Gamma, \infty} }{2} \right) \right] dx + W. 
\end{align}
We first consider the first term in \reff{W+}. Note that the unit vector $n$ normal
to $\Gamma$ points from $\Omega_-$ to $\Omega_+$. We have by Green's formula that 
\begin{align*}
& \frac12 \sum_{i=1}^N Q_i ( \hat{\phi}_{\Gamma, \infty} - \hat{\phi}_\infty  ) (x_i) 
\\
& \quad 
=  \int_{\Omega} \frac{\ve_-}{2} \nabla \hat{\phi}_0 \cdot
\nabla  ( \hat{\phi}_{\Gamma, \infty} - \hat{\phi}_\infty )\, dx
\qquad [\mbox{by \reff{Wpsi0} with } \hat{\phi} = \hat{\phi}_{0} \mbox{ and }  
\eta =  \hat{\phi}_{\Gamma, \infty} - \hat{\phi}_\infty]
\\
& \quad 
=  \int_{\Omega_-} \frac{\ve_-}{2} \nabla \hat{\phi}_0 \cdot
\nabla ( \hat{\phi}_{\Gamma, \infty} - \hat{\phi}_\infty )\, dx
+ \int_{\Omega_+} \frac{\ve_-}{2} \nabla \hat{\phi}_0 \cdot \nabla 
( \hat{\phi}_{\Gamma, \infty} - \hat{\phi}_\infty )\, dx
\\
& \quad 
=  \int_{\Gamma} \frac{\ve_-}{2} \hat{\phi}_0
 \partial_n ( \hat{\phi}^-_{\Gamma, \infty} - \hat{\phi}_\infty )\, dS 
+ \int_{\Omega_+} \frac{\ve_-}{2} \nabla \hat{\phi}_0 \cdot
 \nabla ( \hat{\phi}_{\Gamma, \infty} - \hat{\phi}_\infty )\, dx
\qquad [\mbox{by \reff{fact}}] 
\\
& \quad 
=  \int_{\Gamma} \frac{\ve_+}{2} \hat{\phi}_0 \partial_n  \hat{\phi}^+_{\Gamma, \infty}\, dS 
-  \int_{\Gamma} \frac{\ve_-}{2} \hat{\phi}_0 \partial_n  \hat{\phi}_{\infty} \, dS 
\qquad [\mbox{by \reff{hatphiJumps}}] 
\\
&\quad \quad 
+ \int_{\Omega_+} \frac{\ve_-}{2} \nabla \hat{\phi}_0 \cdot \nabla 
( \hat{\phi}_{\Gamma, \infty} - \hat{\phi}_\infty )\, dx
\\
& \quad 
= -\int_{\partial \Omega_+} \frac{\ve_+}{2} \hat{\phi}_0 \partial_n \hat{\phi}_{\Gamma, \infty} \, dS 
+ \int_{\partial \Omega_+} \frac{\ve_-}{2} \hat{\phi}_0 \partial_n \hat{\phi}_\infty \, dS
\qquad [\mbox{since $ \hat{\phi}_0 = 0 $ on $\partial \Omega$}]
\\
&\quad \quad 
+ \int_{\Omega_+} \frac{\ve_-}{2} \nabla \hat{\phi}_0 \cdot
 \nabla ( \hat{\phi}_{\Gamma, \infty} - \hat{\phi}_\infty )\, dx
\\
& \quad 
= - \int_{\Omega_+} \frac{\ve_+}{2} \nabla \hat{\phi}_0 \cdot \nabla \hat{\phi}_{\Gamma, \infty} \, dx
+ \int_{\Omega_+} \frac{\ve_-}{2} \nabla \hat{\phi}_0 \cdot \nabla \hat{\phi}_\infty \, dx
\\ 
& \quad \quad 
+ \int_{\Omega_+} \frac{\ve_-}{2} \nabla \hat{\phi}_0 \cdot \nabla 
( \hat{\phi}_{\Gamma, \infty} - \hat{\phi}_\infty )\, dx
\\
& \quad 
= \frac{\ve_- - \ve_+}{2} \int_{\Omega_+} \nabla  
\hat{\phi}_{\Gamma, \infty} \cdot \nabla  \hat{\phi}_0 \, dx. 
\end{align*}
Considering now the second and third terms in \reff{W+}, we have 
\begin{align*}
&\frac12 \sum_{i=1}^N Q_i  (\psi_\Gamma - \hat{\phi}_{\Gamma, \infty} ) (x_i) 
\\
& \quad 
+  \int_{\Omega_+} \left[  \frac12\left( \psi_\Gamma - \phi_{\Gamma,\infty}\right) 
B'\left( \psi_\Gamma -\frac{\phi_{\Gamma, \infty}}{2}  \right) 
- B\left( \psi_\Gamma -\frac{\phi_{\Gamma, \infty} }{2} \right) \right] dx 
\\
& \quad = 
\sum_{i=1}^N Q_i (\psi_\Gamma - \hat{\phi}_{\Gamma, \infty} ) (x_i) 
 - \frac12 \sum_{i=1}^N Q_i (\psi_\Gamma - \hat{\phi}_{\Gamma, \infty} ) (x_i) 
\\
& \quad \quad
+  \int_{\Omega_+} \left[  \frac12 ( \psi_\Gamma - \phi_{\Gamma,\infty})
B'\left( \psi_\Gamma -\frac{\phi_{\Gamma, \infty}}{2}  \right) 
- B\left( \psi_\Gamma -\frac{\phi_{\Gamma, \infty} }{2} \right) \right] dx 
\\
& \quad =  \int_{\Omega} \ve_\Gamma  \nabla \hat{\phi}_{\Gamma, \infty} 
\cdot  \nabla ( \psi_\Gamma-\hat{\phi}_{\Gamma, \infty} )\, dx
\quad [\mbox{by \reff{phiGammaWeak}}]
\\
& \quad \quad  
- \frac12 \int_\Omega \left[ \ve_\Gamma \nabla \psi_\Gamma \cdot 
 \nabla ( \psi_\Gamma - \hat{\phi}_{\Gamma, \infty} )
+\chi_+ B'\left( \psi_\Gamma - \frac{\phi_{\Gamma, \infty}}{2} \right)
( \psi_\Gamma - \hat{\phi}_{\Gamma, \infty} ) \right] 
\quad [\mbox{by \reff{weakPBE}}]
\\
& \quad \quad 
 +  \int_{\Omega_+} \left[  \frac12\left( \psi_\Gamma - \phi_{\Gamma,\infty}\right) 
B'\left( \psi_\Gamma -\frac{\phi_{\Gamma, \infty}}{2}  \right) 
- B\left( \psi_\Gamma -\frac{\phi_{\Gamma, \infty} }{2} \right) \right] dx 
\\
& \quad = - \int_\Omega \frac{\ve_\Gamma}{2} | \nabla (\psi_\Gamma 
- \hat{\phi}_{\Gamma, \infty} ) |^2 dx 
- \int_{\Omega_+} B\left( \psi_\Gamma -\frac{\phi_{\Gamma, \infty} }{2} \right) dx 
\\
& \quad \quad 
+ \int_\Omega \frac{\ve_\Gamma}{2}  
\nabla (\psi_\Gamma - \hat{\phi}_{\Gamma, \infty} ) 
\cdot \nabla \hat{\phi}_{\Gamma, \infty}\, dx 
 + \int_{\Omega_+}  \frac12 ( \hat{\phi}_{\Gamma, \infty} - \phi_{\Gamma,\infty})
B'\left( \psi_\Gamma -\frac{\phi_{\Gamma, \infty}}{2}  \right) dx 
\\
& \quad = - \int_\Omega \frac{\ve_\Gamma}{2} | \nabla (\psi_\Gamma 
- \hat{\phi}_{\Gamma, \infty} ) |^2 dx 
- \int_{\Omega_+} B\left( \psi_\Gamma -\frac{\phi_{\Gamma, \infty} }{2} \right) dx 
\\
&\quad \quad 
+ \int_\Omega \frac{\ve_\Gamma}{2}  \nabla (\psi_\Gamma - \hat{\phi}_{\Gamma, \infty} ) 
\cdot \nabla \hat{\phi}_{\Gamma, \infty}\, dx 
\\
& \quad \quad 
 + \int_{\Omega_+}  \frac{\ve_+}{2} ( \hat{\phi}_{\Gamma, \infty} - \phi_{\Gamma,\infty} ) 
\Delta (  \psi_\Gamma - \hat{\phi}_{\Gamma, \infty} )\, dx
\quad [\mbox{by \reff{interfaceform} and \reff{DhatphiGi} }]
\\
& \quad = - \int_\Omega \frac{\ve_\Gamma}{2} | \nabla  (\psi_\Gamma 
- \hat{\phi}_{\Gamma, \infty} ) |^2 dx 
- \int_{\Omega_+} B\left( \psi_\Gamma -\frac{\phi_{\Gamma, \infty} }{2} \right) dx 
\\
& \quad \quad 
+ \int_\Omega \frac{\ve_\Gamma}{2}  \nabla (\psi_\Gamma - \hat{\phi}_{\Gamma, \infty}) \cdot 
\nabla ( \hat{\phi}_{\Gamma, \infty} - \phi_{\Gamma, \infty}) \, dx 
\quad [\mbox{by \reff{hatpsiinftyweak}}]
\\
&\quad \quad 
- \int_{\Omega_+}  \frac{\ve_+}{2} \nabla ( \hat{\phi}_{\Gamma, \infty} 
- \phi_{\Gamma,\infty} ) \cdot \nabla ( \psi_\Gamma - \hat{\phi}_{\Gamma, \infty} )\, dx
\\
& \quad \quad 
- \int_\Gamma \frac{\ve_+}{2} \partial_n ( \psi^+_\Gamma - \hat{\phi}^+_{\Gamma, \infty} )
( \hat{\phi}_{\Gamma, \infty} - \phi_{\Gamma,\infty}) \, dS
\quad [\mbox{since $\hat{\phi}_{\Gamma, \infty} - \phi_{\Gamma, \infty} = 0$ on $\partial \Omega$}]
\\
& \quad = - \int_\Omega \frac{\ve_\Gamma}{2} | \nabla (\psi_\Gamma 
- \hat{\phi}_{\Gamma, \infty} ) |^2 dx 
- \int_{\Omega_+} B\left( \psi_\Gamma -\frac{\phi_{\Gamma, \infty} }{2} \right) dx 
\\
& \quad \quad 
+ \int_{\Omega_-} \frac{\ve_- }{2}  \nabla (\psi_\Gamma - \hat{\phi}_{\Gamma, \infty}) \cdot 
\nabla ( \hat{\phi}_{\Gamma, \infty} - \phi_{\Gamma, \infty}) \, dx 
\\
& \quad \quad 
- \int_\Gamma \frac{\ve_-}{2} \partial_n ( \psi^-_\Gamma - \hat{\phi}^-_{\Gamma, \infty} ) 
 ( \hat{\phi}_{\Gamma, \infty} - \phi_{\Gamma,\infty} )\, dS
\quad [\mbox{by \reff{interfaceform} with $\psi = \psi_\Gamma$ and \reff{hatphiJumps} }]
\\
& \quad = - \int_\Omega \frac{\ve_\Gamma}{2} | \nabla (\psi_\Gamma 
- \hat{\phi}_{\Gamma, \infty}) |^2 dx 
- \int_{\Omega_+} B\left( \psi_\Gamma -\frac{\phi_{\Gamma, \infty} }{2} \right) dx. 
\quad [\mbox{by \reff{fact}}]
\end{align*}
Now \reff{newEGamma} follows directly from \reff{W+} and the above two expressions.
\end{proof}

We define $G_\Gamma: \hat{\phi}_{\Gamma, \infty} + H_0^1(\Omega ) \to \R \cup \{ - \infty \}$ by 
\[
G_\Gamma [\psi] = 
 - \int_\Omega \frac{\ve_\Gamma}{2} | \nabla (\psi - \hat{\phi}_{\Gamma, \infty}) |^2 dx 
- \int_{\Omega_+} B\left( \psi -\frac{\phi_{\Gamma, \infty} }{2} \right) dx + g_{\Gamma, \infty}, 
\]
where 
\begin{align*}
g_{\Gamma, \infty } 
& = \frac{\ve_- - \ve_+}{2} \int_{\Omega_+}
 \nabla  \hat{\phi}_{\Gamma, \infty} \cdot \nabla  \hat{\phi}_0 \, dx
+ \frac12 \sum_{i=1}^N Q_i (\hat{\phi}_\infty - \hat{\phi}_{\rm C} ) (x_i). 
\end{align*}
We shall call $G_\Gamma $ the PB energy functional.  
Note that by Lemma~\ref{l:EGammaReform}, $E[\Gamma] = G_\Gamma [\psi_\Gamma].$  
In fact, we have the following variational principle for the PB energy functional. 

\begin{theorem}
\label{t:PBvariation}
The Euler--Lagrange equation of the PB energy functional  
$G_\Gamma: \hat{\phi}_{\Gamma, \infty} + H_0^1(\Omega ) \to \R \cup \{ - \infty \}$  
is exactly the dielectric-boundary 
PB equation. Moreover, the functional $G_\Gamma[\cdot]$ is
uniquely maximized over 
$ \hat{\phi}_{\Gamma, \infty} + H_0^1(\Omega )$ by the solution $\psi_\Gamma$ 
to the boundary-value
problem of the PB equation \reff{PBE}, and the maximum value is exactly $E[\Gamma].$  
\end{theorem}

\begin{proof}
Direct calculations verify that the Euler--Lagrange equation for the PB energy functional
$G_\Gamma [\cdot ]$ is indeed the dielectric-boundary
PB equation; cf.\ Definition~\ref{d:PBsolution}. 
The existence of a unique maximizer can be proved exactly the same way as in 
the proof of Theorem~\ref{t:DBPBE}. These, together with Lemma~\ref{l:EGammaReform}, 
imply that the maximum value of the free energy is $G_\Gamma [\psi_\Gamma] = E[\Gamma].$   
\end{proof}

We remark that the PB functional $G_\Gamma[\cdot]$ is maximized, not minimized, among 
all the admissible electrostatic potentials. 
In general, for a charged system occupying a region $D \subseteq \R^3 $ 
with the dielectric coefficient $\ve$ and charge density $\rho\in L^2(D)$, the commonly
used energy functional of electrostatic potentials $\psi$ is given by 
\[
\psi \mapsto \int_D \left( - \frac{\ve}{2} |\nabla \psi |^2 + \rho \psi \right)dx.
\]
With suitable boundary conditions, this functional is maximized by a unique electrostatic
potential. This maximizer is exactly the solution to Poisson's equation, 
which is the Euler--Lagrange equation of this functional. 
Moreover, the maximum value of the functional is exactly the electrostatic energy
corresponding to the potential determined by Poisson's equation. See
\cite{CDLM_JPCB08} for more related discussions. 

\subsection{Definition and Formula of the Dielectric Boundary Force}
\label{ss:DBF}

Let $\Gamma$ be a dielectric boundary as given in the assumption A1
in Subsection~\ref{ss:Assumptions}. 
Let $\phi: \R^3 \to \R$ be the signed distance function to $\Gamma,$
negative in $\Omega_-$ (inside $\Gamma$) and 
positive in $\R^3 \setminus \overline{\Omega_-}$ (outside $\Gamma$).  
Then, $n= \nabla \phi $ is exactly the unit normal along $\Gamma$, pointing from 
$\Omega_-$ to $\Omega_+.$  Since $\Gamma$ is assumed to be of the class $C^3$, 
there exists $d_0 > 0$ with 
\begin{equation*}
d_0 <   \frac12 \min \left( \mbox{{\rm dist}}\, (\Gamma, \partial \Omega), 
\min_{1 \le i \le N} \mbox{{\rm dist}}\, (x_i, \Gamma) \right) 
\end{equation*}
such that the signed distance function $\phi$
is a $C^3$-function and $\nabla \phi \ne 0$ in the neighborhood
\begin{equation}
\label{N0Gamma}
\calN_0(\Gamma) = \left\{ x \in \Omega: \mbox{{\rm dist}}\,(x, \Gamma) < d_0 \right\}
\end{equation}
in $\Omega$ of $\Gamma;$ cf.\ \cite{GilbargTrudinger98} 
(Section 14.6) and \cite{KrantzParks_JDE1981}.  
Define 
\begin{equation}
\label{calV}
{\cal V} = \{ V \in C_{\rm c}^2 (\R^3, \R^3): \mbox{supp}\, (V) \subset \calN_0(\Gamma) \}.
\end{equation}
Let $V \in {\cal V}.$
For any $X \in \R^3,$ let $x = x(t, X)$ be the unique solution to the initial-value problem
\begin{equation}
\label{flow}
\dot{x} = V(x) \quad (t \in \R)  \qquad \mbox{and} \qquad x(0,X) = X,  
\end{equation}
where a dot denotes the derivative with respect to $t.$
Define $T_t(X) = x(t, X)$ for any $X \in \R^3 $ and any $t \in \R $. Then, 
$ \{ T_t \}_{t \in \R} $ is a family of diffeomorphisms and $C^2$-maps
from $\R^3$ to $\R^3$ 
with $T_0 = I $ the identity map and $T_{-t} = T_t^{-1}$ for any $t \in \R.$ 


Let $t \in \R.$ Since $\mbox{supp}\,(V) \subset \calN_0(\Gamma) \subset \Omega$, 
we have  $T_t(\Omega) = \Omega$ and $T_t( \partial \Omega) = \partial \Omega. $
Clearly, $T_t(\Omega_-) \subset \Omega$  
and $T_t(\Omega_+)  = \Omega \setminus \overline{T_t(\Omega_-)}. $
Moreover, $\Gamma_t := T_t(\Gamma) = \partial T_t(\Omega_-) 
= \overline{T_t( \Omega_-) } \cap \overline{ T_t( \Omega_+) }$ is of class $C^2$.  
Note that $x_i \in T_t (\Omega_-) $ and $T_t(x_i) = x_i$ for all $i = 1, \dots, N.$ 
Analogous to \reff{veGamma}, $\ve_{\Gamma_t}$ is defined correspondingly 
with respect to $T_t( \Omega_-)$ and $T_t( \Omega_+). $ 
We shall denote $\Gamma_t = \Gamma_t(V)$ to indicate the 
dependence of $\Gamma_t$ on $V \in {\cal V}.$ 
For each $t \in \R$, the electrostatic free energy $E[\Gamma_t(V)]$ is defined in 
\reff{DefineJGamma} (cf.\ also \reff{newEGamma}) 
with $\Gamma_t = \Gamma_t(V)$ replacing $\Gamma.$

\begin{definition}
\label{d:deltaGVG}
 Let $V\in {\cal V}.$ 
The first variation of $E[\Gamma]$ with respect to the perturbation of $\Gamma$ defined by $V$ is 
\[
\delta_{\Gamma, V} E[\Gamma] = \frac{d}{dt} E [\Gamma_t(V)] \biggr|_{t = 0}
= \lim_{t \to 0^+} \frac{E [\Gamma_t(V)] - E [\Gamma]}{t},  
\]
if the limit exists.
\end{definition}

We recall that the tangential gradient along a dielectric boundary $\Gamma$ 
is given by $\nabla_\Gamma = (I - n \otimes n) \nabla$, where $I$ is the identity matrix. 
The following theorem provides an explicit formula of the first variation
$\delta_{\Gamma, V}E[\Gamma]$, and its proof is given in Section~\ref{s:Proof}: 

\begin{theorem}
\label{t:DBF}
Let $\Gamma$ be a given dielectric boundary as described in the assumption A1 in 
Subsection~\ref{ss:Assumptions}. Let $\psi_\Gamma \in \hat{\phi}_{\Gamma, \infty} + H_0^1(\Omega)$ 
be the unique weak solution to the boundary-value problem of the dielectric-boundary PB equation
\reff{PBE}.  Then, for any $V\in {\cal V}$, the first variation
$\delta_{\Gamma, V} E[\Gamma]$ exists, and is given by  
\[
\delta_{\Gamma, V} E[\Gamma] = \int_\Gamma  q_\Gamma (V \cdot n) \, dS,  
\]
where 
\begin{align}
\label{mainG1}
q_\Gamma  
& = - \frac12 \left( \frac{1}{\ve_+} - \frac{1}{\ve_-} \right)
\left( |\ve_\Gamma  \partial_n \psi_\Gamma |^2 - \ve_\Gamma \partial_n \psi_\Gamma 
\ve_\Gamma \partial_n \phi_{\Gamma, \infty} \right)
\nonumber \\
& \quad 
+ \frac{\ve_+-\ve_-}{2} \left( | \nabla_\Gamma  \psi_\Gamma|^2 
- \nabla_\Gamma \psi_\Gamma \cdot \nabla_\Gamma \phi_{\Gamma, \infty} \right)
 + B\left(\psi_\Gamma - \frac{\phi_{\Gamma, \infty}}{2} \right).     
\end{align}
\end{theorem}

We identify $ q_\Gamma$ in \reff{mainG1} as the first variation of
$E[\Gamma]$ and denote it as $q_\Gamma  = \delta_\Gamma E[\Gamma].$ 
We call $-\delta_\Gamma E[\Gamma]$ the (normal component of the) dielectric boundary force. 



\section{Some Lemmas: The Calculus of Boundary Variations}
\label{s:Lemmas}

\subsection{Properties of the Transformation $T_t$}
\label{ss:properties}

We first recall some properties of the family of
transformations $T_t:\R^3\to\R^3$ $(t \in \R)$ defined by \reff{flow} in 
Subsection~\ref{ss:DBF} above via a vector field $V \in C_{\rm c}^2 (\R^3, \R^3).$ 
These properties hold true if we change $\R^3$ to $\R^d$ with a general dimension 
$d \ge 2.$ They can be proved by direct calculations; 
cf.\ \cite{DelfourZolesio_Book87} (Section 4 of Chapter 9). 

\begin{compactenum}
\item[(1)]
Let $X \in \R^3$ and $t \in \R.$ Let $\nabla T_t(X)$ be the
gradient matrix of $T_t$ at $X$ with its entries $(\nabla T_t(X))_{ij}
=\partial_j T_t^i(X)$ $(i,j=1, 2, 3)$, where $T_t^i$ is the $i$th component of $T_t.$ Let 
\begin{equation}
\label{Jt}
J_t(X)= \det \nabla T_t(X).
\end{equation}
Then for each $X \in \R^3$ the function $t \mapsto J_t(X)$ is a $C^2$-function and  
  \begin{equation*}
    \frac{d J_t}{dt}=( (\nabla \cdot  V ) \circ T_t) J_t,  
  \end{equation*}
where $\circ$ denotes the composition of functions or maps. 
Clearly, $\nabla T_0 = I$, the identity matrix, and $J_0  = 1.$
Moreover, 
\begin{equation}
\label{Jexp}
J_t(X) =1+t ( \nabla \cdot V) (X)  + H(t, X) t^2 \qquad \forall t \in \R \quad \forall X \in \R^3, 
\end{equation}
where $H(t, X)$ satisfies
\begin{equation}
\label{HtX}
\sup \{ | H(t, X) |: t \in \R,  X \in \R^3 \} < \infty,  
\end{equation}
since $V$ is compactly supported. 

\item[(2)]
For each $t \in \R$, we define $A_V(t): \R^3 \to \R$ by 
\begin{equation}
\label{At}
 A_V(t) (X) = J_t(X) \left( \nabla T_t (X) \right)^{-1} \left( \nabla T_t (X)\right)^{-T},   
\end{equation}
where a superscript $T$ denotes the matrix transpose.  
Clearly, $A(t) \in C^1(\R^3, \R^{3 \times 3}),$ 
and the $t$-derivative of $A_V(t)$ at each point in $\R^3$ is 
\begin{align}
 \label{A_prime_t}
 A_V'(t) & = \left[ ((\nabla \cdot V) \circ T_t)  - (\nabla T_t)^{-1} 
( ( \nabla  V ) \circ T_t) \nabla T_t  
\right. 
\nonumber \\
& \qquad \left. -  (\nabla T_t)^{-1} 
( (\nabla  V ) \circ T_t)^T (\nabla T_t) \right] A_V(t). 
\end{align}
In particular 
\begin{equation}
\label{Ap0}
A_V'(0) = (\nabla \cdot V) I - \nabla V - (\nabla V)^T.  
\end{equation}
Moreover, 
\begin{equation}
\label{Aexp} 
 A_V(t)(X) =I+tA_V'(0)(X) + K(t, X) t^2 \qquad \forall t \in \R \quad \forall X \in \R^3,
\end{equation}
where $K(t, X)$ satisfies
\begin{equation}
\label{KtX}
\sup \{ | K(t, X) |: t \in \R, X \in \R^3 \} < \infty. 
\end{equation}

\item[(3)]
For any $u \in L^2(\Omega)$ and $t \in \R,$ 
$u  \circ T_t \in L^2(\Omega)$ and  $ u \circ T_t^{-1} \in L^2(\Omega)$. Moreover,
\begin{equation}
\label{f_T_t}
    \lim_{t\to 0} u \circ T_t = u   \qquad \mbox{and}\qquad
    \lim_{t\to 0}u  \circ T_t^{-1} = u  \qquad \mbox{in } L^2(\Omega). 
\end{equation}
For any $u \in H^1(\Omega)$ and $t \in \R,$
$u  \circ T_t \in H^1(\Omega)$ and  $ u \circ T_t^{-1} \in H^1(\Omega)$. Moreover,
\begin{align}
\label{dudut}
& \nabla ( u  \circ T_t^{-1})=  (\nabla T_t^{-1})^T \left( \nabla u \circ T_t^{-1} \right)
\quad \mbox{and} \quad \nabla ( u  \circ T_t)=( \nabla T_t)^T \left( \nabla u \circ T_t \right), 
\\
\label{TtuH1conv}
&    \lim_{t\to 0} u \circ T_t = u   \qquad \mbox{and}\qquad
    \lim_{t\to 0}u  \circ T_t^{-1} = u  \qquad \mbox{in } H^1(\Omega).   
\end{align}
If $u \in H^2(\Omega),$ then 
\begin{equation}
\label{dtuTt}
\lim_{t \to 0} \left\| \frac{ u\circ T_t - u }{t} - \nabla u \cdot V \right\|_{H^1(\Omega)} = 0. 
\end{equation}
\end{compactenum}

\subsection{Tangential Force}
\label{ss:TangentialForce}

This is a geometrical property on the effect of tangential component of a velocity
vector field to the motion of an interface. We shall state and prove it for a general
$d$-dimensional space $\R^d$ with $d \ge 2.$ 
We assume that $\Gamma$ is a $C^3$, closed, hypersurface in $\R^d.$ 
We denote as before the interior and exterior of $\Gamma$ by $\Omega_-$ and $\Omega_+$, 
respectively.  We also denote by $n$ the unit vector normal to the surface $\Gamma$ 
at a point on $\Gamma$, pointing from the interior to 
exterior of $\Gamma$. We assume that $V \in C_{\rm c}^2(\R^d, \R^d)$
and define the transformation $T_t: \R^d \to \R^d$ $(t \in \R)$ by $T_t(X) = x(t, X)$
for any $X \in \R^d$ and $t \in \R$,  where $x = x(t,X)$ is the unique solution to the
initial-value problem \reff{flow}. 

\begin{lemma}
\label{l:TForce}
If $V\cdot n=0$ on $\Gamma$, then there exist $t_0 > 0$ and $C > 0$, depending 
on $\Gamma$ and $V$, such that
\begin{align} 
\label{Ct2}
&
\mbox{{\rm dist}}\,(T_t(X), \Gamma ) \leq C t^2   \qquad  \forall X \in \Gamma \ \forall
t \in \R \mbox{ with }  |t| \le t_0,  
\\
\label{mismatch}
&
\left| \left\{ x\in \R^d:  \chi_{T_t(\Omega_{\rm s}) }(x) \neq \chi_{\Omega_{\rm s}}(x) \right\} \right| 
\leq C t^2 \qquad \mbox{if } |t | \le t_0, \ \mbox{ {\rm s} }  = - \mbox{ or } +. 
\end{align}
\end{lemma}

\begin{proof}
We first prove \reff{Ct2}.  Since $\Gamma$ is of $C^3$, there exist finitely many 
open balls in $\R^d$ such that their union covers $\Gamma$ and that the intersection of $\Gamma$ with
each of such open balls is the graph of a $C^3$ function in a local coordinate system. 
Let us fix one of such open balls, ${\cal B}$, and assume without loss of 
generality that the corresponding $C^3$-function
is given by $h: \prod_{j=1}^{d-1} (a_j-\delta, b_j+\delta) \to \R$ 
for some $a_j, b_j \in \R$ with $a_j < b_j$ $(j = 1, \dots, d-1)$ and $\delta > 0,$
where ${\cal B}\cap \Gamma$ is the graph of $h$ on $ \prod_{j=1}^{d-1} (a_j-\delta, b_j+\delta)$. 
Here, we use the local coordinate system depending on $\cal B$ with the notation 
\[
X = (X', X_d) \in \R^d, \qquad X'=(X_1, \cdots, X_{d-1}) \in \R^{d-1}, \qquad X_d \in \R.   
\]
So, $X_d = h(X')$ for all $ X'\in \prod_{j=1}^{d-1} (a_j-\delta, b_j+\delta)$. 
We shall assume that $\delta  > 0$ is small enough so that 
the corresponding concentric balls with radius reduced by $\delta$ still cover $\Gamma,$
and that in particular the union of the graphs of $h$ on $\prod_{j=1}^{d-1} (a_j, b_j)$
with all open balls $\cal B$ is still $\Gamma.$ 
Moreover, since all $\Gamma$ and $T_t$ $(t\in \R)$ are smooth, 
there exists $t_0' \in (0, 1)$ such that, for any $t \in \R$ with $|t| \le t_0'$ 
and for any $X = (X', X_d) \in \Gamma$ with $X' \in \prod_{j=1}^{d-1} (a_j, b_j)$, 
the coordinate $(T_t(X))'$ of $T_t(X) = ( (T_t(X))', (T_t(X))_{d}) $
is still in $ \prod_{j=1}^{d-1} (a_j-\delta, b_j+\delta)$, the domain of $h$. 
With this setup, we shall prove that there exist $t_0 \in (0, t_0')$ and $C > 0$ such that 
\begin{equation}
\label{geometry1}
\mbox{{\rm dist}}\, ( T_t(X_0), \Gamma ) \leq C t^2 \quad \mbox{if } X_0 = (X_0', X_{0d}) \in \Gamma
\mbox{ with } X_0' \in \prod_{j=1}^{d-1} (a_j, b_j) \mbox{ and } |t | \le t_0. 
\end{equation}
This then implies \reff{Ct2}. 

\begin{figure}[hbtp]
\centering
\includegraphics[scale=0.7]{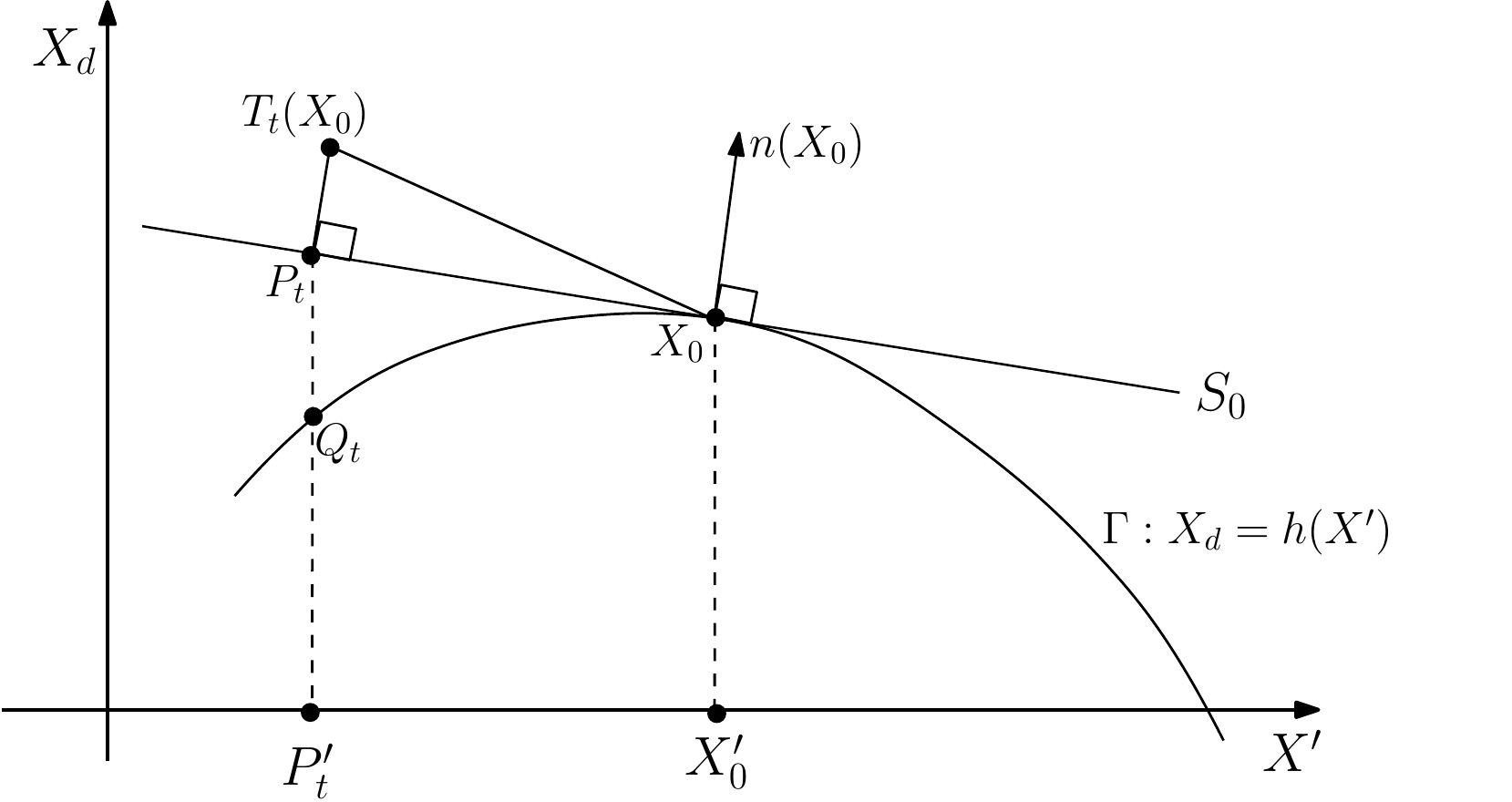}
\caption{A local graph representation of the surface $\Gamma$.} 
\label{f:graph}
\end{figure}

Let us fix an arbitrary point $X_0= (X_0', X_{0d}) \in \Gamma$ with 
$X_0' \in \prod_{j=1}^{d-1}(a_j, b_j)$ and $X_{0d} = h(X_0')$; cf.\ Figure~2. 
The equation for the plane, $S_0$, that is tangent to the surface $\Gamma$ at $X_0\in \Gamma$ 
is given by 
\[
( X - X_0) \cdot n(X_0) = 0 \qquad \mbox{with } X = (X', X_d) \in \R^d, 
\]
where $n(X_0)$ is the unit vector normal to $\Gamma$ at $X_0$,  
\begin{equation}
\label{nX0}
n(X_0)=\frac{(-\nabla_{X'} h(X_0'), 1)}{\sqrt{1+|\nabla_{X'} h(X_0')|^2}}.
\end{equation}
Let $t \in \R$ with $|t| \le t_0'.$ Denote by $P_{t} = (P'_t, P_{td}) \in S_0$ the point 
of the orthogonal projection of the vector $T_t(X_0) -X_0$ on this tangent plane $S_0$, given by 
\begin{equation}
\label{Pt}
P_t = -\left[ \left(T_t(X_0) - X_0\right) \cdot n(X_0) \right] n(X_0) + T_t(X_0).
\end{equation}
Denote 
\begin{equation}
\label{Qt}
Q_t =(Q_t', Q_{td}) = ( P'_t, h(P'_t) ) \in \Gamma. 
\end{equation}
We show that 
\begin{equation}
\label{geometry2}
|T_t(X_0) -Q_t  | \leq C t^2 \qquad \mbox{if } |t| \le t_0  
\end{equation}
for some constants $t_0 \in (0, 1)$ 
and  $C > 0$ independent of $X_0.$ This should imply \reff{geometry1} as 
\[
\mbox{{\rm dist}}\,(T_t(X_0), \Gamma) \leq |T_t(X_0) -Q_t|.
\]

To prove \reff{geometry2}, it suffices by the triangle inequality to prove 
\begin{align}
\label{star1}
& |T_t(X_0) -P_t| \leq C t^2, \\
\label{star2} 
& |P_t -Q_t| = |P_{td}- Q_{td} | \leq C t^2, 
\end{align}
for all $t$ with $|t| \le t_0.$
Here and below in the proof, $C$ denotes a generic constant independent of
$X_0$ and $t \in [-t_0, t_0].$ 
Since $V\in C_{\rm c}^2(\R^3, \R^3)$, we have by \reff{flow} that 
\begin{align*}
\partial_{tt} T_t(X) 
&= \partial_{tt} x (t, X) = \partial_t ( V(x(t,X))) = (\nabla V (x (t, X)))  \partial_t x(t,X) 
\\
& = ( \nabla V (x (t, X))) V(x(t,X)) \le C \qquad \forall t \in \R \  \forall X \in \R^d.  
\end{align*}
Therefore, by Taylor's expansion, \reff{Pt}, \reff{flow}, 
the fact that $|n(X_0)| = 1,$ and the assumption that $V(X_0) \cdot n(X_0) = 0$, we have 
\begin{align*}
|T_t(X_0) - P_t| 
&= |[ x(t, X_0)  - X_0 ] \cdot n(X_0)| 
\\
&\le |  x(0, X_0) - X_0 | + |t \partial_t x(0, X_0) \cdot n (X_0) | + C t^2 
\\
& = | t V(X_0) \cdot n(X_0) | +  C t^2 
\\
&= Ct^2, 
\end{align*}
proving \reff{star1}.

Since $V$ is compactly supported, we have by \reff{flow} and Taylor's expansion that 
\[
|T_t(X_0) - X_0 | = |x(t, X_0)  - X_0| 
= | t \partial_t x(\xi_t, X_0) | = |t | | V (x (\xi_t, X_0)) | \le C |t|,
\]
where $\xi_t$ is in between $0$ and $t$. 
This and \reff{star1} imply 
\begin{equation}
\label{t}
|P_t' -X'_0|\leq |P_t-X_0| \leq |P_t-T_t(X_0)|+|T_t(X_0)-X_0| \leq C |t| 
\quad \mbox{ if } |t| < t_0.
\end{equation}
By \reff{Qt} and Taylor's expansion, 
\begin{align}
\label{QtdExpand}
Q_{td}&=h(P_t') 
\nonumber \\
&= h(X'_0) +\nabla_{X'} h(X'_0)\cdot (P'_t -X'_0) 
+ \frac12 \nabla^2_{X'} h(Y_t')(P_t'-X'_0) \cdot (P_t' - X_0')
\end{align}
for some $Y_t' \in \prod_{j=1}^{d-1}(a_j, b_j).$
Since $(P_t-X_0) \cdot n(X_0) = 0,$  $ X_{0d}=h(X'_0),$ and $n(X_0)$ is
given by \reff{nX0}, we have that 
\[
P_{td}= h(X'_0) +\nabla_{X'} h(X_0) \cdot (P'_t-X'_0).
\]
This, together with \reff{QtdExpand} and \reff{t}, 
implies \reff{star2}. The constant $C$ depends on $h$ and $V$, and hence on $\Gamma$ and $V$  only.

We now prove \reff{mismatch}.  We only consider the case $s = -$, as the case
$s = +$ can be treated similarly. 
Moreover, for $t = 0$, the inequality in \reff{mismatch} holds true obviously, since
$T_0$ is the identity map. 
So, let $t \in \R$ be such that $0 < |t| \le t_0.$ We further assume that $0 < t \le t_0$
as the case $-t_0 \le t < 0$ is similar.  

We claim that 
\begin{equation}
\label{symdiff}
\Omega_-\Delta T_t(\Omega_-) := 
(\Omega_-\setminus T_t( \Omega_-) ) \cup (T_t(\Omega_-)\setminus \Omega_-)
\subseteq \{ X\in \R^d: \mbox{{\rm dist}}\,(X, \Gamma) \le C t^2\}, 
\end{equation}
where the constant $C$ is exactly the same as in \reff{Ct2}. 
In fact, if $X \in \Omega_-\setminus T_t( \Omega_-),$ 
then $X\in \Omega_-$ and $T_{-t} (X) \in \overline{\Omega_+}.$  
Hence, $ \{ T_s (T_{-t}(X)) \}_{0\le s \le t}$ is a $C^3$-curve in $\R^d$, 
connecting one endpont $T_0(T_{-t}(X)) = T_{-t}(X) \in \overline{\Omega_+}$ 
to the other $T_{t}(T_{-t}(X)) = X \in \Omega_-$. Since $\Gamma$ is a closed
hypersurface in $\R^d$, there must exist $s_0 \in (0, t)$ such that 
$T_{s_0}(T_{-t} (X)) = T_{s_0-t} (X) \in \Gamma. $ Hence, by \reff{Ct2} with 
$t -s_0$ and $T_{s_0-t}(X) $ replacing $t$ and $X$, respectively,    
\[
\mbox{{\rm dist}}\,(X, \Gamma) = 
\mbox{{\rm dist}}\,(T_{t-s_0} (T_{s_0-t}(X)), \Gamma) \le C (t-s_0)^2 \le Ct^2.  
\]
Similarly, if $X \in T_t(\Omega_-) \setminus \Omega_-$, then 
$ \mbox{{\rm dist}}\,(X, \Gamma) \le Ct^2.$  Hence, \reff{symdiff} holds true. 

By \reff{symdiff}, we have 
\[
\left|\{ X\in \R^d: 
\chi_{\Omega_-(t)} \neq  \chi_{\Omega_-(X)} \} \right|
= \left| \Omega_- \Delta T_t( \Omega_-) \right| 
\le \left| \{ X\in \R^d: \mbox{{\rm dist}}\,(X, \Gamma) \le C t^2\} \right|.  
\]
This implies \reff{mismatch} (for ${\rm s} = -$), as  
the right-hand side of the above inequality is bounded by $Ct^2$
if $|t| \le t_0$ with a possible smaller $t_0$ 
(cf.\ Lemma~2.1 in \cite{Li_Order_SIMA06}).  
\end{proof}

\subsection{Continuity and Differentiability}
\label{ss:contdiff}

Let $\Gamma$ be a dielectric boundary satisfying the assumptions in A1 of 
Subsection~\ref{ss:Assumptions} and $V \in \calV$ (cf.\ \reff{calV}). 
Let $\{T_t\}_{t \in \R}$ be the corresponding family of diffeomorphisms defined by \reff{flow}. 
Let $\hat{\phi} \in W^{1,1}(\Omega)$ satisfy  \reff{Wpsi0}. 
We consider the approximations $\hat{\phi} \circ T_t.$ 
Note that $\hat{\phi}  \circ T_t - \hat{\phi}$ 
and $\nabla \hat{\phi} \cdot V$ vanish in any small neighborhood of $ \cup_{i=1}^N x_i$, 
as $V (X) = 0$ and $T_t (X)  = X $ for any $X$ in such a neighborhood and any $t \in \R.$ 


\begin{lemma}
\label{l:phi0}
Let $\hat{\phi} \in \hat{\phi}_{\rm C} + H^{1}(\Omega)$ satisfy  \reff{Wpsi0}.  We have 
\begin{align}
\label{phi00}
&\lim_{t\to 0} \| \hat{\phi}  \circ T_t - \hat{\phi}  \|_{H^1(\Omega)} = 0.
\end{align} 
Moreover, $\nabla \hat{\phi} \cdot V \in H^1(\Omega)$ and 
\begin{align}
\label{phi0t}
&\lim_{t\to 0} \left\| \frac{ \hat{\phi} \circ T_t-\hat{\phi}}{t}
- \nabla \hat{\phi}  \cdot V \right\| _{H^1(\Omega)} = 0.  
\end{align}
\end{lemma}

\begin{proof}
Let $\sigma > 0$ be such that $B_{\sigma} := \cup_{i=1}^N B(x_i, \sigma) \subset \Omega$ and 
$V = 0 $ on $B_{\sigma}.$ Then, there exists $\tilde{\phi} \in C^\infty(\Omega)\cap H^2(\Omega)$ such that 
$\tilde{\phi} = 0$ in $B_{\sigma/2}$ and $\tilde{\phi} = \hat{\phi}$ a.e.\ in 
$\Omega \setminus B_\sigma.$ These imply that 
$\tilde{\phi} \circ T_t - \tilde{\phi} = \hat{\phi}  \circ T_t - \hat{\phi},$
and $\nabla \tilde{\phi} \cdot V = \nabla \hat{\phi} \cdot V$ a.e.\ in $\Omega$ for all $t.$
This implies that $\nabla \hat{\phi} \cdot V \in H^1(\Omega).$ Moreover, 
it follows from \reff{TtuH1conv} that 
\[
\lim_{t\to 0} \| \hat{\phi}  \circ T_t - \hat{\phi}  \|_{H^1(\Omega)}
= \lim_{t\to 0} \| \tilde{\phi} \circ T_t - \tilde{\phi} \|_{H^1(\Omega)} = 0, 
\]
implying \reff{phi00}, and from \reff{dtuTt} that 
\[
\lim_{t\to 0} \left\| \frac{ \hat{\phi}\circ T_t-\hat{\phi}} {t}
- \nabla \hat{\phi}  \cdot V \right\| _{H^1(\Omega)} 
= \lim_{t\to 0} \left\| \frac{\tilde{\phi}\circ T_t-\tilde{\phi}}{t}
- \nabla \tilde{\phi} \cdot V \right\| _{H^1(\Omega)} = 0,   
\] 
implying \reff{phi0t}. 
\end{proof}

We recall that $\phi_{\Gamma, \infty} \in H^1(\Omega)\cap C(\overline{\Omega})$ 
is the unique weak solution
to the boundary-value problem \reff{hatpsiinfty}, defined in \reff{hatpsiinftyweak}.
Similarly, $\phi_{\Gamma_t, \infty}\in H^1(\Omega)\cap C(\overline{\Omega})$ 
for each $t \in \R$ is  the
unique weak solution to the same boundary-value problem with $\Gamma_t = T_t(\Gamma)$
replacing $\Gamma.$

\begin{lemma}
\label{l:pGinfty}
\begin{compactenum}
\item[\rm (1)]
There exists a unique $\zeta_{\Gamma, V} \in H^1_0(\Omega)$ such that
\begin{align}
\label{hat_w}
\int_\Omega \varepsilon_\Gamma \nabla \zeta_{\Gamma, V} \cdot \nabla \eta \, dx
=-\int_\Omega \varepsilon_\Gamma A_V'(0) \nabla \phi_{\Gamma, \infty} \cdot\nabla\eta \, dx
\qquad \forall \eta \in H_0^1(\Omega),
\end{align}
where $A'_V(0) $ is defined in \reff{Ap0}.  Moreover, the mapping
$V \mapsto \zeta_{\Gamma, V} $ is linear in $V$, i.e.,
\[
\zeta_{\Gamma, c_1 V_1 + c_2 V_2} = c_1 \zeta_{\Gamma, V_1 } + c_2 \zeta_{\Gamma, V_2 }
\qquad \mbox{for all }\, V_1, V_2 \in {\cal V} \mbox{ and }  c_1, c_2 \in \R.
\]

\item[\rm (2)]
By modifying the value of $\zeta_{\Gamma, V}$ on a set of zero Lebesgue measure, 
we have that  $\zeta_{\Gamma, V}^{\rm s}
\in H^2(\Omega_{\rm s}) \cap C^1(\Omega_{\rm s})$  
for ${\rm s} = -$ or $+$. Moreover, 
\begin{align}
\label{DwG}
& \Delta \zeta_{\Gamma, V}
 = -\nabla \cdot \left[ A_V'(0) \nabla  \phi_{\Gamma, \infty} \right]
= \Delta ( \nabla \phi_{\Gamma, \infty} \cdot V )
\qquad \text{in } \Omega_- \cup \Omega_+,
\\
\label{jump_pnwG}
& \llbracket \varepsilon_\Gamma \partial_n \zeta_{\Gamma, V} \rrbracket_\Gamma
= - \llbracket \varepsilon_\Gamma  A'_V(0) \nabla \phi_{\Gamma, \infty} \cdot n   \rrbracket_\Gamma
\qquad \mbox{on } \Gamma.
\end{align}

\item[\rm (3)]
We have
\begin{align}
\label{limpGi}
&\lim_{t\to 0} \| \phi_{\Gamma_t, \infty} \circ T_t - \phi_{\Gamma,\infty}  \|_{H^1(\Omega)}= 0,
\\
\label{limzetaG}
& \lim_{t \to 0} \left\| \frac{\phi_{\Gamma_t, \infty} \circ T_t - \phi_{\Gamma, \infty}}{t}
-  \zeta_{\Gamma, V} \right\|_{H^1(\Omega)} = 0.
\end{align}

\item[\rm (4)]
If $V\cdot n = 0$ on $\Gamma$, then $\zeta_{\Gamma, V} = \nabla \phi_{\Gamma, \infty}
\cdot V $ in $\Omega.$
\end{compactenum}
\end{lemma}

\begin{proof}
(1) The existence and uniqueness of $\zeta_{\Gamma, V} \in H^1_0(\Omega)$ that satisfies
\reff{hat_w} follow from the Lax--Milgram Lemma \cite{EvansBook2010,GilbargTrudinger98}.
By \reff{Ap0},  $A'_V(0)$ is linear in $V$. Therefore, by the definition \reff{hat_w}
of $\zeta_{\Gamma, V} \in H_0^1(\Omega),$ $\zeta_{\Gamma, V}$ is linear in $V.$

(2) Let $\mbox{s} $ denote $-$  or $+.$ Note by
\reff{phiGireg}, \reff{calV}, and \reff{Ap0} that $A_V'(0) \nabla \phi_{\Gamma, \infty}
\in C^1({\Omega_{\rm s}})\cap H^1(\Omega_{\rm s})$. 
For any $\eta \in C_{\rm c}^1(\Omega)$ with $\mbox{supp}\,(\eta)
\subset \Omega_{\rm s},$ we have by \reff{hat_w} and the Divergence Theorem that 
\[
\int_{\Omega_{\rm s}} \ve_{\rm s} \nabla \zeta_{\Gamma, V} \cdot \nabla \eta \, dx
= \int_{\Omega_{\rm s}} \ve_{\rm s}
\nabla \cdot \left[ A_V'(0) \nabla  \phi_{\Gamma, \infty} \right] \eta  \, dx.
\]
Hence,
$ - \Delta \zeta_{\Gamma, V} = \nabla \cdot \left[ A_V'(0) \nabla  \phi_{\Gamma, \infty} \right]
$ in $\Omega_{\rm s}$.
Since the right-hand side is in $L^2(\Omega_{\rm s})\cap C(\Omega_{\rm s})$, it follows from
the elliptic regularity theory \cite{Lady_EllipticBook1968, GilbargTrudinger98} that
$\zeta_{\Gamma, V}^{\rm s} \in H^2(\Omega_{\rm s})\cap C^1(\Omega_{\rm s}),$ 
after a possible modification of the value of $\zeta_{\Gamma, V}$ on a set of 
zero Lebesgue measure.  Moreover, the first equality in \reff{DwG} follows.

Let us denote by $V^i $ $(i = 1, 2, 3)$ the components of $V.$
With the conventional summation notation (i.e., repeated indices are summed),
we have by \reff{Ap0}, \reff{phiGireg}, and \reff{DphiGi0} that
\begin{align}
\label{hidden}
&- \nabla \cdot (  A_V'(0) \nabla \phi_{\Gamma, \infty} )
\nonumber \\
&\quad 
= \nabla \cdot \left[ \nabla V + (\nabla V)^T - (\nabla \cdot V) I  \right] \nabla
\phi_{\Gamma, \infty}
\nonumber \\
& \quad = \partial_i \left( \partial_j V^i \partial_j \phi_{\Gamma, \infty}
+ \partial_i V^j \partial_j \phi_{\Gamma, \infty}
- \partial_k V^k \partial_i \phi_{\Gamma, \infty} \right)
\nonumber \\
& \quad = 2 \partial_{ij}  \phi_{\Gamma, \infty} \partial_i V^j + \partial_{ii}V^j
\partial_j \phi_{\Gamma, \infty}
\nonumber \\
& \quad = \partial_{ii} \partial_j \phi_{\Gamma, \infty} V^j
+ 2 \partial_{ij}  \phi_{\Gamma, \infty} \partial_i V^j + \partial_{ii}V^j
\partial_j \phi_{\Gamma, \infty}
\quad \mbox{[since $\partial_{ii} \phi_{\Gamma, \infty} = 0$]}
\nonumber \\
&\quad = \partial_{ii} \left( \partial_j \phi_{\Gamma, \infty} V^j \right)
\nonumber \\
& \quad = \Delta ( \nabla  \phi_{\Gamma, \infty} \cdot V)
\qquad \mbox{in }  \Omega_-\cup \Omega_+,
\end{align}
implying the second equation in \reff{DwG}.

Since $\zeta_{\Gamma, V} \in H^1_0(\Omega) $ and $\Delta \zeta_{\Gamma, V}
\in L^2(\Omega_{\rm s})$ for s being $-$ or $+$,
and since the unit normal $n$ at the $\Gamma$ points from $\Omega_-$ to $\Omega_+$,
we have by the Divergence Theorem that both sides of the equation in \reff{hat_w} are
\begin{align*}
\int_\Omega \varepsilon_\Gamma \nabla \zeta_{\Gamma, V} \cdot \nabla \eta \, dx
& = \int_{\Omega_-} \varepsilon_- \nabla \zeta_{\Gamma, V} \cdot \nabla \eta \, dx
+ \int_{\Omega_+} \varepsilon_+ \nabla \zeta_{\Gamma, V} \cdot \nabla \eta \, dx
\nonumber
\\
&=- \int_{\Omega_-} \varepsilon_- \Delta {\zeta}_{\Gamma, V}  \eta \, dx
-\int_{\Omega_+ } \varepsilon_+ \Delta {\zeta}_{\Gamma, V}  \eta \, dx
-\int_\Gamma \llbracket \varepsilon_\Gamma \partial_n {\zeta}_{\Gamma, V}
\rrbracket_{\Gamma}  \eta  \, dS,
\end{align*}
and
\begin{align*}
& -\int_\Omega \varepsilon_\Gamma A_V'(0) \nabla \phi_{\Gamma,\infty} \cdot \nabla \eta \, dx
\nonumber
\\
&\quad
= -\int_{\Omega_-} \varepsilon_- A_V'(0) \nabla  \phi_{\Gamma, \infty} \cdot \nabla \eta \, dx
-\int_{\Omega_+} \varepsilon_+ A_V'(0) \nabla \phi_{\Gamma, \infty}  \cdot \nabla \eta \, dx
\nonumber
\\
&\quad
= \int_{\Omega_-} \varepsilon_- \nabla \cdot ( A_V'(0) \nabla \phi_{\Gamma, \infty}) \eta \, dx
+ \int_{\Omega_+} \varepsilon_+ \nabla \cdot ( A_V'(0) \nabla \phi_{\Gamma, \infty} ) \eta \, dx
\nonumber
\\
&\quad \quad
+\int_\Gamma \llbracket \varepsilon_\Gamma  A_V'(0) \nabla
\phi_{\Gamma, \infty}   \cdot n\rrbracket_{\Gamma} \eta  \, dS,
\end{align*}
respectively.  These, together with and \reff{DwG}, imply \reff{jump_pnwG}.

(3) Replacing $\Gamma$, $\phi_{\Gamma, \infty}$, and $\eta$ by $\Gamma_t$, $\phi_{\Gamma_t, \infty}$,
and $\eta \circ T_t^{-1}$ for $t \in \R$, respectively,  in the weak formulation
\reff{hatpsiinftyweak},
we get by the change of variable $ x = T_t(X)$ that
\[
\int_\Omega \ve_{\Gamma} A_V(t) \nabla ( \phi_{\Gamma_t, \infty} \circ T_t) \cdot
\nabla \eta \, dX  = 0  \qquad \forall \eta \in H_0^1(\Omega).
\]
This and \reff{hatpsiinftyweak} imply for any $\eta \in H_0^1(\Omega)$ that
\begin{equation}
\label{phiGtH1}
\int_\Omega \ve_{\Gamma} \nabla ( \phi_{\Gamma_t, \infty} \circ T_t - \phi_{\Gamma, \infty} )
\cdot  \nabla \eta \, dX =  \int_\Omega \ve_\Gamma [ I - A_V(t)]
\nabla (\phi_{\Gamma_t, \infty} \circ T_t) \cdot  \nabla \eta \, dX.
\end{equation}
It follows from a change of variable, \reff{Jexp}, \reff{HtX}, \reff{Aexp}, and \reff{KtX}
that $\| \nabla ( \phi_{\Gamma_t, \infty} \circ T_t) \|_{L^2(\Omega)}$ is bounded
uniformly in $t.$ 
Setting $\eta= \phi_{\Gamma_t, \infty} \circ T_t -\phi_{\Gamma, \infty} \in H_0^1(\Omega)$ in
\reff{phiGtH1}, we then obtain \reff{limpGi} by \reff{Aexp}, \reff{KtX}, and
the Cauchy--Schwarz and Poincar\'e inequalities.

Dividing both sides of \reff{phiGtH1} by $t \ne 0$ and setting now
$\eta = ( \phi_{\Gamma_t, \infty} \circ T_t - \phi_{\Gamma, \infty})/t - \zeta_{\Gamma, V} $
in the resulting equation and also in \reff{hat_w}, we have
by the Cauchy--Schwarz inequality that
\begin{align*}
&\int_\Omega \ve_{\Gamma}  \left| \nabla \left(  \frac{ \phi_{\Gamma_t, \infty} \circ T_t
- \phi_{\Gamma, \infty}}{t}-\zeta_{\Gamma, V} \right) \right|^2 dX
\\
&\quad
= \int_\Omega \ve_\Gamma \left[ \frac{I - A_V(t)}{t} + A_V'(0) \right]
\nabla (  \phi_{\Gamma_t, \infty} \circ T_t)  \cdot
\nabla \left(  \frac{ \phi_{\Gamma_t, \infty} \circ T_t
- \phi_{\Gamma, \infty}}{t}-\zeta_{\Gamma, V} \right) dX
\\
& \quad \quad
+ \int_\Omega \ve_\Gamma A_V'(0)
\nabla \left[  \phi_{\Gamma, \infty} -   \phi_{\Gamma_t, \infty} \circ T_t) \right] \cdot
\nabla \left(  \frac{ \phi_{\Gamma_t, \infty} \circ T_t
- \phi_{\Gamma, \infty}}{t}-\zeta_{\Gamma, V} \right) dX
\\
& \quad
\le C \left\|\frac{ A_V(t) - I - t A_V'(0)}{t} \right\|_{L^\infty(\Omega)}
\| \phi_{\Gamma_t, \infty} \circ T_t  \|_{H^1(\Omega)}
\left\| \frac{ \phi_{\Gamma_t, \infty} \circ T_t - \phi_{\Gamma, \infty}}{t}-\zeta_{\Gamma, V}
\right\|_{H^1(\Omega)}
\\
&\quad \quad
+ C \| \phi_{\Gamma, \infty} -   \phi_{\Gamma_t, \infty} \circ T_t \|_{H^1(\Omega)}
\left\| \frac{ \phi_{\Gamma_t, \infty} \circ T_t - \phi_{\Gamma, \infty}}{t}-\zeta_{\Gamma, V}
\right\|_{H^1(\Omega)}.
\end{align*}
This, together with Poincar\'e's inequality, \reff{Aexp}, \reff{KtX}, and \reff{limpGi}, 
leads to \reff{limzetaG}.

(4) Assume now $V\cdot n = 0 $ on $\Gamma$.
Recall from subsection~\ref{ss:DBF} that the signed distance to $\Gamma$,
$\phi: \R^3 \to \R$, which is negative in $\Omega_-$ and positive outside $\Gamma$,
is in fact a $C^3$-function and also $\nabla \phi \ne 0$
in $\calN_0(\Gamma)$, a neighborhood of $\Gamma$ in $\Omega$; cf.\ \reff{N0Gamma}.
Moreover, $n = \nabla \phi$ on $\Gamma.$ We define $n = \nabla \phi $ on $\calN_0(\Gamma).$
Now, since $V\cdot n = 0$ on $\Gamma$,
 by Lemma~\ref{l:TForce}, there exists $t_0 = t_0(V)> 0$ and a constant $C_0 > 0$ which
may depend on $\Gamma$, such that
$\mbox{{\rm dist}}\,(x, \Gamma) \le C_0 t^2$ for all $x \in \Gamma_t$ and all $t \in [-t_0, t_0]$.
Let $D(t) = \{ x \in \Omega: \mbox{{\rm dist}}\,(x, \Gamma) \le C_0t^2\}.$
Then $\Gamma_t \subset D(t),$
and hence $\ve_{\Gamma_t} = \ve_\Gamma$ on $\Omega \setminus D(t)$, for all $|t| \le t_0.$
Moreover, the measure $|D(t)| = O(t^2)$ as $t \to 0$; cf.\ Lemma~2.1 in \cite{Li_Order_SIMA06}.

Let $ h_t = (\phi_{\Gamma_t, \infty} - \phi_{\Gamma, \infty}) / t $ with $|t| \le t_0.$ 
We have now by \reff{hatpsiinfty} and that with $\Gamma_t $ replacing $\Gamma$ that
\begin{equation}
\label{veGtveG}
\int_\Omega \ve_{\Gamma_t}  \nabla h_t \cdot \nabla \eta \, dx
= - \int_\Omega \frac{ \ve_{\Gamma_t} - \ve_\Gamma}{t} \nabla \phi_{\Gamma, \infty}
\cdot \nabla \eta \, dx \qquad \forall \eta \in H_0^1(\Omega).
\end{equation}
Setting $\eta = h_t\in H_0^1(\Omega)$, we
have by the uniform boundedness of $\nabla \phi_{\Gamma_t, \infty}$ (cf.\ \reff{phiGammabounds}),
the Poincar\'e and Cauchy--Schwarz inequalities, and the fact
that $|D(t)| = O(t^2)$ as $t \to 0$  that
\begin{align*}
\| h_t \|_{H^1(\Omega)} \le C \left( \int_{D(t)}
\left| \frac{\ve_{\Gamma_t} - \ve_\Gamma}{t} \right|^2 dx \right)^{1/2} \le C,
\end{align*}
where $C > 0$ is a generic constant, independent of $t$.
Thus, there exists a subsequence of $h_t$ $( |t| \le t_0 )$, not relabeled,
and some $h\in H^1_0(\Omega)$, such that $h_t \to h$
weakly in  $H^1(\Omega)$ as $ t\to 0$.

Working on this subsequence, we have by \reff{hatpsiinftyweak}
(with $\phi_{\Gamma_t, \infty}$ replacing $\phi_{\Gamma, \infty}$), \reff{veGtveG},
\reff{phiGammabounds}, and the fact that the measure $|D(t)| = O(t^2) \to 0 $ as $t \to 0$
that  for any $\eta \in H_0^1(\Omega)$
\begin{align*}
\left| \int_\Omega \ve_{\Gamma}  \nabla h_t \cdot \nabla \eta \, dx  \right|
&
= \left| \int_\Omega ( \ve_\Gamma - \ve_{\Gamma_t} ) \nabla h_t \cdot \nabla \eta \, dx
- \int_\Omega \frac{ \ve_{\Gamma_t} - \ve_\Gamma}{t} \nabla \phi_{\Gamma, \infty}
\cdot \nabla \eta \, dx \right|
\\
&
\le \left|   \int_{D(t)} (\ve_{\Gamma_t} - \ve_\Gamma) \nabla h_t \cdot \nabla \eta \, dx \right|
+ \left|  \int_{D(t)} \frac{ \ve_{\Gamma_t} - \ve_\Gamma}{t} \nabla \phi_{\Gamma, \infty}
\cdot \nabla \eta \, dx \right|
\\
& \le C \left[ \| h_t \|_{H^1(\Omega)} 
+ \frac{1}{|t|} | D(t)|^{1/2} \right]
\left( \int_{D(t)} |\nabla \eta|^2 dx \right)^{1/2}
\\
& \le C \left( \int_{D(t)} |\nabla \eta|^2 dx \right)^{1/2}
\\
& \to 0
\qquad \mbox{as } t \to 0.
\end{align*}
Since $h_t \to h$ weakly in $H^1(\Omega)$, we have
\[
\int_\Omega \ve_{\Gamma}  \nabla h \cdot \nabla \eta \, dx = 0
\qquad \forall \eta \in H_0^1(\Omega).
\]
Setting $\eta = h \in H_0^1(\Omega)$, we see that $h = 0$ in $H^1_0(\Omega).$

We now show that $\zeta_\Gamma = \nabla \phi_{\Gamma, \infty} \cdot V$ in $\Omega.$
Let $\eta \in L^2(\Omega)$ and $t \ne 0$.
Since $h_t \to h = 0$ weakly in $H^1(\Omega)$, we
have by the properties of the transformations $T_t $ $(t \in \R)$
\reff{f_T_t}, \reff{dtuTt}, and \reff{Jt}--\reff{HtX} that
\begin{align*}
&\int_\Omega \frac{\phi_{\Gamma_t, \infty} \circ T_t - \phi_{\Gamma, \infty}}{t} \eta \, dX
\\
&\quad
 = \int_\Omega \frac{  \phi_{\Gamma_t, \infty} \circ T_t -
\phi_{\Gamma, \infty} \circ T_t }{t} \eta \, dX
 +  \int_\Omega \frac{ \phi_{\Gamma, \infty} \circ T_t }{t} \eta \, dX
 -  \int_\Omega \frac{ \phi_{\Gamma, \infty} \eta}{t} \, dX
\\
& \quad
= \int_\Omega h_t \left( \eta \circ T_t^{-1} \right) \det \nabla T_t^{-1} d x
+ \int_\Omega  \frac{ \phi_{\Gamma, \infty}}{t} \left( \eta  \circ T_t^{-1} \right)
 \det \nabla T_t^{-1} dx -  \int_\Omega \frac{ \phi_{\Gamma, \infty} \eta}{t} \, dx
\\
& \quad
= \int_\Omega h_t \left( \eta \circ T_t^{-1} \right) \det \nabla T_t^{-1} d x
\\
&\quad \quad
+ \int_\Omega \phi_{\Gamma, \infty}
\left( \frac{ \eta \circ T_t^{-1} - \eta}{t} \det \nabla T_t^{-1}
+ \eta \frac{ \det \nabla T_t^{-1}  - 1 }{t} \right) dx
\\
&\quad
\to - \int_\Omega \phi_{\Gamma, \infty} \nabla \eta \cdot V \, dx
 - \int_\Omega \phi_{\Gamma, \infty} \eta (\nabla \cdot  V) \, dx
\\
& \quad
= \int_\Omega (\nabla \phi_{\Gamma, \infty} \cdot V ) \eta \, dx
\qquad \mbox{as } t \to 0.
\end{align*}
This and \reff{limzetaG}, together with the arbitrariness of $\eta \in L^2(\Omega)$,
 imply that $\zeta_{\Gamma, V} = \nabla \phi_{\Gamma, \infty} \cdot V$ in $\Omega.$
\end{proof}

We recall that $\hat{\phi}_{\Gamma, \infty} $ is determined by \reff{phiGammaWeak}
and the boundary condition $\hat{\phi}_{\Gamma,\infty} = \phi_\infty$ on $\partial\Omega.$ 
For each $t \in \R$, we denote by $\hat{\phi}_{\Gamma_t,\infty}$
the unique function that is defined by \reff{phiGammaWeak}
with $\Gamma_t $ replacing $\Gamma$ and the same boundary condition 
$\hat{\phi}_{\Gamma_t,\infty} = \phi_\infty$ on $\partial \Omega.$ 
Note that all the singularities $x_i$ $(i = 1, \dots, N)$ are outside
the support of the vector field $V.$ 





\begin{lemma}
\label{l:xi}
\begin{compactenum}
\item[\rm (1)]
There exists a unique $\xi_{\Gamma, V} \in H^1_0(\Omega) $ such that 
\begin{align}
\label{xi_eqn}
&\int_\Omega  \varepsilon_\Gamma \nabla \xi_{\Gamma, V}  \cdot \nabla \eta \, dx   
= -\int_\Omega  \varepsilon_\Gamma A_V' (0) \nabla
\hat{\phi}_{\Gamma, \infty}    \cdot \nabla \eta \, d x 
\qquad \forall \eta \in H_0^1(\Omega).
\end{align}

\item[\rm (2)]
By modifying the value of $\xi_{\Gamma, V}$ on a set of zero Lebesgue measure, 
we have that  $\xi_{\Gamma, V}^{\rm s} \in H^2(\Omega_{\rm s}) \cap C^1(\Omega_{\rm s})$  
for ${\rm s} = -$ or $+$. Moreover, 
\begin{align}
\label{xi-}
&  \Delta \xi_{\Gamma, V} = - \nabla \cdot A_V'(0) \nabla  \hat{\phi}_{\Gamma,\infty}
= \Delta (\nabla \hat{\phi}_{\Gamma, \infty} \cdot V)
\qquad \mbox{in } \Omega_- \cup \Omega_+, 
\\
&
 \llbracket \varepsilon_\Gamma \partial_n \xi_{\Gamma, V} \rrbracket_\Gamma  
= - \llbracket \varepsilon_\Gamma  A'_V(0) \partial_n 
 \hat{\phi}_{\Gamma,\infty}  \rrbracket_\Gamma 
\qquad \mbox{on } \Gamma.  
\nonumber 
\end{align}

\item[\rm (3)]
We have 
\begin{align*}
&\lim_{t\to 0} \| \hat{\phi}_{\Gamma_t,\infty} \circ T_t 
- \hat{\phi}_{\Gamma,\infty} \|_{H^1(\Omega)} = 0,     
\\
&\lim_{t\to 0} \left\| 
\frac{  \hat{\phi}_{\Gamma_t,\infty} \circ T_t - \hat{\phi}_{\Gamma,\infty}}{t} 
- \xi_{\Gamma,V} \right\|_{H^1(\Omega)} = 0.    
\end{align*}
\end{compactenum}
\end{lemma}

\begin{proof}
The proof is the same as and simpler than that of the next lemma, Lemma~\ref{l:w_t}, 
as there is an extra term $B$ there, which can be set to $0$ here. 
The only exception is the second equality in \reff{xi-} which can be obtained by the calculations
same as in \reff{hidden}. 
\end{proof}

We recall that $\psi_\Gamma\in \hat{\phi}_{\rm C} + H^1(\Omega)\cap C(\overline{\Omega})
 = \hat{\phi}_{\Gamma, \infty} + H^1(\Omega)\cap C(\overline{\Omega}) $ 
is the unique weak solution to the boundary-value problem of the dielectric-boundary
PB equation \reff{PBE}; cf.\ Definition~\ref{d:PBsolution}. 
For each $t \in \R$, we denote by $\psi_{\Gamma_t} \in \hat{\phi}_{\rm C} 
+ H^1(\Omega)\cap C(\overline{\Omega})$ the unique
solution to the same boundary-value problem 
with $\Gamma_t$ replacing $\Gamma.$ 



\begin{lemma}
\label{l:w_t}
\begin{compactenum}
\item[\rm (1)]
There exists a unique $\omega_{\Gamma, V} \in H^1_0(\Omega) $ such that 
\begin{align}
\label{w_eqn}
&\int_\Omega \left[  \varepsilon_\Gamma \nabla \omega_{\Gamma, V}  \cdot \nabla \eta 
+ \chi_+ B''\left(\psi_\Gamma -\frac{\phi_{\Gamma, \infty}}{2} \right) 
\omega_{\Gamma, V}  \eta\, \right] dx 
\nonumber\\
&\quad  
= -\int_\Omega  \varepsilon_\Gamma A_V' (0) \nabla \psi_\Gamma \cdot \nabla \eta \, d x 
\nonumber \\
&\quad \quad 
-\int_{\Omega_+} \left[ 
(\nabla \cdot V) 
 B'\left( \psi_\Gamma -\frac{\phi_{\Gamma, \infty}}{2} \right) 
-\frac{\zeta_{\Gamma, V} }{2} 
 B'' \left( \psi_\Gamma -\frac{\phi_{\Gamma, \infty}}{2} \right) \right] \eta\,  d x 
 \qquad \forall \eta \in H_0^1(\Omega).
\end{align}

\item[\rm (2)]
By modifying the value of $\omega_{\Gamma, V}$ on a set of zero Lebesgue measure, 
we have that  $\omega_{\Gamma, V}^{\rm s}  
\in H^2(\Omega_{\rm s}) \cap C^1(\Omega_{\rm s})$  
for ${\rm s} = -$ or $+$. Moreover, 
\begin{align}
\label{omega-}
&  \Delta \omega_{\Gamma, V} = - \nabla \cdot 
A_V'(0) \nabla  \psi_\Gamma  = \Delta ( \nabla \psi_\Gamma \cdot V)  \qquad \mbox{in } \Omega_-, 
\\
\label{omega+}
&  
\ve_+ \Delta \omega_{\Gamma, V} - B''\left( \psi_\Gamma -\frac{\phi_{\Gamma, \infty}}{2} \right) 
\omega_{\Gamma, V} 
=  - \ve_+ \nabla \cdot A'_V(0)  \nabla \psi_\Gamma
\nonumber \\
& \quad 
+  (\nabla \cdot V) B'\left( \psi_\Gamma -\frac{\phi_{\Gamma, \infty}}{2} \right) 
- \frac{\zeta_{\Gamma, V} }{2} B'' \left( \psi_\Gamma -\frac{\phi_{\Gamma, \infty}}{2} \right) 
\qquad  \mbox{in } \Omega_+, 
\\
\label{omegaGamma}
& \llbracket \varepsilon_\Gamma \partial_n \omega_{\Gamma, V} \rrbracket_\Gamma  
= - \llbracket \varepsilon_\Gamma  A'_V(0) \partial_n  
\psi_\Gamma \rrbracket_\Gamma 
\qquad \mbox{on } \Gamma.  
\end{align}

\item[\rm (3)]
We have 
\begin{align}
\label{psi_t_convergence}
&\lim_{t\to 0} \|  \psi_{\Gamma_t} \circ T_t - \psi_\Gamma \|_{H^1(\Omega)} = 0,     
\\
\label{w_t_convergence}
&\lim_{t\to 0} \left\| 
\frac{ \psi_{\Gamma_t}  \circ T_t - \psi_\Gamma }{t}
  - \omega_{\Gamma,V} \right\|_{H^1(\Omega)} = 0.    
\end{align}
\end{compactenum}
\end{lemma}

\begin{proof}
(1) Since $B''> 0$, the support of $V$ does not contain any of the singularities $x_i$ 
$(i  = 1, \dots, N)$, and $\psi_\Gamma$ and $\phi_{\Gamma, \infty}$ are uniformly bounded on 
the union of the support of $V$ and $\Omega_+$ (cf.\ \reff{hatphiGibounds} and \reff{PBsolnbounds}), 
the existence and uniqueness of $\omega_{\Gamma,V} \in H_0^1(\Omega)$
that satisfies \reff{w_eqn} follows from the Lax--Milgram Lemma
\cite{EvansBook2010,GilbargTrudinger98}.

(2) Choosing $\eta \in C_{\rm c}^1(\Omega)$ in \reff{w_eqn} with $\mbox{supp}\,(\eta) 
\subset \Omega_-$ and applying the Divergence Theorem, we obtain 
the first equation of \reff{omega-} in a.e.\ $\Omega_-.$  
Since the right-hand side of this first equation is 
in $L^2(\Omega_-)\cap C(\Omega_-)$, it follows from the regularity theory
\cite{Lady_EllipticBook1968,GilbargTrudinger98}
that, with a possible modification of the value of $\omega_{\Gamma, V}$ on a set of 
zero Lebesgue measure, $\omega_{\Gamma, V}^- \in H^2(\Omega_-)\cap C^1(\Omega_-).$ 
Now, the first equation in \reff{omega-} holds for each point in $\Omega_-.$ The second
equation is similar to that in \reff{DwG} (cf.\ \reff{hidden}). 
By similar arguments, we obtain 
that $\omega_{\Gamma, V}^+ \in H^2(\Omega_+)\cap C^1(\Omega_+)$ and \reff{omega+}.  
By splitting each of those two integrals in \reff{w_eqn} that has the term $\nabla \eta$ into 
integrals over $\Omega_-$ and $\Omega_+$, respectively, using the Divergence 
Theorem, and using \reff{omega-} and \reff{omega+}, we obtain \reff{omegaGamma}.  

(3) Let $\hat{\phi}_{\rm C}$ be given as in \reff{psi0} and $t \in \R$. Denote 
\[
\psi_{\rm r} = \psi_{\Gamma} - \hat{\phi}_{\rm C} \qquad \mbox{and} \qquad 
\psi_{{\rm r}, t} = \psi_{\Gamma_t } - \hat{\phi}_{\rm C}.  
\]
We first prove \reff{psi_t_convergence}.  By \reff{phi00} (with $\hat{\phi} = \hat{\phi}_{\rm C}$)
in Lemma~\ref{l:phi0}, it suffices to prove that 
\begin{equation}
\label{hatphiC_conv}
\lim_{t\to 0} \| \psi_{{\rm r}, t} \circ T_t - \psi_{\rm r}  \|_{H^1(\Omega)}  = 0. 
\end{equation}

By Definition \ref{d:PBsolution} and \reff{Wpsi0} (cf.\ also \reff{equivweak}), we have 
\begin{align}
\label{shiftedweakPBE}
& \int_\Omega \left[ \ve_\Gamma \nabla \psi_r  \cdot \nabla \eta + 
\chi_+ B'\left( \psi_r + \hat{\phi}_{\rm C} - \frac{\phi_{\Gamma, \infty}}{2} \right) \eta \right] dx 
\nonumber \\
&\quad 
= -(\ve_+  - \ve_- ) \int_{\Omega_+} \nabla \hat{\phi}_{\rm C} \cdot \nabla \eta \, dx    
\qquad \forall \eta \in H^1_0(\Omega).    
\end{align}
Replacing $\Gamma$, $\Omega_+$, $\psi$, and $\eta$ in \reff{shiftedweakPBE} by 
$\Gamma_t = T_t(\Gamma)$, $T_t(\Omega_+)$, $\psi_{\Gamma_t}$, and $\eta=\eta\circ T_t^{-1},$ 
respectively, we obtain by the change of variable $x = T_t(X)$ and \reff{At} that
\begin{align}
\label{psi_ref_weak_t1}
&\int_{\Omega} \left[  \varepsilon_{\Gamma}A_V(t)\nabla(\psi_{r,t}\circ T_t) \cdot \nabla \eta 
+ \chi_{+} B' \left( \left( \psi_{{\rm r}, t}+\hat{\phi}_{\rm C} 
-\frac{\phi_{\Gamma_t, \infty}}{2} \right)
\circ T_t \right)J_t \, \eta \right] dX 
\nonumber\\
&\quad 
= - (\varepsilon_+ -\varepsilon_-)
\int_{\Omega_{+}} A_V(t)\nabla ( \hat{\phi}_{\rm C}  \circ T_t)\cdot \nabla \eta \, dX 
\qquad \forall \eta \in H_0^1(\Omega).
\end{align}
Subtracting \reff{shiftedweakPBE} 
from \reff{psi_ref_weak_t1} and rearranging terms, we get 
\begin{align}
\label{diff_phi_r}
&\int_\Omega  \varepsilon_\Gamma 
\left[ \nabla(\psi_{{\rm r},t}\circ T_t) -\nabla \psi_{\rm r} \right] \cdot \nabla \eta \, dX
\nonumber \\
& \quad 
= - \int_\Omega  \varepsilon_\Gamma 
[ A_V(t) - I ] \nabla (\psi_{{\rm r},t}\circ T_t) \cdot \nabla \eta \, dX
\nonumber \\
& \quad \quad 
-\int_{\Omega_+} B'\left( \left( \psi_{{\rm r}, t} + \hat{\phi}_{\rm C}   
-\frac{\phi_{\Gamma_t, \infty}}{2} \right) \circ T_t \right) ( J_t - 1) \eta\, dX   
\nonumber \\
&\quad \quad
-\int_{\Omega_+} \left[ B'\left( \left( \psi_{{\rm r}, t} + \hat{\phi}_{\rm C}   
-\frac{\phi_{\Gamma_t, \infty}}{2} \right) \circ T_t \right) 
- B'\left( \psi_{\rm r}+ \hat{\phi}_{\rm C}  -\frac{\phi_{\Gamma, \infty}}{2} \right) \right] \eta\, dX
\nonumber \\
&\quad \quad 
 - (\ve_+ - \ve_-) \int_{\Omega_+} 
[ \nabla ( \hat{\phi}_{\rm C} \circ T_t)  -\nabla \hat{\phi}_{\rm C} ] \cdot \nabla \eta \, dX
\nonumber \\
&\quad \quad 
 - (\ve_+ - \ve_-) \int_{\Omega_+} 
[ A_V(t)  - I ]  \nabla ( \hat{\phi}_{\rm C} \circ T_t)  \cdot \nabla \eta \, dX
\qquad \forall \eta \in H_0^1(\Omega). 
\end{align}
Setting $\eta = \psi_{{\rm r}, t} \circ T_t - \psi_{\rm r}$, we have 
by the uniform bound of all $\psi_{{\rm r}, t}$ and $\phi_{\Gamma_t, \infty}$
(cf.\ \reff{PBsolnbounds} and \reff{phiGammabounds}), 
the Mean-Value Theorem, and the convexity of $B$ that 
\begin{align}
\label{convexB}
& -\left[ B'\left( \left( \psi_{{\rm r}, t} + \hat{\phi}_{\rm C} 
-\frac{\phi_{\Gamma_t, \infty}}{2} \right) 
\circ T_t \right) - B'\left( \psi_{\rm r}+ \hat{\phi}_{\rm C} 
-\frac{\phi_{\Gamma, \infty}}{2} \right)\right] \eta 
\nonumber \\
&\quad 
= -  B''(\lambda_t) \left(  \psi_{{\rm r}, t} \circ T_t - \psi_{\rm r} 
+  \hat{\phi}_{\rm C} \circ T_t - \hat{\phi}_{\rm C}  
+ \frac12 \phi_{\Gamma_t, \infty} \circ T_t 
- \frac12 \phi_{\Gamma, \infty}\right) ( \psi_{{\rm r}, t} \circ T_t - \psi_{\rm r}) 
\nonumber \\
&\quad 
= -  B''(\lambda_t) (  \psi_{{\rm r}, t} \circ T_t - \psi_{\rm r} )^2 
\nonumber \\
&\quad \quad  
- B''(\lambda_t) \left( \hat{\phi}_{\rm C}  \circ T_t - \hat{\phi}_{\rm C}
 + \frac12 \phi_{\Gamma_t, \infty} \circ T_t 
- \frac12 \phi_{\Gamma, \infty} \right) ( \psi_{{\rm r}, t} \circ T_t - \psi_{\rm r})   
\nonumber \\
&\quad 
\le  C  | ( \hat{\phi}_{\rm C}  \circ T_t - \hat{\phi}_{\rm C} ) 
( \psi_{{\rm r}, t} \circ T_t - \psi_{\rm r})  |  
 + C \left|  ( \phi_{\Gamma_t, \infty} \circ T_t - \phi_{\Gamma, \infty}) 
( \psi_{{\rm r}, t} \circ T_t - \psi_{\rm r})  \right|,  
\end{align}
where $\lambda_t$ is in between 
$ \left( \psi_{{\rm r}, t} + \hat{\phi}_{\rm C}  
-{\phi_{\Gamma_t, \infty}}/{2} \right) \circ T_t $
and $\psi_{\rm r}+ \hat{\phi}_{\rm C}  -{\phi_{\Gamma, \infty}}/{2}$ at each point in $\Omega_+,$ 
and the constant $C > 0$ 
is independent of $t$ and $\Gamma$. 
Now, the combination of \reff{diff_phi_r} with $\eta = \psi_{{\rm r}, t} \circ T_t - \psi_{\rm r}$
 and \reff{convexB}, together with 
the uniform bounds for $\psi_{{\rm r}, t}$ and $\phi_{\Gamma_t, \infty},$ 
and the Cauchy--Schwarz and Poincar\'e inequalities, leads to 
\begin{align*}
\| \psi_{{\rm r} ,t} \circ T_t -\psi_{\rm r}  \|_{H^1(\Omega)} 
& \le  C  \| A_V(t) - I \|_{L^\infty(\Omega)}  \| \psi_{{\rm r},t}\circ T_t \|_{H^1(\Omega)} 
+  C \| J_t - 1 \|_{L^\infty(\Omega)}
\\
&\quad  + C  \| \hat{\phi}_{\rm C}  \circ T_t - \hat{\phi}_{\rm C}  \|_{H^1({\Omega_+})} 
+ C \left\| \phi_{\Gamma_t, \infty} \circ T_t - \phi_{\Gamma, \infty}  \right\|_{H^1(\Omega)}
\\
& \quad + C \left\| A_V(t) - I \right\|_{L^\infty(\Omega_+)} 
\| \hat{\phi}_{\rm C} \circ T_t \|_{H^1(\Omega_+)}. 
\end{align*}
Now the convergence \reff{hatphiC_conv} follows from \reff{Jexp}, \reff{HtX}, 
\reff{Aexp}, \reff{KtX}, the uniform bound of $\psi_{{\rm r}, t}$, Lemma
\ref{l:phi0} (with $\hat{\phi} = \hat{\phi}_{\rm C}$), and Lemma \ref{l:pGinfty}. 

We now prove \reff{w_t_convergence}. 
Let us denote $\hat{\omega}_{\Gamma, V} = \omega_{\Gamma, V} - \nabla \hat{\phi}_{\rm C} \cdot V.$ 
By Lemma~\ref{l:phi0} (cf.\ \reff{phi0t} with $\hat{\phi} = \hat{\phi}_{\rm C}$), we need only to prove that 
\begin{equation}
\label{Only_omega}
\lim_{t\to 0} \left\| \frac{  \psi_{{\rm r}, t}  \circ T_t - \psi_{\rm r}}{t} 
- \hat{\omega}_{\Gamma,V} \right\|_{H^1(\Omega)} = 0.
\end{equation}

We first note that the Divergence Theorem and the calculations in \reff{hidden} imply that 
\begin{align*}
& \int_\Omega [  A'_V(0) \nabla \hat{\phi}_{\rm C}  + \nabla ( \nabla \hat{\phi}_{\rm C} \cdot V )]
\cdot \nabla \eta \, dx
\\
& \quad 
= - \int_\Omega [ \nabla \cdot ( A'_V(0) \nabla \hat{\phi}_{\rm C})  + \Delta ( \nabla \hat{\phi}_{\rm C} \cdot V )]
\, \eta \, dx 
\\
&\quad  = 0 \qquad \forall \eta \in H_0^1(\Omega). 
\end{align*}
This allows us to rewrite \reff{w_eqn} into the following equation for $\hat{\omega}_{\Gamma, V}$:  
\begin{align}
\label{New_w_eqn}
&\int_\Omega \left[  \varepsilon_\Gamma \nabla \hat{\omega}_{\Gamma, V}  \cdot \nabla \eta
+ \chi_+ B''\left(\psi_{\rm r} +\hat{\phi}_{\rm C}  -\frac{ \phi_{\Gamma, \infty}}{2} \right)
\hat{\omega}_{\Gamma, V}  \eta\, \right] dX
\nonumber\\
&\quad
= -\int_\Omega  \varepsilon_\Gamma A_V' (0) \nabla (\psi_\Gamma - \hat{\phi}_{\rm C} )
\cdot \nabla \eta \, dX
\nonumber \\
&\quad \quad
-\int_{\Omega_+} \left[  B'\left( \psi_{\rm r} + \hat{\phi}_{\rm C} -\frac{\phi_{\Gamma, \infty}}{2} \right)
(\nabla \cdot V) 
\right. 
\nonumber \\
& \qquad \quad 
\left. 
 +   B'' \left( \psi_{\rm r} + \hat{\phi}_{\rm C}  -\frac{\phi_{\Gamma, \infty}}{2} \right)
\left( \nabla \hat{\phi}_{\rm C} \cdot V -\frac{\zeta_{\Gamma, V} }{2} \right) \right] \eta\,  dX
\nonumber\\
&\quad \quad
 -(\varepsilon_+ -\varepsilon_-) \int_{\Omega_+} [ A'_V (0) \nabla \hat{\phi}_{\rm C} +
\nabla (\nabla \hat{\phi}_{\rm C} \cdot V) ]
\cdot \nabla \eta \, dX \qquad \forall \eta \in H_0^1(\Omega).
\end{align}

Multiplying both sides of \reff{diff_phi_r} by $1/t$  and combining the resulting 
equation with \reff{New_w_eqn}, we obtain by rearranging terms that 
\begin{align}
\label{MVTMVT}
&\int_\Omega  \varepsilon_\Gamma \nabla  \left( \frac{  \psi_{{\rm r},t} \circ T_t - \psi_{\rm r} }{t}  
-\hat{\omega}_{\Gamma, V}  \right) \cdot \nabla \eta \, dX 
\nonumber \\
& \quad 
= - \int_\Omega  \varepsilon_\Gamma \left[  \left( \frac{A_V(t) - I}{t} \right) 
\nabla (\psi_{{\rm r},t}\circ T_t) - A'_V(0) \nabla \psi_{\rm r} \right]  \cdot \nabla \eta \, dX
\nonumber \\
&\quad \quad 
- \int_{\Omega_+} \left[ B'\left( \left( \psi_{{\rm r}, t} + \hat{\phi}_{\rm C} 
-\frac{\phi_{\Gamma_t, \infty}}{2} \right) \circ T_t \right) \left( \frac{J_t - 1}{t}  \right) 
\right.
\nonumber \\
&\quad \quad  
\quad 
\left. 
- B'\left( \psi_{\rm r}+ \hat{\phi}_{\rm C}  -\frac{\phi_{\Gamma, \infty}}{2} \right) 
(\nabla \cdot V) \right] \eta\, dX 
\nonumber \\
&\quad \quad   
-\int_{\Omega_+} \left\{  \frac{1}{t} 
\left[   B'\left( \left( \psi_{{\rm r}, t} + 
\hat{\phi}_{\rm C} -\frac{\phi_{\Gamma_t, \infty}}{2} \right) 
\circ T_t \right)   - B'\left( \psi_{\rm r}+ \hat{\phi}_{\rm C} 
-\frac{\phi_{\Gamma, \infty}}{2} \right) \right] 
\right.  
\nonumber \\
& \quad \quad 
\quad 
\left.   - B''\left( \psi_{\rm r} + \hat{\phi}_{\rm C} 
-\frac{\phi_{\Gamma, \infty}}{2} \right) 
 \left( \hat{\omega}_{\Gamma,V} + \nabla \hat{\phi}_{\rm C} 
\cdot V - \frac{\zeta_{\Gamma, V}}{2} \right) \right\} 
\eta \, dX 
\nonumber \\
& \quad \quad 
- (\ve_+ - \ve_-) \int_{\Omega_+} 
\left[ \left(  \frac{ A_V(t)  - I}{t} \right)  \nabla ( \hat{\phi}_{\rm C} \circ T_t)
- A_V'(0) \nabla \hat{\phi}_{\rm C} \right] \cdot \nabla  \eta \, dX
\nonumber \\
& \quad \quad 
- (\ve_+ - \ve_-) \int_{\Omega_+} 
\nabla \left(  \frac{ \hat{\phi}_{\rm C} \circ T_t - \hat{\phi}_{\rm C} }{t} 
- \nabla \hat{\phi}_{\rm C} \cdot V \right) 
\cdot \nabla \eta \, dX \qquad \forall \eta \in H_0^1(\Omega). 
\end{align}
Specifying $\eta = ( \psi_{{\rm r}, t} \circ T_t - \psi_{\rm r})/t 
- \hat{\omega}_{\Gamma, V} \in H^1_0(\Omega),$ we have by the fact that $B''> 0$, the 
Mean-Value Theorem, 
the uniform bound for all the functions $\psi_{{\rm r}, t}$, $\zeta_{\Gamma,V}$, and $\omega_{\Gamma,V}$ 
(cf.\ \reff{phiGammabounds}, \reff{hatphiGibounds}, \reff{PBsolnbounds})
that in $\Omega_+$ 
\begin{align*}
& -  \left\{ \frac{1}{t} 
\left[   B'\left( \left( \psi_{{\rm r}, t} + \hat{\phi}_{\rm C} 
-\frac{\phi_{\Gamma_t, \infty}}{2} \right) \circ T_t \right)   
- B'\left( \psi_{\rm r}+ \hat{\phi}_{\rm C}  -\frac{\phi_{\Gamma, \infty}}{2} \right) \right] 
\right. 
 \\
& \quad 
\left. 
- B''\left( \psi_{\rm r} + \hat{\phi}_{\rm C} -\frac{\phi_{\Gamma, \infty}}{2} \right) 
\left( \hat{\omega}_{\Gamma,V} + \nabla \hat{\phi}_{\rm C} \cdot V - \frac{\zeta_{\Gamma,V}}{2} \right) \right\} 
\left( \frac{  \psi_{{\rm r}, t} \circ T_t - \psi_{\rm r} }{t}  - \hat{\omega}_{\Gamma,V} \right)
\\
& \quad
= - B''\left( \xi_t \right) \left( \frac{ \psi_{{\rm r}, t} \circ  T_t  - \psi_{\rm r}}{t}  +
\frac{ \hat{\phi}_{\rm C} \circ T_t - \hat{\phi}_{\rm C} }{t} - \frac{ \phi_{\Gamma_t, \infty} \circ T_t -
\phi_{\Gamma, \infty}}{2t} \right) 
\left( \frac{  \psi_{{\rm r}, t} \circ T_t - \psi_{\rm r} }{t} -\hat{\omega}_{\Gamma,V} \right)
\\
& \quad \quad 
+ B''\left( \psi_{\rm r} + \hat{\phi}_{\rm C}  -\frac{\phi_{\Gamma, \infty}}{2} \right) 
\left( \hat{\omega}_{\Gamma,V} + \nabla \hat{\phi}_{\rm C} \cdot V - \frac{\zeta_{\Gamma,V}}{2} \right) 
\left( \frac{  \psi_{{\rm r}, t} \circ T_t - \psi_{\rm r} }{t} - \hat{\omega}_{\Gamma,V} \right)
\\
& \quad
= - B''\left( \xi_t \right) \left( \frac{ \psi_{{\rm r}, t} \circ  T_t  - \psi_{\rm r}}{t}
-\hat{\omega}_{\Gamma,V}\right)^2 
\\
& \quad \quad 
- B''\left( \xi_t \right) \left( \hat{\omega}_{\Gamma,V} +  
\frac{ \hat{\phi}_{\rm C} \circ T_t - \hat{\phi}_{\rm C} }{t} - \frac{ \phi_{\Gamma_t, \infty} \circ T_t -
\phi_{\Gamma, \infty}}{2t} \right) 
\left( \frac{  \psi_{{\rm r}, t} \circ T_t - \psi_{\rm r} }{t} -\hat{\omega}_{\Gamma,V} \right)
\\
& \quad \quad 
+ \left[ B''\left( \psi_{\rm r} + \hat{\phi}_{\rm C}  -\frac{\phi_{\Gamma, \infty}}{2} \right) 
- B''(\xi_t ) \right]
\left( \hat{\omega}_{\Gamma,V} + \nabla \hat{\phi}_{\rm C} \cdot V - \frac{\zeta_{\Gamma,V}}{2} \right) 
\left( \frac{  \psi_{{\rm r}, t} \circ T_t - \psi_{\rm r} }{t} - \hat{\omega}_{\Gamma,V} \right)
\\
& \quad \quad 
+ B''(\xi_t ) \left( \hat{\omega}_{\Gamma,V} + \nabla \hat{\phi}_{\rm C} 
\cdot V - \frac{\zeta_{\Gamma,V}}{2} \right) 
\left( \frac{  \psi_{{\rm r}, t} \circ T_t - \psi_{\rm r} }{t} - \hat{\omega}_{\Gamma, V} \right)
\\
& \quad 
\le - B''\left( \xi_t \right) \left( 
\frac{ \hat{\phi}_{\rm C} \circ T_t - \hat{\phi}_{\rm C} }{t} - \nabla \hat{ \phi}_{\rm C} \cdot V
- \frac{ \phi_{\Gamma_t, \infty} \circ T_t - \phi_{\Gamma, \infty}}{2t}
+ \frac{\zeta_{\Gamma, V}}{2} \right) 
\\
&\quad \quad 
\cdot  \left( \frac{  \psi_{{\rm r}, t} \circ T_t - \psi_{\rm r} }{t} -\hat{\omega}_{\Gamma, V} \right)
\\
& \quad \quad 
+  B'''(\sigma_t ) \left( \psi_{\rm r} + \hat{\phi}_{\rm C} 
 - \frac{\phi_{\Gamma, \infty}}{2} - \xi_t \right)
\left( \hat{\omega}_{\Gamma, V} + \nabla \hat{\phi}_{\rm C} 
\cdot V - \frac{\zeta_{\Gamma,V}}{2} \right) 
\left( \frac{  \psi_{{\rm r}, t} \circ T_t - \psi_{\rm r} }{t} - \hat{\omega}_{\Gamma, V} \right)
\\
& \quad 
\le  C \left( \left| \frac{ \hat{\phi}_{\rm C} \circ T_t - \hat{\phi}_{\rm C} } {t} 
- \nabla \hat{\phi}_{\rm C}  \cdot V\right|
+\left|  \frac{ \phi_{\Gamma_t, \infty} \circ T_t - \phi_{\Gamma, \infty}}{t}
- \zeta_{\Gamma, V} \right| \right) 
\left|  \frac{  \psi_{{\rm r}, t} \circ T_t - \psi_{\rm r} }{t} -\hat{\omega}_{\Gamma, V} \right| 
\\
& \quad \quad 
+  C \left(  \left| \psi_{{\rm r}, t} \circ T_t - \psi_{\rm r} \right|
+ \left| \hat{\phi}_{\rm C}  \circ T_t - \hat{\phi}_{\rm C}  \right| 
+  \left| \phi_{\Gamma_t, \infty} \circ T_t - \phi_{\Gamma, \infty} \right|\right)
\left| \frac{  \psi_{{\rm r}, t} \circ T_t - \psi_{\rm r} }{t} - \hat{\omega}_{\Gamma, V} \right|,
\end{align*}
where $\xi_t$ and $\sigma_t$ are in between 
$ ( \psi_{{\rm r}, t} + \hat{\phi}_{\rm C}  -{\phi_{\Gamma_t, \infty}}/{2} ) \circ T_t $
and $\psi_{\rm r}+ \hat{\phi}_{\rm C} -{\phi_{\Gamma, \infty}}/{2}$ at each point in $\Omega_+.$ 
Now, combining this inequality and the identity \reff{MVTMVT}
with $\eta = ( \psi_{{\rm r}, t} \circ T_t - \psi_{\rm r})/t 
- \hat{\omega}_{\Gamma, V} \in H^1_0(\Omega),$ 
we obtain by the Poincar\'e and Cauchy--Schwarz inequalities and rearranging terms that 
\begin{align}
\label{S1toS4}
& \left\| \frac{  \psi_{{\rm r},t}\circ T_t - \psi_{\rm r} }{t}  
-\hat{\omega}_{\Gamma, V}   \right\|^2_{H^1(\Omega)}  
\nonumber \\
& \quad 
\le C \int_\Omega   \left|   \left( \frac{A_V(t) - I}{t} \right) 
\nabla (\psi_{{\rm r},t}\circ T_t) - A'_V(0) \nabla \psi_{\rm r} \right|^2  dX
\nonumber \\
&\quad \quad 
+ C \int_{\Omega_+} \left|  B'\left( \left( \psi_{{\rm r}, t} + \hat{\phi}_{\rm C} 
-\frac{\phi_{\Gamma_t, \infty}}{2} \right) \circ T_t \right) \left( \frac{J_t - 1}{t}  \right) 
\right.
\nonumber \\
&\quad \quad \quad 
\left. 
- B'\left( \psi_{\rm r}+ \hat{\phi}_{\rm C} 
-\frac{\phi_{\Gamma, \infty}}{2} \right) (\nabla \cdot V) \right|^2 dX 
\nonumber \\
& \quad \quad 
+ C \biggl(  \left\| \frac{ \hat{\phi}_{\rm C} \circ T_t - \hat{\phi}_{\rm C} }{t} - \nabla 
\hat{\phi}_{\rm C}  \cdot V 
\right\|_{H^1(\Omega_+)}^2 
+  \left\|  \frac{ \phi_{\Gamma_t, \infty} \circ T_t - \phi_{\Gamma, \infty}}{t}
- \zeta_{\Gamma, V} \right\|_{L^2(\Omega_+)}^2  
\nonumber \\
&  \quad \quad 
+  \| \psi_{{\rm r}, t} \circ T_t - \psi_{\rm r} \|^2_{L^2(\Omega_+)}
+  \| \hat{\phi}_{\rm C}  \circ T_t - \hat{\phi}_{\rm C}  \|^2_{L^2(\Omega_+)}
+  \| \phi_{\Gamma_t, \infty} \circ T_t - \phi_{\Gamma, \infty}\|^2_{L^2(\Omega_+)} \biggr) 
\nonumber \\
&\quad  \quad 
+ C \int_{\Omega_+} \left|  \left(  \frac{ A_V(t)  - I}{t} \right)  
\nabla ( \hat{\phi}_{\rm C}  \circ T_t) - A_V'(0) \nabla \hat{\phi}_{\rm C}  \right|^2 dX
\nonumber \\
& \quad 
= C \left[ S_1(t) +S_2(t) +S_3(t) +S_4(t) \right]. 
\end{align}

It follows from \reff{Ap0}--\reff{KtX}, Lemma~\ref{l:phi0} (with $\hat{\phi} = \hat{\phi}_{\rm C}$), 
and \reff{psi_t_convergence} that 
\begin{align}
\label{S1t}
S_1(t)&= \int_\Omega   \left|   \left[ \frac{A_V(t) - I}{t} \right] 
\nabla (\psi_{{\rm r},t}\circ T_t) - A'_V(0) \nabla \psi_{\rm r} \right|^2  dX
\nonumber \\
& \le 2 \int_\Omega   \left|   \left[  \frac{A_V(t) - I}{t} - A_V'(0)  \right]  
\nabla (\psi_{{\rm r},t}\circ T_t) \right|^2  dX 
+  2 \int_\Omega   \left|   A'_V(0) \nabla (\psi_{{\rm r},t}\circ T_t -  \psi_{\rm r} ) \right|^2  dX
\nonumber \\ 
&\
\to 0 \qquad \mbox{as } t \to 0. 
\end{align}
By the uniform boundedness of $\psi_{{\rm r}, t}$ and $\phi_{\Gamma_t, \infty}$ 
(cf.\ \reff{phiGammabounds}, \reff{hatphiGibounds}, \reff{PBsolnbounds}) the Mean-Value Theorem, 
\reff{Jexp} and \reff{HtX}, Lemmas \ref{l:phi0} and \ref{l:pGinfty}, 
and \reff{psi_t_convergence}, we have 
\begin{align}
\label{S2t}
S_2(t) & = 
 \int_{\Omega_+} \left|  B'\left( \left( \psi_{{\rm r}, t} + \hat{\phi}_{\rm C}   
-\frac{\phi_{\Gamma_t, \infty}}{2} \right) \circ T_t \right) \left( \frac{J_t - 1}{t}  \right) 
- B'\left( \psi_{\rm r}+ \hat{\phi}_{\rm C} 
-\frac{\phi_{\Gamma, \infty}}{2} \right) (\nabla \cdot V) \right|^2 dX 
\nonumber \\
& \le 2  \int_{\Omega_+} \left|  B'\left( \left( \psi_{{\rm r}, t} + \hat{\phi}_{\rm C}  
-\frac{\phi_{\Gamma_t, \infty}}{2} \right) \circ T_t \right) 
\left( \frac{J_t - 1}{t}   - \nabla \cdot V  \right) \right|^2 dX 
\nonumber \\
& \quad 
 + 2 \int_{\Omega_+} \left|  B'\left( \left( \psi_{{\rm r}, t} + \hat{\phi}_{\rm C}  
-\frac{\phi_{\Gamma_t, \infty}}{2} \right) \circ T_t \right) 
- B'\left( \psi_{\rm r}+ \hat{\phi}_{\rm C} 
-\frac{\phi_{\Gamma, \infty}}{2} \right) \right|^2 
 |\nabla \cdot V |^2   dX 
\nonumber \\
&\le C \int_{\Omega_+} \left| \frac{J_t - 1}{t}   - \nabla \cdot V  \right|^2 dX 
\nonumber \\
& \quad 
+ C \int_{\Omega_+}  ( |  \psi_{{\rm r}, t} \circ T_t - \psi_{\rm r} |^2 
+ | \hat{\phi}_{\rm C} \circ T_t - \hat{\phi}_{\rm C} |^2 
+  | \phi_{\Gamma_t, \infty} \circ T_t - \phi_{\Gamma, \infty} |^2 )\, dX
\nonumber \\
& \to 0 \qquad \mbox{as } t \to 0.
\end{align}
By Lemma \ref{l:phi0} (with $\hat{\phi} = \hat{\phi}_{\rm C}$), 
Lemma~\ref{l:pGinfty}, and \reff{psi_t_convergence}, we have 
\begin{align}
\label{S3t}
S_3(t) & = 
  \left\| \frac{ \hat{\phi}_{\rm C} \circ T_t - \hat{\phi}_{\rm C}  }{t} 
- \nabla \hat{\phi}_{\rm C} \cdot V \right\|_{H^1(\Omega_+)}^2 
+  \left\|  \frac{ \phi_{\Gamma_t, \infty} \circ T_t - \phi_{\Gamma, \infty}}{t}
- \zeta_{\Gamma, V} \right\|_{L^2(\Omega_+)}  
\nonumber \\
& \quad 
+  \| \psi_{{\rm r}, t} \circ T_t - \psi_{\rm r} \|_{L^2(\Omega_+)}
+  \| \hat{\phi}_{\rm C} \circ T_t - \hat{\phi}_{\rm C}  \|^2_{L^2(\Omega_+)}
+  \| \phi_{\Gamma_t, \infty} \circ T_t - \phi_{\Gamma, \infty}\|^2_{L^2(\Omega_+)} 
\nonumber \\
& \quad  \to 0 \qquad \mbox{as } t \to 0.
\end{align}
It follows from \reff{Ap0}--\reff{KtX} and Lemma \ref{l:phi0} 
(with $\hat{\phi} = \hat{\phi}_{\rm C}$) that  
\begin{align}
\label{S4t}
S_4(t) & = 
\int_{\Omega_+} \left|  \left[   \frac{ A_V(t)  - I}{t} \right]   \nabla 
( \hat{\phi}_{\rm C} \circ T_t)
- A_V'(0) \nabla \hat{\phi}_{\rm C}  \right|^2 dX
\nonumber \\
& \le C  \int_{\Omega_+} \left|  \left[ \frac{ A_V(t)  - I}{t} - A'_V(0)  \right]   
\nabla (\hat{\phi}_{\rm C}  \circ T_t) \right|^2 + C \int_{\Omega_+} 
| \nabla ( \hat{\phi}_{\rm C}  \circ T_t - \hat{\phi}_{\rm C}  ) |^2  dX
\nonumber \\
& \to 0 \qquad \mbox{as } t \to 0.
\end{align}
Now the desired convergence \reff{Only_omega} follows from \reff{S1toS4}--\reff{S4t}. 
\end{proof}

\section{Proof of Theorem~\ref{t:DBF}}
\label{s:Proof}

\begin{proof}[Proof of Theorem~\ref{t:DBF}]
Fix $V \in {\cal V}$ (cf.\ \reff{calV}).  Let $\{ T_t\}_{t \in \R} $ be the family of diffeomorphisms
from $\R^3$ to $\R^3$ defined by $T_t(X) = x(t,X)$ as the solution to the 
initial-value problem \reff{flow}. 
We proceed in five steps. In Step 1, we calculate the limit
as $t\to 0$ that defines the variation $\delta_{\Gamma, V}E[\Gamma];$ 
cf.\ Definition~\ref{d:deltaGVG}. 
In Step 2, we simplify the expression of $\delta_{\Gamma, V}E[\Gamma].$ 
In Step 3, we convert all the volume integrals in $\delta_{\Gamma, V} E[\Gamma]$
into surface integrals on the boundary $\Gamma$, except one volume integral that involves the $B'$ term. 
In Step 4, we rewrite the surface integrals to have the desired form (i.e., with
a factor $V\cdot n$ in the integrand). 
Finally, in Step 5, we treat the only 
volume integral term that involves $B'$ to get the desired formula.

\medskip

{\it Step 1.}
Let $t \in \R.$ We recall that $\phi_{\Gamma_t, \infty}$, $\hat{\phi}_{\Gamma_t, \infty}$, 
and $\psi_{\Gamma_t}$ are the solutions to \reff{hatpsiinftyweak}, \reff{phiGammaWeak}, 
and \reff{weakPBE} with $\Gamma_t = T_t(\Gamma) $ replacing $\Gamma,$  respectively, 
and that all these functions have the boundary value $\phi_\infty$ on $\partial \Omega.$  
Recall that $\hat{\phi}_0$ and $\hat{\phi}_\infty$ are defined by \reff{Wpsi0}
and \reff{hatphiBC}. We denote in this proof
\begin{equation}
\label{4notations}
\psi_{\rm r} = \psi_{\Gamma} - \hat{\phi}_{\Gamma, \infty} \quad  
\mbox{and} \quad 
\psi_{{\rm r},t} = \psi_{\Gamma_t} - \hat{\phi}_{\Gamma_t, \infty}.
\end{equation}
By \reff{DefineJGamma} and \reff{newEGamma} with $\Gamma_t$ replacing $\Gamma$,
the definition of $A_V(t)$ \reff{At}
and $J_t$ \reff{Jt}, and the change of variable $x = T_t(X)$, we have
\begin{align*}
E[\Gamma_t] & = -\int_\Omega \frac{\ve_{\Gamma_t}}{2} |\nabla \psi_{{\rm r}, t} |^2 dx
-  \int_{ T_t(\Omega_+) } B \left( \psi_{\Gamma_t} - \frac{\phi_{\Gamma_t, \infty}}{2} \right) dx
\nonumber
\\
&\quad
+ \frac{\ve_- - \ve_+}{2} \int_{T_t(\Omega_+)}
 \nabla  \hat{\phi}_{\Gamma_t, \infty} \cdot \nabla  \hat{\phi}_0 \, dx + W
\\
& = - \int_{\Omega}  \frac{\ve_{\Gamma} }{2}\left[ A_V(t) \nabla ( \psi_{r, t} \circ T_t)
\cdot \nabla ( \psi_{r, t} \circ T_t ) \right] dX
\\
& \quad
- \int_{\Omega_+} B\left( \left(\psi_{\Gamma_t} -\frac{\phi_{\Gamma_t,\infty} }{2} \right)
\circ T_t\right) J_t \, dX
\\
&\quad
 + \frac{\ve_- - \ve_+}{2} \int_{\Omega_+} A_V(t)
 \nabla  (\hat{\phi}_{\Gamma_t, \infty} \circ T_t )  \cdot  \nabla  (\hat{\phi}_0
\circ T_t)  \, dX + W,
\end{align*}
where $W = (1/2)  \sum_{i=1}^N Q_i (\hat{\phi}_\infty - \hat{\phi}_0 ) (x_i)$ is
independent of $\Gamma.$ Consequently,
\begin{align}
\label{deltat}
\frac{ E[\Gamma_t]-E[\Gamma] }{t}
& = - \int_\Omega \frac{\varepsilon_{\Gamma} }{2t}
\left[ A_V(t) \nabla (\psi_{{\rm r},t} \circ T_t) \cdot \nabla (\psi_{{\rm r},t}\circ T_t)
- \nabla \psi_{\rm r} \cdot \nabla \psi_{\rm r} \right] dX
\nonumber \\
&\qquad
 -\int_{\Omega_+} \frac{1}{t}
\left[ B\left( \left(\psi_{\Gamma_t} -\frac{\phi_{\Gamma_t,\infty} }{2} \right) \circ T_t\right) J_t
- B\left( \psi_{\Gamma} -\frac{\phi_{\Gamma,\infty} }{2} \right) \right] dX
\nonumber \\
&\qquad
+\frac{\varepsilon_- -\varepsilon_+}{2} \int_{\Omega_{+}}
\frac{1}{t} \left[ A_V(t) \nabla ( \hat{\phi}_{\Gamma_t, \infty} \circ T_t ) \cdot
\nabla(\hat{\phi}_0 \circ T_t)
-  \nabla \hat{\phi}_{\Gamma, \infty}  \cdot \nabla \hat{\phi}_0 \right] \, dX
\nonumber
\\
&= -\delta_1(t) - \delta_2(t) + \frac{\ve_- - \ve_+}{2} \delta_3(t).
\end{align}

By rearranging the terms, we obtain that 
\begin{align*}
\delta_1(t) 
&= \int_\Omega \frac{\ve_\Gamma }{2} \left[ \frac{ A_V(t) - I - tA_V'(0)}{t}  \right] 
\nabla (\psi_{ {\rm r},t}\circ T_t) \cdot \nabla (\psi_{ {\rm r},t}\circ T_t) \, dX  
\nonumber\\
&\qquad +  \int_\Omega \frac{\ve_\Gamma }{2} A_V'(0) 
\nabla (\psi_{ {\rm r},t}\circ T_t) \cdot \nabla (\psi_{ {\rm r},t}\circ T_t) \, dX  
\nonumber\\
&\qquad 
+ \int_\Omega \frac{\varepsilon_{\Gamma}}{2}  
\left[ \nabla(\psi_{ {\rm r},t}\circ T_t)+\nabla \psi_{\rm r} \right]  \cdot \nabla 
\left( \frac{\psi_{{\rm r},t}\circ T_t-\psi_{\rm r}}{t} \right) dX. 
\end{align*}
It thus follows from \reff{Aexp}, \reff{KtX}, Lemma~\ref{l:xi}, and Lemma~\ref{l:w_t} that 
\begin{equation}
\label{limitdelta1}
\lim_{t\to 0} \delta_1(t) = 
\int_\Omega \varepsilon_{\Gamma}
\left[ \frac{1}{2} A_V'(0)\nabla \psi_{\rm r} \cdot \nabla \psi_{\rm r}
+ \nabla \psi_{\rm r} \cdot \nabla ( \omega_{\Gamma, V} - \xi_{\Gamma, V} ) \right] dX, 
\end{equation}
where $\xi_{\Gamma, V}$ and $\omega_{\Gamma, V}$ are defined in 
\reff{xi_eqn} in Lemma~\ref{l:xi} and \reff{w_eqn} in Lemma~\ref{l:w_t}, respectively.

Denote $q = \psi_{\Gamma} -\phi_{\Gamma,\infty}/ 2$ and 
$q_t =  (\psi_{\Gamma_t} - \phi_{\Gamma_t,\infty}/2 ) \circ T_t.$  
The second term $\delta_2(t)$ in \reff{deltat} can be written as  
\begin{align}
\label{d2t}
\delta_2(t) 
& = \int_{\Omega_+} \frac{J_t - 1}{t} B(q_t)\, dX 
+  \int_{\Omega_+} \frac{B(q_t) - B(q)}{t} \, dX, 
\end{align}
Since the $L^\infty(\Omega)$-norm of $q_t$ 
is bounded uniformly in $t \in \R $ (cf.\ \reff{hatphi_reg} and \reff{PBsolnbounds}), it follows
from  Lemma \ref{l:pGinfty} and Lemma \ref{l:w_t} that  $q_t \to q$ in $L^2(\Omega)$. 
Hence, $B(q_t) \to B(q)$ in $L^2(\Omega_+)$ as $t\to 0.$ 
This, together with \reff{Jexp} and \reff{HtX}, implies that 
\begin{equation}
\label{Jtone}
\lim_{t\to 0} \int_{\Omega_+}  \frac{J_t - 1}{t} B ( q_t)\, dX 
= \int_{\Omega_+} (\nabla \cdot V ) B (q) \, dX. 
\end{equation}
Now Taylor's expansion implies that 
\begin{align*}
&\frac{ B(q_t(X)) -  B(q(X))}{t}  
\\
&\quad 
= B'(q(X)) \frac{q_t(X)-q(X)}{t} + \frac12 B''(\eta_t (X)) [ q_t(X) - q(X)] \frac{q_t(X) - q(X)}{t},  
\quad \mbox{a.e. } X \in \Omega_+,  
\end{align*}
where $\eta_t (X)$ is in between $q(X)$ and $q_t(X),$ and its 
$L^\infty(\Omega)$-norm is bounded uniformly in $t.$ 
It then follows from Lemma~\ref{l:pGinfty} and Lemma~\ref{l:w_t} that 
\[
\left| \int_{\Omega_+} B''(\eta_t )  (q_t - q) \frac{q_t - q}{t} \, dX \right|
\le C \| q_t - q \|_{L^2(\Omega_+)}  
 \left\| \frac{q_t - q}{t} \right\|_{L^2(\Omega_+)} \to 0 \qquad \mbox{as } t \to 0,  
\]
where $C $ is a constant independent of $t$. 
Consequently, by Lemma~\ref{l:pGinfty} and Lemma~\ref{l:w_t} that 
\begin{align*}
 \lim_{t\to 0} \int_{\Omega_+} \frac{B(q_t) - B(q)}{t} \, dX 
= \lim_{t\to 0} \int_{\Omega_+} B'(q) \frac{q_t-q}{t} \, dX
= \int_{\Omega_+} B'(q) \left( \omega_{\Gamma, V}  - \frac{\zeta_{\Gamma,V}}{2} \right) dX, 
\end{align*}
where $\omega_{\Gamma, V}$ and $\zeta_{\Gamma, V}$ are given 
in \reff{w_eqn} and \reff{hat_w}, respectively. 
This, together with \reff{d2t} and \reff{Jtone}, and our definition of $q$ and $q_t$, implies that 
\begin{align}
\label{limitdelta2}
\lim_{t\to 0} \delta_2(t) 
 = \int_{\Omega_+} \left[ (\nabla \cdot V ) 
B \left( \psi_{\Gamma} - \frac{\phi_{\Gamma,\infty} }{2} \right) 
  + B'\left( \psi_{\Gamma}  -\frac{\phi_{\Gamma,\infty} }{2} \right)
\left( \omega_{\Gamma, V} - \frac{\zeta_{\Gamma, V}}{2} \right) \right] dX. 
\end{align}

Rearranging the terms, we have 
\begin{align*}
\delta_3(t) 
& = \int_{\Omega_+} \frac{A_V(t) - I - t A'_V(0)}{t} \nabla(\hat{\phi}_{\Gamma_t, \infty} \circ T_t) 
\cdot \nabla( \hat{\phi}_0\circ T_t) \, dX
\nonumber \\
& \quad 
+ \int_{\Omega_+} \frac{\nabla( \hat{\phi}_{\Gamma_t, \infty}  
\circ T_t)   -  \nabla  \hat{\phi}_{\Gamma,\infty}  }{t} \cdot 
\nabla(\hat{\phi}_0\circ T_t)\, dx
\nonumber \\
& \quad 
+ \int_{\Omega_+} \nabla \hat{\phi}_{\Gamma, \infty} \cdot 
\frac{\nabla( \hat{\phi}_0\circ T_t)   - \nabla \hat{\phi}_0 }{t} \, dX 
\nonumber \\
& \quad 
+ \int_{\Omega_+}
A_V'(0)  \nabla( \hat{\phi}_{\Gamma_t, \infty} \circ T_t) \cdot \nabla( \hat{\phi}_0\circ T_t) \, dX. 
\end{align*}
Therefore, we have by \reff{Aexp}, \reff{KtX}, Lemma~\ref{l:phi0} (with $\hat{\phi} 
= \hat{\phi}_{\rm C}$), and Lemma~\ref{l:xi} that 
\begin{align}
\label{limitdelta3}
\lim_{t\to 0} \delta_3(t) 
&= \int_{\Omega_{+}} \left[ \nabla  \xi_{\Gamma, V} \cdot \nabla \hat{\phi}_0 
+ \nabla \hat{\phi}_{\Gamma, \infty} \cdot \nabla  ( \nabla \hat{\phi}_0 \cdot V)
+ A'_V(0) \nabla \hat{\phi}_{\Gamma,\infty} \cdot \nabla \hat{\phi}_0 \right] dX.  
\end{align}

It now follows from Definition~\ref{d:deltaGVG}, \reff{deltat},
\reff{limitdelta1}, \reff{limitdelta2}, and \reff{limitdelta3} that
the first variation $\delta_\Gamma E[\Gamma]$ exists and is given by
\begin{align}
\label{afterdelta}
\delta_{\Gamma, V} E[\Gamma]
& = - \int_\Omega \frac{\varepsilon_{\Gamma}}{2} 
 A_V'(0)\nabla \psi_{\rm r} \cdot \nabla \psi_{\rm r}\, dX 
+\int_\Omega \ve_\Gamma \nabla \psi_{\rm r} \cdot \nabla \xi_{\Gamma, V} \, dX
- \int_\Omega \ve_\Gamma \nabla \psi_{\rm r} \cdot \nabla  \omega_{\Gamma, V}\, dX
\nonumber \\
&\quad  - \int_{\Omega_+} \left[ (\nabla \cdot V )
B \left( \psi_{\Gamma} - \frac{\phi_{\Gamma,\infty} }{2} \right)
+ B'\left( \psi_{\Gamma}  -\frac{\phi_{\Gamma,\infty} }{2} \right)
\left( \omega_{\Gamma, V} 
- \frac{\zeta_{\Gamma, V}}{2} \right) \right] dX
\nonumber \\
&\quad + \frac{\ve_- - \ve_+}{2} \int_{\Omega_{+}} 
\nabla  \xi_{\Gamma, V} \cdot \nabla \hat{\phi}_0 \, dX
\nonumber \\
&\quad + \frac{\ve_- - \ve_+}{2} \int_{\Omega_{+}} 
\left[ \nabla \hat{\phi}_{\Gamma, \infty} \cdot \nabla  ( \nabla \hat{\phi}_0 \cdot V)
+ A'_V(0) \nabla \hat{\phi}_{\Gamma,\infty} \cdot \nabla \hat{\phi}_0 \right]\,dX
\nonumber 
\\
&= M_1 + M_2+M_3 +M_4 + M_5 + M_6.
\end{align}

{\it Step 2.} 
We now simplify this expression.  By  Lemma~\ref{l:xi} and  
our notation $\psi_{\rm r} = \psi_\Gamma - \hat{\phi}_{\Gamma, \infty}$,  we can 
express the sum of the first two integrals above as
\begin{align}
\label{M1M2} 
M_1 + M_2  
& = - \int_\Omega \frac{\varepsilon_{\Gamma}}{2} 
 A_V'(0) \nabla ( \psi_{\Gamma } - \hat{\phi}_{\Gamma, \infty} ) \cdot 
\nabla ( \psi_{\Gamma } - \hat{\phi}_{\Gamma, \infty} )\, dX 
\nonumber 
\\
&\quad \quad
- \int_\Omega \ve_\Gamma A_V'(0) \nabla \hat{\phi}_{\Gamma, \infty} \cdot 
\nabla ( \psi_\Gamma - \hat{\phi}_{\Gamma, \infty} ) \, dX
\nonumber 
\\
&\quad 
= - \int_\Omega \frac{\varepsilon_{\Gamma}}{2} 
 A_V'(0)\nabla \psi_{\Gamma} \cdot \nabla \psi_{\Gamma }\, dX 
+ \int_\Omega \frac{\varepsilon_{\Gamma}}{2} 
 A_V'(0)\nabla \hat{\phi}_{\Gamma, \infty} \cdot \nabla \hat{\phi}_{\Gamma,\infty }\, dX.  
\end{align}
Note that the last two integrals exist as the singularities $x_i$ $(1 \le i \le N)$
of $\psi_{\Gamma}$ and $\hat{\phi}_{\Gamma, \infty}$ are 
outside the support of $V$ and $A'_V(0)$ is given in \reff{Ap0}. 
By \reff{weakPBE} in Definition \ref{d:PBsolution} and \reff{phiGammaWeak}
 we have
\begin{equation*}
 \int_\Omega \left[ \ve_\Gamma \nabla \psi_r  \cdot \nabla \eta +
\chi_+ B'\left( \psi_\Gamma - \frac{\phi_{\Gamma, \infty}}{2} \right) \eta \right] dX
= 0
\end{equation*}
for all $\eta \in C_{\rm c}^1(\Omega)$ and hence all $\eta \in H_0^1(\Omega)$.
Setting $\eta = \omega_{\Gamma, V}$, we get the two-$\omega_{\Gamma, V}$ terms in 
\reff{afterdelta} (one is $M_3$ and the other is part of $M_4$) cancelled: 
\begin{align}
\label{M3M4}
& \int_{\Omega} \left[ \ve_\Gamma \nabla \psi_{\rm r} \cdot \nabla \omega_{\Gamma, V} 
+ \chi_+ B'\left(\psi_\Gamma - \frac{\phi_{\Gamma, \infty}}{2} \right) \omega_{\Gamma, V} \right] dX = 0. 
\end{align}

To simplify $M_5$, we note that we can replace $\eta$ in \reff{Wpsi0}
(with $\hat{\phi} = \hat{\phi}_0$) 
and \reff{phiGammaWeak} by $\xi_{\Gamma,V} \in H_0^1(\Omega)$, as 
 $\xi_{\Gamma, V}|_{\Omega_-} \in C^2(\Omega_-)$; cf.\ the remark below \reff{hatphiBC}
and that below \reff{equivweak}. 
It then follows that 
\begin{align}
\label{M5}
M_5 
& = \frac{\ve_- }{2} \int_{\Omega_+} \nabla \xi_{\Gamma, V} \cdot \nabla \hat{\phi}_0\, dX 
-  \frac{\ve_+}{2} \int_{\Omega_+} \nabla \xi_{\Gamma, V} \cdot \nabla \hat{\phi}_0\, dX 
\nonumber 
\\
&= -\int_\Omega \frac{\ve_\Gamma}{2} \nabla \xi_{\Gamma, V} \cdot \nabla \hat{\phi}_0 \, dX 
+ \frac12 \sum_{i=1}^N Q_i \xi_{\Gamma, V}(x_i)    
\quad  [\mbox{by \reff{Wpsi0} with $\hat{\phi}=\hat{\phi}_0$}] 
\nonumber 
\\
&= -\int_\Omega \frac{\ve_\Gamma}{2} \nabla \xi_{\Gamma, V} \cdot \nabla \hat{\phi}_0 \, dX 
+ \int_\Omega \frac{\ve_\Gamma}{2} \nabla \hat{\phi}_{\Gamma, \infty} \cdot
\nabla \xi_{\Gamma, V}\, dX  
\quad  [\mbox{by \reff{phiGammaWeak}}]
\nonumber 
\\
&= \int_\Omega \frac{\ve_\Gamma}{2} \nabla \xi_{\Gamma, V} \cdot 
\nabla (\hat{\phi}_{\Gamma, \infty} -  \hat{\phi}_0 - \phi_{\Gamma, \infty})   \, dX 
\quad  [\mbox{by \reff{hatpsiinftyweak}}]
\nonumber 
\\
&= - \int_\Omega \frac{\ve_\Gamma}{2} A'_V(0) 
\nabla \hat{\phi}_{\Gamma, \infty} \cdot 
\nabla (\hat{\phi}_{\Gamma, \infty} -  \hat{\phi}_0 - \phi_{\Gamma, \infty})   \, dX.  
\quad  [\mbox{by Lemma~\ref{l:xi}}]
\end{align}

Since $\hat{\phi}_0$ is harmonic in the support of $V$ that excludes all 
$x_i $ $ (i=1, \dots, N)$, we have by the same calculations as in 
\reff{hidden} that 
\[
 \nabla \cdot \left[ \nabla ( \nabla \hat{\phi}_0 \cdot V)
+ A'_V(0) \nabla \hat{\phi}_0 \right] = 0 \quad \mbox{in } \Omega. 
\]
Thus, since the normal $n$ along $\Gamma$ points from $\Omega_-$ to $\Omega_+,$
we have by the Divergence Theorem that 
\begin{align*}
&\frac{ \ve_-}{2}  \int_{\Omega_+} \nabla \hat{\phi}_{\Gamma, \infty} 
\cdot  [ \nabla ( \nabla \hat{\phi}_0 \cdot V) + A'_V(0) \nabla \hat{\phi}_0 ] \, dX 
\\
&\quad 
 = -\frac{ \ve_-}{2}  \int_{\Gamma} \hat{\phi}_{\Gamma, \infty} 
[ \nabla ( \nabla \hat{\phi}_0 \cdot V) + A'_V(0) \nabla \hat{\phi}_0 ] 
\cdot n \, dS  
\\
& \quad 
= -\frac{ \ve_-}{2} \int_{\Omega_-} \nabla \hat{\phi}_{\Gamma, \infty} 
\cdot  [ \nabla ( \nabla \hat{\phi}_0 \cdot V) + A'_V(0) \nabla \hat{\phi}_0 ] \, dX.  
\end{align*}
Therefore, since $A_V'(0)$ (cf.\ \reff{Ap0}) is symmetric, 
\begin{align}
\label{M6}
M_6 & = \frac{\ve_- - \ve_+}{2} \int_{\Omega_{+}} 
 \nabla \hat{\phi}_{\Gamma, \infty} \cdot [ \nabla  ( \nabla \hat{\phi}_0 \cdot V)
+ A'_V(0) \nabla \hat{\phi}_0 ]\,dX
\nonumber 
\\
& = - \int_\Omega \frac{\ve_\Gamma}{2} 
[ \nabla \hat{\phi}_{\Gamma, \infty} \cdot \nabla  ( \nabla \hat{\phi}_0 \cdot V)
+ A'_V(0) \nabla \hat{\phi}_{\Gamma,\infty} \cdot \nabla \hat{\phi}_0 ]\,dX
\nonumber 
\\
& = - \int_\Omega \frac{\ve_\Gamma}{2} 
 A'_V(0) \nabla \hat{\phi}_{\Gamma,\infty} \cdot \nabla \hat{\phi}_0 \,dX. 
\quad  [\mbox{by \reff{phiGammaWeak}}]
\end{align}

It now follows from \reff{afterdelta}--\reff{M6} that 
\begin{align}
\label{P1-P4}
\delta_{\Gamma, V} E[\Gamma]
&= - \int_\Omega \frac{\varepsilon_{\Gamma}}{2} 
 A_V'(0)\nabla \psi_{\Gamma} \cdot \nabla \psi_{\Gamma }\, dX 
+ \int_\Omega \frac{\varepsilon_{\Gamma}}{2} 
 A_V'(0)\nabla \hat{\phi}_{\Gamma, \infty} \cdot \nabla {\phi}_{\Gamma,\infty }\, dX  
\nonumber \\
&\quad
+ \int_{\Omega_+} \left[  \frac{\zeta_{\Gamma,V}}{2}
B'\left( \psi_\Gamma -\frac{\phi_{\Gamma,\infty} }{2} \right)
- (\nabla \cdot V ) B \left( \psi_\Gamma - \frac{\phi_{\Gamma,\infty} }{2} \right) \right] dX
\nonumber
\\
& = P_1 + P_2 + P_3. 
\end{align}

{\it Step 3.}
We convert most of these volume integrals into surface integrals on $\Gamma.$ 
We shall use the following identities that can be verified by using
the Divergence Theorem and approximations by smooth functions: 
\begin{align}
\label{Uab1}
& \int_D (\nabla \cdot U ) \nabla a \cdot \nabla b \, dx 
= - \int_D U \cdot  (\nabla^2 a  \nabla b + \nabla^2 b \nabla a) \, dx
+ \int_{\partial D} ( \nabla a \cdot \nabla b) (U \cdot \nu) \, dx, 
\\
\label{Uab2}
& \int_D (\nabla U ) \nabla a \cdot \nabla b \, dx 
= - \int_D  U \cdot ( \Delta a \nabla b + \nabla^2 b \nabla a ) \, dx
+ \int_{\partial D} ( \nabla a \cdot \nu ) (\nabla b \cdot U) \, dx.  
\end{align}
Here, $D \subset \R^3$ is a bounded open set with a $C^1$ boundary $\partial D,$
$U \in H^1(D, \R^3)$, $a,b\in H^2(D),$
$\nabla^2 a$ is the Hessian matrix of $a$, and 
$\nu$ is the unit exterior normal at the boundary $\partial D.$
If in addition $\Delta a = \Delta b = 0$ in $D$, then 
we have by \reff{Uab1} and \reff{Uab2} that 
\begin{align}
\label{UabMore}
&\int_D (\nabla U + (\nabla U)^T - (\nabla \cdot U ) I ) \nabla a \cdot \nabla b \, dx 
\nonumber 
\\
& \quad 
= \int_{\partial D} [ ( \nabla a \cdot U) \cdot (\nabla b \cdot \nu)  
+ ( \nabla b \cdot U) \cdot (\nabla a \cdot \nu) 
- ( \nabla a \cdot \nabla b ) \cdot ( U \cdot \nu)] \, dS. 
\end{align}

Note that $V = 0$ in a neighborhood of all $x_i$ $(1 \le i \le N)$ and $V = 0$ on $\partial \Omega$ 
and that the unit normal vector $n$ on $\Gamma$ points from $\Omega_-$ to $\Omega_+.$ 
By Theorem~\ref{t:DBPBE}, $\Delta \psi_\Gamma  = 0$ 
on $\Omega_- \cap \mbox{supp}\,(V) $ and $\ve_+ \Delta \psi_\Gamma = 
B'(\psi_\Gamma - \phi_{\Gamma, \infty}/2)$ on $\Omega_+.$ 
Therefore, we have by \reff{Ap0}, \reff{Uab1}, and \reff{Uab2} that 
\begin{align}
\label{TermP1}
P_1 
& =  \int_\Omega \frac{\ve_\Gamma}{2} 
[\nabla V + (\nabla V)^T - (\nabla \cdot V)I ]\nabla \psi_\Gamma \cdot \nabla \psi_\Gamma \, dX
\nonumber 
\\
&=  \int_{\Omega_-}  \ve_-  (\nabla V) \nabla \psi_\Gamma \cdot \nabla \psi_\Gamma \, dX 
+\int_{\Omega_+}  \ve_+ (\nabla V) \nabla \psi_\Gamma  \cdot \nabla \psi_\Gamma \, dX
\nonumber 
\\
&\quad 
- \int_{\Omega_-} 
 \frac{\ve_-}{2} (\nabla \cdot V) \nabla \psi_\Gamma \cdot \nabla \psi_\Gamma \,  dX
- \int_{\Omega_+} 
\frac{\ve_+}{2} (\nabla \cdot V) \nabla \psi_\Gamma \cdot \nabla \psi_\Gamma \, dX
\nonumber 
\\
& = - \int_{\Omega_-} \ve_-  V \cdot ( \Delta \psi_\Gamma \nabla \psi_\Gamma 
+ \nabla^2 \psi_\Gamma  \nabla \psi_\Gamma )  \, dX
+ \int_\Gamma \ve_- (\nabla \psi^-_\Gamma \cdot V) (\nabla \psi^-_\Gamma  \cdot n) \, dS
\nonumber 
\\
&\quad 
- \int_{\Omega_+} \ve_+  V \cdot ( \Delta \psi_\Gamma \nabla \psi_\Gamma
+ \nabla^2 \psi_\Gamma \nabla \psi_\Gamma ) \, dX
- \int_\Gamma \ve_+ (\nabla \psi^+_\Gamma \cdot V) (\psi^+_\Gamma \cdot n) \, dS
\nonumber 
\\
&\quad 
+ \int_{\Omega_-} \ve_-   V \cdot  \nabla^2 \psi_\Gamma \nabla \psi_\Gamma \, dX
- \int_\Gamma \frac{\ve_-}{2}  | \nabla \psi^-_\Gamma |^2 ( V \cdot n) \, dS
\nonumber 
\\
&\quad 
+ \int_{\Omega_+} \ve_+   V \cdot  \nabla^2 \psi_\Gamma \nabla \psi_\Gamma \, dX
+ \int_\Gamma \frac{\ve_+}{2}  | \nabla \psi^+_\Gamma |^2 ( V \cdot n) \, dS  
\nonumber 
\\
&= - \int_{\Omega_-}  \ve_- \Delta  \psi_\Gamma ( \nabla \psi_\Gamma \cdot V) \, dX
+ \int_\Gamma \ve_- (\nabla \psi^-_\Gamma \cdot V) (\nabla \psi^-_\Gamma \cdot n) \, dS
- \int_\Gamma \frac{\ve_-}{2} | \nabla \psi^-_\Gamma  |^2 ( V\cdot n)\, dS
\nonumber 
\\
&\quad 
- \int_{\Omega_+}  \ve_+  \Delta  \psi_\Gamma ( \nabla \psi_\Gamma \cdot V) \, dX
- \int_\Gamma \ve_+ (\nabla \psi^+_\Gamma \cdot V) (\nabla \psi^+_\Gamma \cdot n) \, dS
+ \int_\Gamma \frac{\ve_+}{2} | \nabla \psi^+_\Gamma |^2 ( V\cdot n)\, dS
\nonumber 
\\
& = \int_\Gamma \ve_- (\nabla \psi^-_\Gamma \cdot V) (\nabla \psi^-_\Gamma \cdot n) \, dS
 - \int_\Gamma \ve_+ (\nabla \psi^+_\Gamma \cdot V) (\nabla \psi^+_\Gamma \cdot n) \, dS
\nonumber \\
&\quad 
- \int_\Gamma \frac{\ve_-}{2} | \nabla \psi^-_\Gamma |^2 ( V\cdot n)\, dS
+ \int_\Gamma \frac{\ve_+}{2} | \nabla \psi^+_\Gamma |^2 ( V\cdot n)\, dS
\nonumber \\
&\quad 
- \int_{\Omega_+} B'\left(\psi_\Gamma - \frac{ \phi_{\Gamma, \infty}}{2} \right) 
(\nabla \psi_\Gamma  \cdot V)\, dX,  
\end{align}
where a superscript $-$ or $+$ denotes the restriction from $\Omega_-$ or $\Omega_+$, respectively.

Since $\hat{\phi}_{\Gamma, \infty}$ and $\phi_{\Gamma, \infty}$ are harmonic 
in $\Omega_- \cap \mbox{supp}\,(V) $ and $\Omega_+$, and since the normal 
$n$ points from $\Omega_-$ to $\Omega_+$, we have by \reff{Ap0}, 
\reff{UabMore}, and the notation of jumps \reff{JumpSign}  that 
\begin{align}
\label{TermP2}
P_2 &= 
\int_{\Omega_-} \frac{\ve_-}{2}
[ (\nabla \cdot V) I - \nabla V - (\nabla V)^T]
\nabla \hat{\phi}_{\Gamma, \infty} \cdot \nabla \phi_{\Gamma, \infty} \, dX
\nonumber 
\\
&\quad 
+ \int_{\Omega_+} \frac{\ve_+ }{2}
[ (\nabla \cdot V) I - \nabla V - (\nabla V)^T]
\nabla \hat{\phi}_{\Gamma, \infty} \cdot \nabla \phi_{\Gamma, \infty} \, dX
\nonumber 
\\
& =  \frac12 \int_\Gamma  \llbracket  \ve_\Gamma   
(\nabla \hat{\phi}_{\Gamma, \infty} \cdot V) (\nabla {\phi}_{\Gamma, \infty} \cdot n) 
+ \ve_\Gamma (\nabla \hat{\phi}_{\Gamma, \infty} \cdot n) (\nabla {\phi}_{\Gamma, \infty} \cdot V) 
\nonumber \\
&\quad 
- \ve_\Gamma (\nabla \hat{\phi}_{\Gamma, \infty} \cdot 
\nabla {\phi}_{\Gamma, \infty}) ( V\cdot n)  \rrbracket_{\Gamma}\, dS. 
\end{align}

Using the Divergence Theorem and noting again that the normal $n$ at $\Gamma$ 
 points from $\Omega_-$ to $\Omega_+$, we obtain
\begin{align}
\label{TermP3}
P_3 
& = \int_{\Omega_+} \frac{\zeta_{\Gamma,V}}{2} 
B'\left( \psi_\Gamma -\frac{\phi_{\Gamma,\infty} }{2} \right) dX 
+ \int_{\Omega_+} V\cdot B'\left( \psi_\Gamma  -\frac{\phi_{\Gamma,\infty} }{2} \right)
\left( \nabla \psi_\Gamma -\frac{\nabla \phi_{\Gamma,\infty} }{2} \right) dX 
\nonumber \\
&\qquad 
+ \int_\Gamma B \left( \psi_\Gamma  -\frac{\phi_{\Gamma,\infty} }{2} \right) (V\cdot n)\, dS
\nonumber \\
& = \int_{\Omega_+} 
\left[ \frac{1}{2}\left(  \zeta_{\Gamma,V} - \nabla \phi_{\Gamma, \infty} \cdot V \right)
+  \nabla \psi_\Gamma \cdot V \right] B'\left( \psi_\Gamma -\frac{\phi_{\Gamma,\infty} }{2} \right) dX 
\nonumber \\
&\qquad 
+ \int_\Gamma B \left( \psi_\Gamma -\frac{\phi_{\Gamma,\infty} }{2} \right) (V\cdot n)\, dS. 
\end{align}

It now follows from \reff{P1-P4} and \reff{TermP1}--\reff{TermP3} that
\begin{align}
\label{dGVE_2}
\delta_{\Gamma, V} E[\Gamma] 
& = \int_\Gamma \ve_- (\nabla \psi^-_\Gamma \cdot V) (\nabla \psi^-_\Gamma \cdot n) \, dS
 - \int_\Gamma \ve_+ (\nabla \psi^+_\Gamma \cdot V) (\nabla \psi^+_\Gamma \cdot n) \, dS
\nonumber \\
&\quad 
- \int_\Gamma \frac{\ve_-}{2} | \nabla \psi^-_\Gamma |^2 ( V\cdot n)\, dS
+ \int_\Gamma \frac{\ve_+}{2} | \nabla \psi^+_\Gamma |^2 ( V\cdot n)\, dS
\nonumber \\
&\quad 
  +  \frac12 \int_\Gamma  \llbracket  \ve_\Gamma   
(\nabla \hat{\phi}_{\Gamma, \infty} \cdot V) (\nabla {\phi}_{\Gamma, \infty} \cdot n) 
\rrbracket_\Gamma \, dS
\nonumber \\
&\quad 
+ \frac12 \int_\Gamma \llbracket \ve_\Gamma (\nabla \hat{\phi}_{\Gamma, \infty} \cdot n) 
(\nabla {\phi}_{\Gamma, \infty} \cdot V) \rrbracket_\Gamma \, dS 
\nonumber \\
&\quad 
- \frac12 \int_\Gamma \llbracket \ve_\Gamma  
(\nabla \hat{\phi}_{\Gamma, \infty} \cdot \nabla {\phi}_{\Gamma, \infty}) ( V\cdot n) 
\rrbracket_{\Gamma}\, dS 
\nonumber \\
&\quad 
+ \int_{\Omega_+} \frac{1}{2}\left(  \zeta_{\Gamma,V} - \nabla 
\phi_{\Gamma, \infty} \cdot V \right)
 B'\left( \psi_\Gamma -\frac{\phi_{\Gamma,\infty} }{2} \right) dX 
\nonumber \\
&\quad 
+ \int_\Gamma B \left( \psi_\Gamma -\frac{\phi_{\Gamma,\infty} }{2} \right) (V\cdot n)\, dS. 
\end{align}

\medskip

{\it Step 4.} We express the surface integrals into those with the factor $V\cdot n$ in
the integrand.  Note that on each side of $\Gamma$, we can write 
\[
\nabla \psi_\Gamma = (\nabla \psi_\Gamma \cdot n) n + \nabla_\Gamma \psi_\Gamma
= \partial_n \psi_\Gamma  n + \nabla_\Gamma \psi_\Gamma \quad \mbox{on } \Gamma,  
\]
where $\nabla_\Gamma \psi_\Gamma = (I - n \otimes n) \nabla \psi_\Gamma$ 
is the tangential derivative. Clearly $n \cdot \nabla_\Gamma \psi_\Gamma = 0.$  
Moreover, $\nabla_\Gamma \psi_\Gamma^+ = \nabla_\Gamma \psi_\Gamma^-$ on $\Gamma$. 
Thus, 
\[
 \nabla \psi_\Gamma^+ - \nabla \psi_\Gamma ^- 
= (\partial_n \psi_\Gamma^+- \partial_n \psi_\Gamma^- ) n \qquad \mbox{on } \Gamma. 
\]
By Theorem~\ref{t:DBPBE}, we have also $ \ve_+ \nabla \psi_\Gamma^+ \cdot n = 
\ve_- \nabla \psi_\Gamma^- \cdot n = \ve_\Gamma \nabla \psi_\Gamma \cdot n $ on $\Gamma$. 
Therefore, the first four terms in \reff{dGVE_2} are 
\begin{align}
\label{1234}
&  \int_\Gamma \ve_- (\nabla \psi^-_\Gamma \cdot V) (\nabla \psi^-_\Gamma \cdot n) \, dS
 - \int_\Gamma \ve_+ (\nabla \psi^+_\Gamma \cdot V) (\nabla \psi^+_\Gamma \cdot n) \, dS
\nonumber 
\\
&\quad 
- \int_\Gamma \frac{\ve_-}{2} | \nabla \psi^-_\Gamma |^2 ( V\cdot n)\, dS
+ \int_\Gamma \frac{\ve_+}{2} | \nabla \psi^+_\Gamma |^2 ( V\cdot n)\, dS
\nonumber \\
&\quad  
= - \int_\Gamma \ve_\Gamma \partial_n \psi_\Gamma 
(\partial_n \psi_\Gamma^+ - \partial_n \psi_\Gamma^-) (V\cdot n) \, dS  
\nonumber \\
&\qquad 
+ \int_\Gamma \frac{\ve_+}{2} | \partial_n \psi_\Gamma^+|^2 (V\cdot n)\, dS
+ \int_\Gamma \frac{\ve_+}{2} | \nabla_\Gamma  \psi_\Gamma|^2 (V\cdot n)\, dS
\nonumber \\
&\qquad 
- \int_\Gamma \frac{\ve_-}{2} | \partial_n \psi_\Gamma^-|^2 (V \cdot n)\, dS
- \int_\Gamma \frac{\ve_-}{2} | \nabla_\Gamma  \psi_\Gamma|^2 (V \cdot n)\, dS
\nonumber \\
&\quad  
= - \int_\Gamma \ve_+ | \partial_n \psi_\Gamma^+ |^2 (V \cdot n)\, dS 
+ \int_\Gamma \ve_- | \partial_n \psi_\Gamma^- |^2 (V \cdot n) \, dS 
\nonumber \\
&\qquad 
+ \int_\Gamma \frac{\ve_+}{2} | \partial_n \psi_\Gamma^+|^2 (V\cdot n)\, dS
+ \int_\Gamma \frac{\ve_+}{2} | \nabla_\Gamma  \psi_\Gamma|^2 (V\cdot n)\, dS
\nonumber \\
&\qquad 
- \int_\Gamma \frac{\ve_-}{2} | \partial_n \psi_\Gamma^-|^2 (V \cdot n)\, dS
- \int_\Gamma \frac{\ve_-}{2} | \nabla_\Gamma  \psi_\Gamma|^2 (V \cdot n)\, dS
\nonumber \\
&\quad  
= - \frac12 \left( \frac{1}{\ve_+} - \frac{1}{\ve_-} \right)
 \int_\Gamma  |\ve_\Gamma  \partial_n \psi_\Gamma |^2 (V \cdot n)\, dS 
+ \frac{\ve_+-\ve_-}{2} \int_\Gamma | \nabla_\Gamma  \psi_\Gamma|^2 (V \cdot n)\, dS. 
\end{align}

Similarly, on each side of $\Gamma$, we have with $u_\Gamma = \phi_{\Gamma, \infty}$
or $\hat{\phi}_{\Gamma, \infty}$ that 
\begin{align*}
\nabla u_\Gamma \cdot V &= (\partial_n u_\Gamma n + \nabla_\Gamma u_\Gamma) \cdot
( (V\cdot n) n + ( I - n \otimes n) V ) 
\\
& = \partial_n u_\Gamma (V \cdot n) +  \nabla_\Gamma u_\Gamma ( I - n\otimes n) V. 
\end{align*}
Moreover, $\ve_+ \partial_n u_\Gamma^+ = \ve_- \partial_n u_\Gamma^-$ and 
$\partial_\Gamma u_\Gamma^+ =  \partial_\Gamma u_\Gamma^-$ on $\Gamma$. Therefore, 
the next three terms in \reff{dGVE_2} become
\begin{align}
\label{567}
& \frac12 \int_\Gamma  \llbracket  \ve_\Gamma   
(\nabla \hat{\phi}_{\Gamma, \infty} \cdot V) (\nabla {\phi}_{\Gamma, \infty} \cdot n) 
\rrbracket_\Gamma \, dS
+ \frac12 \int_\Gamma \llbracket \ve_\Gamma (\nabla \hat{\phi}_{\Gamma, \infty} \cdot n) 
(\nabla {\phi}_{\Gamma, \infty} \cdot V) \rrbracket_\Gamma \, dS 
\nonumber \\
&\quad 
- \frac12 \int_\Gamma \llbracket \ve_\Gamma  
(\nabla \hat{\phi}_{\Gamma, \infty} \cdot \nabla {\phi}_{\Gamma, \infty}) ( V\cdot n) 
\rrbracket_{\Gamma}\, dS 
\nonumber \\
& \quad 
=  \int_\Gamma  \llbracket  \ve_\Gamma   
\partial_n \hat{\phi}_{\Gamma, \infty} \partial_n \phi_{\Gamma, \infty} 
\rrbracket_\Gamma (V \cdot n) \, dS
\nonumber \\
& \qquad 
- \frac12 \int_\Gamma  \llbracket  \ve_\Gamma   
( \partial_n \hat{\phi}_{\Gamma, \infty} \partial_n \phi_{\Gamma, \infty} 
+ \nabla_\Gamma \hat{\phi}_{\Gamma, \infty} \cdot \nabla_\Gamma \phi_{\Gamma, \infty} )
\rrbracket_\Gamma (V \cdot n) \, dS
\nonumber \\
&\quad 
= \frac12 \int_\Gamma  \llbracket  \ve_\Gamma   
\partial_n \hat{\phi}_{\Gamma, \infty} \partial_n \phi_{\Gamma, \infty} 
\rrbracket_\Gamma (V \cdot n) \, dS
- \frac12 \int_\Gamma  \llbracket  \ve_\Gamma   
 \nabla_\Gamma \hat{\phi}_{\Gamma, \infty} \cdot \nabla_\Gamma \phi_{\Gamma, \infty} 
\rrbracket_\Gamma (V \cdot n) \, dS. 
\end{align}

It now follows from \reff{dGVE_2}--\reff{567} that 
\begin{align}
\label{dGVE3}
\delta_{\Gamma, V} E[\Gamma] 
& = - \frac12 \left( \frac{1}{\ve_+} - \frac{1}{\ve_-} \right)
 \int_\Gamma  |\ve_\Gamma  \partial_n \psi_\Gamma |^2 (V \cdot n)\, dS 
+ \frac{\ve_+-\ve_-}{2} \int_\Gamma | \nabla_\Gamma  \psi_\Gamma|^2 (V \cdot n)\, dS 
\nonumber \\
& \quad 
+ \frac12 \int_\Gamma  \llbracket  \ve_\Gamma   
\partial_n \hat{\phi}_{\Gamma, \infty} \partial_n \phi_{\Gamma, \infty} 
\rrbracket_\Gamma (V \cdot n) \, dS
- \frac12 \int_\Gamma  \llbracket  \ve_\Gamma   
 \nabla_\Gamma \hat{\phi}_{\Gamma, \infty} \cdot \nabla_\Gamma \phi_{\Gamma, \infty} 
\rrbracket_\Gamma (V \cdot n) \, dS
\nonumber \\
& \quad 
+ \int_{\Omega_+} \frac{1}{2}\left( \zeta_{\Gamma, V} 
- \nabla \phi_{\Gamma, \infty} \cdot V  \right)
B'\left( \psi_\Gamma -\frac{\phi_{\Gamma,\infty} }{2} \right) dX 
\nonumber \\
& \quad 
+ \int_\Gamma B \left( \psi_\Gamma -\frac{\phi_{\Gamma,\infty} }{2} \right) (V\cdot n)\, dS. 
\end{align}

{\it Step 5.}
We finally rewrite the volume integral above into a surface integral on the boundary $\Gamma.$ 
Recall from the beginning of Subsection~\ref{ss:DBF} that the signed distance 
function $\phi: \R^3 \to \R$ with respect to $\Gamma$ is a $C^3$-function and 
$\nabla \phi \ne 0$ in the neighborhood $\calN_0(\Gamma)$ of $\Gamma.$ 
We extend $n = \nabla \phi $ on $\Gamma$ 
to $\calN_0(\Gamma)$, i.e., we define $n = \nabla \phi $ at every point in $\calN_0(\Gamma)$. 
Note that $n \in C^2(\calN_0(\Gamma)).$ 
Since $V \in {\cal V}$ vanishes outside $\calN_0(\Gamma)$, 
both the normal component $(V\cdot n) n$ and the tangential component 
$V - (V \cdot n)n = (I-n \otimes n) V $ of $V$ are in the class of vector fields $\cal V;$
cf.\ \reff{calV}. 
Since 
\[
V = (V\cdot n) n + (I-n \otimes n) V  \quad \mbox{and} \quad (I-n \otimes n) V \cdot n = 0,
\]
we have by Lemma~\ref{l:pGinfty} that 
\begin{align*}
\zeta_{\Gamma, V} - \nabla \phi_{\Gamma, \infty} \cdot V
& =  \zeta_{\Gamma, (V\cdot n) n + (I-n \otimes n) V} - 
\nabla \phi_{\Gamma, \infty} \cdot [ (V\cdot n) n + (I-n \otimes n) V ]   
\\
& = \zeta_{\Gamma, (V\cdot n) n } - \nabla \phi_{\Gamma, \infty} \cdot (V \cdot n) n
+ \zeta_{\Gamma, (I-n\otimes n) V } - \nabla \phi_{\Gamma, \infty} \cdot 
(I -n \otimes n) V 
\\
& = \zeta_{\Gamma, (V\cdot n) n } - \nabla \phi_{\Gamma, \infty} \cdot (V \cdot n) n
\qquad \mbox{in } \Omega. 
\end{align*}
Therefore, we may assume that 
\begin{equation}
\label{VisVnn}
V = (V \cdot n)n \qquad \mbox{in } \mathcal{N}_0(\Gamma).
\end{equation}




By Lemma~\ref{l:pGinfty}, $\zeta_{\Gamma, V}^{\rm s} \in H^2(\Omega_{\rm s})$ for 
${\rm s} = -$ or $+$. Thus, by 
\reff{DwG}, $\Delta (\nabla \phi_{\Gamma, \infty} \cdot V) \in L^2(\Omega_{\rm s})$
for ${\rm s} = - $ or $+.$ Therefore, 
\begin{equation}
\label{good}
\nabla \phi_{\Gamma, \infty} \cdot V \in H^2(\Omega_{\rm s})
\qquad \mbox{for s } = - \mbox{ or } +.
\end{equation}
Recall from \reff{4notations} that 
$\psi_{\rm r} = \psi_\Gamma - \hat{\phi}_{\Gamma, \infty} \in H^1_0(\Omega)$.  
Note by Theorem~\ref{t:DBPBE} that $\Delta \psi_{\rm r}= 0$ in $\Omega_-$ and 
$\ve_+ \Delta \psi_{\rm r} = B'(\psi_\Gamma - \phi_{\Gamma, \infty}/2)$ in $\Omega_+$. 
Note also by \reff{DwG} in Lemma~\ref{l:pGinfty} that 
$\Delta (\zeta_{\Gamma, V} - \nabla \phi_{\Gamma, \infty} \cdot V) =0$ 
in $\Omega_- \cup \Omega_+$. 
We then obtain by Green's second identity with our convention that the 
normal $n$ at $\Gamma$ pointing from $\Omega_-$ to $\Omega_+$ and the fact that 
$ \llbracket  \varepsilon_\Gamma \zeta_{\Gamma, V} \partial_n \psi_{\rm r}  \rrbracket_\Gamma = 0$ 
which follows from the third equation in \reff{interfaceform} that 
twice of the volume term in \reff{dGVE3} is 
\begin{align}
\label{Q}
Q & := \int_{\Omega_+} \left(  \zeta_{\Gamma, V} - \nabla \phi_{\Gamma, \infty} \cdot V  \right)
B'\left( \psi_\Gamma  -\frac{\phi_{\Gamma,\infty} }{2} \right) dX  
\nonumber 
\\
& = \int_{\Omega_+} \ve_+ 
\left[ \left(  \zeta_{\Gamma, V} - \nabla \phi_{\Gamma, \infty} \cdot V  \right) \Delta \psi_{\rm r} 
- \psi_{\rm r} \Delta \left(  \zeta_{\Gamma, V} 
- \nabla \phi_{\Gamma, \infty} \cdot V  \right) \right] dX  
\nonumber 
\\
 & \quad + \int_{\Omega_-} \ve_- 
\left[ \left(  \zeta_{\Gamma, V} - \nabla \phi_{\Gamma, \infty} \cdot V  \right) 
\Delta \psi_{\rm r} 
- \psi_{\rm r} \Delta \left(  \zeta_{\Gamma, V} 
- \nabla \phi_{\Gamma, \infty} \cdot V  \right) \right] dX  
\nonumber 
\\
& = - \int_\Gamma \llbracket  \varepsilon_\Gamma 
\left[  \left( \zeta_{\Gamma, V} - \nabla \phi_{\Gamma, \infty} \cdot V \right) 
\partial_n \psi_{\rm r} 
- \psi_{\rm r} \partial_n \left( \zeta_{\Gamma, V} - \nabla \phi_{\Gamma, \infty} \cdot V \right) \right]  
 \rrbracket_\Gamma  \, dS
\nonumber 
\\
& = \int_\Gamma \llbracket  \varepsilon_\Gamma  
(\nabla \phi_{\Gamma, \infty} \cdot V) \partial_n \psi_{\rm r} \rrbracket_\Gamma  \, dS
+ \int_\Gamma \llbracket  \varepsilon_\Gamma  \psi_{\rm r} 
\partial_n \zeta_{\Gamma, V} \rrbracket_\Gamma  \, dS
- \int_\Gamma \llbracket  \varepsilon_\Gamma  \psi_{\rm r} 
\partial_n ( \nabla {\phi}_{\Gamma, \infty} \cdot V ) \rrbracket_\Gamma \, dS 
\nonumber 
\\
&  = Q_1 + Q_2 - Q_3. 
\end{align}

It follows from \reff{VisVnn} that 
\begin{align}
\label{Q1}
Q_1 & = 
 \int_\Gamma \llbracket  \varepsilon_\Gamma  (\nabla \phi_{\Gamma, \infty} \cdot V) 
\partial_n \psi_{\rm r} \rrbracket_\Gamma  \, dS
= \int_\Gamma \llbracket  \varepsilon_\Gamma  \partial_n \phi_{\Gamma, \infty} 
\partial_n \psi_{\rm r} \rrbracket_\Gamma ( V \cdot n)  dS. 
\end{align}

Since $\llbracket \phi_{\rm r} \rrbracket_\Gamma = 0$ and 
$\llbracket \ve_\Gamma \partial_n \phi_{\Gamma, \infty} \rrbracket_\Gamma = 0$, we have 
by Lemma~\ref{l:pGinfty} (cf.\ \reff{jump_pnwG}) that 
\begin{align}
\label{Q2}
Q_2 & = \int_\Gamma \llbracket \varepsilon_\Gamma \psi_{\rm r} 
\partial_n \zeta_{\Gamma,V} \rrbracket_\Gamma  \, dS
\nonumber \\
&= - \int_\Gamma \llbracket \varepsilon_\Gamma  \psi_{\rm r}  A_V'(0) 
\nabla \phi_{\Gamma, \infty} \cdot n  \rrbracket_\Gamma \,  dS
\nonumber\\
&=\int_\Gamma \llbracket \varepsilon_\Gamma  \psi_{\rm r}
\left[ \nabla V +(\nabla V)^T - (\nabla \cdot V)I \right] 
\nabla \phi_{\Gamma, \infty} \cdot n\rrbracket_\Gamma  \, dS
\nonumber\\
&=\int_\Gamma \llbracket \varepsilon_\Gamma  \psi_{\rm r}
\left[ \nabla V +(\nabla V)^T \right] 
\nabla \phi_{\Gamma, \infty} \cdot n\rrbracket_\Gamma \, dS
\nonumber\\
&=\int_\Gamma \llbracket \varepsilon_\Gamma  \psi_{\rm r}
\nabla \phi_{\Gamma, \infty} \cdot \left[\nabla V +(\nabla V)^T \right] n 
\rrbracket_\Gamma  \, dS. 
\end{align}
Denoting by $n^j$ the $j$th component of $n$
and noting that $ \partial_i n^j n^j = (1/2) \partial_i \| n \|^2  = 0,$
we obtain on each side of $\Gamma$ (i.e., on $\mathcal{N}_0(\Gamma) \cup \Omega_-$
and $\mathcal{N}_0(\Gamma) \cup \Omega_+$) that 
\begin{align*}
& \nabla \phi_{\Gamma, \infty} \cdot (\nabla V +(\nabla V)^T) n
\\
&\qquad  
=\partial_i \phi_{\Gamma, \infty} 
\left( \partial_j V^i+\partial_i V^j \right) n^j 
\nonumber\\
&\qquad =
 \partial_i \phi_{\Gamma, \infty} \partial_j( (V\cdot n) n^i) n^j 
+  \partial_i \phi_{\Gamma, \infty} \partial_i ((V\cdot n)n^j) n^j
\qquad \mbox{[by \reff{VisVnn}]}
\nonumber\\
&\qquad = \partial_i \phi_{\Gamma, \infty} \partial_j(V\cdot n)n^i n^j+
 \partial_i \phi_{\Gamma, \infty} (V\cdot n)\partial_jn^i n^j 
\nonumber\\
&\qquad \qquad 
+ \partial_i \phi_{\Gamma, \infty} \partial_i ( V\cdot n)n^j n^j
+  \partial_i \phi_{\Gamma, \infty} (V\cdot n)\partial_i n^j n^j
\nonumber\\
&\qquad = (\nabla\phi_{\Gamma, \infty} \cdot n) \nabla(V\cdot n)\cdot n+
\nabla \phi_{\Gamma, \infty} \cdot ((\nabla n) n)  (V\cdot n)
 + \nabla\phi_{\Gamma, \infty} \cdot\nabla(V\cdot n). 
\end{align*}
This and \reff{Q2}, together with the fact that 
$ \llbracket \varepsilon_\Gamma \nabla \phi_{\Gamma, \infty} \cdot n  \rrbracket_\Gamma =0 $ 
on $\Gamma$, lead to 
\begin{align}
\label{Q2again}
Q_2 & =\int_\Gamma \llbracket\varepsilon_\Gamma \psi_{\rm r} \nabla \phi_{\Gamma, \infty} \cdot 
(\nabla n) n \rrbracket_\Gamma (V\cdot n)\, dS
+\int_\Gamma \llbracket\varepsilon_\Gamma \psi_{\rm r}
 \nabla \phi_{\Gamma, \infty} \cdot \nabla (V\cdot n)\rrbracket_\Gamma  \, dS
\nonumber \\
& = Q_{2, 1} + Q_{2, 2}. 
\end{align}

To further simplify these terms, let us recall the surface divergence 
$\nabla_\Gamma v$ for a vector field $v$ along the boundary $\Gamma$
and its integral on $\Gamma$
\begin{align}
\label{Gdiv}
&\nabla_\Gamma \cdot v =  \nabla \cdot v - (\nabla v)  n  \cdot n, 
\\
\label{Hvn}
&
\int_\Gamma \nabla_\Gamma \cdot v\, dS = 2 \int_\Gamma H   ( v\cdot n)  \, dS, 
\end{align}
where H is the mean curvature; cf.\ \cite{DelfourZolesio_Book87} (Section 5 of Chapter 9). 

Consider the term $Q_{2, 1}$ in \reff{Q2again}.  Since $n = \nabla \phi$ is a unit vector field, 
we have $n \cdot (\nabla n) n = n^i \partial_j n^i n^j =  (1/2) n^j \partial_j (n^i n^i) = 0.$ Hence, 
on each side of $\Gamma$, we have 
\begin{equation}
\label{nnn}
\nabla \phi_{\Gamma, \infty} \cdot (\nabla n) n 
= \nabla_\Gamma  \phi_{\Gamma, \infty} \cdot (\nabla n) n.
\end{equation}
Let us denote $\alpha_\Gamma =\psi_{\rm r}
\nabla_\Gamma \phi_{\Gamma, \infty}$ and note that 
$\llbracket \alpha_\Gamma \rrbracket_\Gamma = 0.$ 
Hence $\alpha_\Gamma \in H^1(\mathcal{N}_0(\Gamma),\R^3).$ 
Note also that  $\alpha_\Gamma  \cdot n = 0.$ Thus, 
\begin{align}
\label{nTn}
(\nabla \alpha_\Gamma) n \cdot n + \alpha_\Gamma \cdot (\nabla n) n 
&= [ (\nabla \alpha_\Gamma)^T n + ( \nabla n)^T \alpha_\Gamma ] \cdot n
\nonumber \\
&= \nabla (\alpha_\Gamma \cdot n) \cdot n 
\nonumber \\
&= 0 \qquad \mbox{in } \mathcal{N}_0(\Gamma).
\end{align}
This implies that 
\begin{equation}
\label{alphaG}
(\nabla \alpha_\Gamma n) \cdot n = - \alpha_\Gamma \cdot (\nabla n)n
\in H^1(\mathcal{N}_0(\Gamma)).   
\end{equation}
By \reff{VisVnn}, we have for ${\rm s}  = - $ or $+$ that 
\begin{align*} 
 \nabla (\nabla \phi_{\Gamma, \infty} \cdot V ) \cdot n
& = \nabla ( ( \nabla \phi_{\Gamma, \infty} \cdot n)  (V \cdot n)) \cdot n
\\
& = ( \nabla ( \nabla \phi_{\Gamma, \infty} \cdot n) \cdot n)   (V \cdot n) 
+ ( \nabla \phi_{\Gamma, \infty} \cdot n)  \nabla (V \cdot n) \cdot n
\qquad \mbox{in } \Omega_{\rm s} \cap \mathcal{N}_0(\Gamma).  
\end{align*}
This, together with \reff{phiGireg} and \reff{good}, implies for ${\rm s}  = - $ or $+$ that 
\begin{equation}
\label{ddp}
( \nabla ( \nabla \phi_{\Gamma, \infty} \cdot n) \cdot n)   (V \cdot n) \in 
H^1(\Omega_{\rm s} \cap \mathcal{N}_0(\Gamma)). 
\end{equation}
Therefore, since $\nabla_\Gamma \phi_{\Gamma, \infty}
= \nabla \phi_{\Gamma, \infty} - (\nabla \phi_{\Gamma, \infty} \cdot n) n, $ 
$\Delta \phi_{\Gamma, \infty} =0$ 
in $\Omega_-$ and $\Omega_+$, and $\psi_{\rm r}$ and $\phi_{\Gamma, \infty}$
are in $W^{1,\infty}$ on each side of $\Gamma$, 
we can verify that for ${\rm s}  = - $ or $+$ 
\begin{align}
\label{dalpha}
& (\nabla \cdot \alpha_\Gamma) (V \cdot n)
 = (\nabla \psi_{\rm r}  \cdot \nabla \phi_{\Gamma, \infty} ) (V \cdot n) 
- (\nabla \psi_{\rm r} \cdot n) (\nabla \phi_{\Gamma, \infty} \cdot n) (V\cdot n)
\nonumber \\
&\qquad 
- \psi_{\rm r} ( \nabla (\nabla \phi_{\Gamma, \infty} \cdot n) \cdot n) (V \cdot n)
- \psi_{\rm r} ( \nabla \phi_{\Gamma, \infty} \cdot n) ( \nabla \cdot n) ( V\cdot n)
\in H^1(\Omega_{\rm s} \cap \mathcal{N}_0(\Gamma)). 
\end{align}
By \reff{alphaG}, \reff{dalpha}, and \reff{Gdiv} (with $\alpha_\Gamma$ replacing $v$), 
we have for ${\rm s}  = - $ or $+$ that 
\begin{equation}
\label{Tdiv}
 ( \nabla_\Gamma  \cdot \alpha_\Gamma ) (V \cdot n) 
= ( \nabla \cdot \alpha_\Gamma ) (V \cdot n) - 
( \nabla \alpha_\Gamma  n  \cdot n ) ( V \cdot n) 
\in H^1(\Omega_{\rm s} \cap \mathcal{N}_0(\Gamma)). 
\end{equation}
With all the regularity results \reff{alphaG}, \reff{dalpha}, and \reff{Tdiv}, 
we have now by \reff{nnn}, \reff{nTn}, 
and \reff{Gdiv} (with $\alpha_\Gamma$ replacing $v$) that 
\begin{align}
\label{Q21}
Q_{2, 1} 
& = \int_\Gamma \llbracket\varepsilon_\Gamma \alpha_\Gamma \cdot 
(\nabla n) n \rrbracket_\Gamma (V\cdot n)\, dS
\nonumber \\
& = - \int_\Gamma \llbracket\varepsilon_\Gamma (\nabla \alpha_\Gamma) n  \cdot n 
\rrbracket_\Gamma (V\cdot n)\, dS
\nonumber \\
& =  \int_\Gamma \llbracket\varepsilon_\Gamma ( \nabla_\Gamma \cdot \alpha_\Gamma
-  \nabla \cdot \alpha_\Gamma) \rrbracket_\Gamma (V\cdot n)\, dS. 
\end{align}

Consider now the term $Q_{2, 2}$ in \reff{Q2again}.  
On each side of $\Gamma$, 
\begin{align*}
\nabla \phi_{\Gamma, \infty} \cdot \nabla (V\cdot n)
&= \left[  (\nabla \phi_{\Gamma, \infty} \cdot n) n + \nabla_\Gamma \phi_{\Gamma, \infty} \right]
\cdot \left[ ( \nabla (V \cdot n) \cdot n ) n + \nabla_\Gamma (V \cdot n) \right]
\nonumber\\
&=(\nabla \phi_{\Gamma, \infty} \cdot n ) ( \nabla ( V \cdot n) \cdot n)
+\nabla_\Gamma \phi_{\Gamma, \infty} \cdot \nabla_\Gamma (V \cdot n). 
\end{align*}
Since $\llbracket \psi_{\rm r} \rrbracket_\Gamma=0$ and 
$\llbracket \varepsilon_\Gamma\nabla\phi_{\Gamma, \infty} \cdot n \rrbracket_\Gamma=0$, we thus have 
\begin{align}
\label{veuDD}
& \llbracket \varepsilon_\Gamma \psi_{\rm r} \nabla \phi_{\Gamma, \infty} 
\cdot \nabla (V\cdot n) \rrbracket_\Gamma
\nonumber\\
& \qquad  =\llbracket  \varepsilon_\Gamma \psi_{\rm r} 
 (\nabla \phi_{\Gamma, \infty} \cdot n) (\nabla (V\cdot n) \cdot n) \rrbracket_\Gamma 
+ \llbracket \varepsilon_\Gamma \psi_{\rm r}  \nabla_\Gamma \phi_{\Gamma, \infty} 
\cdot \nabla_\Gamma (V \cdot n)  \rrbracket_\Gamma
\nonumber\\
& \qquad =\llbracket   \varepsilon_\Gamma \alpha_\Gamma 
\cdot \nabla_\Gamma (V \cdot n)  \rrbracket_\Gamma. 
\end{align}
One can verify that on both side of $\Gamma$ 
\begin{align*}
\nabla_\Gamma \cdot ( ( V\cdot n)  \alpha_\Gamma ) 
= (V \cdot n)  \nabla_\Gamma \cdot  \alpha_\Gamma 
+ \alpha_\Gamma \cdot \nabla_\Gamma (V  \cdot n).  
\end{align*}
Consequently, we have by \reff{Q2again},  \reff{veuDD}, \reff{Hvn}, and the fact that 
$ \nabla_\Gamma \phi_{\Gamma, \infty} \cdot n = 0$ on each side of $\Gamma$  that 
\begin{align*}
Q_{2, 2} 
& = \int_\Gamma \llbracket\varepsilon_\Gamma \alpha_\Gamma 
\cdot \nabla_\Gamma (V \cdot n) \rrbracket_\Gamma \, dS
\nonumber\\
&  = \int_\Gamma \llbracket\varepsilon_\Gamma \nabla_\Gamma \cdot ( (V \cdot n)
\alpha_\Gamma ) \rrbracket_\Gamma \, dS
-\int_\Gamma \llbracket\varepsilon_\Gamma (V \cdot n) \nabla_\Gamma \cdot 
\alpha_\Gamma \rrbracket_\Gamma\, dS 
\\
& = \int_\Gamma \llbracket 2 \varepsilon_\Gamma H ((V \cdot n) \alpha_\Gamma 
\cdot n ) \rrbracket_\Gamma \, dS - \int_\Gamma \llbracket\varepsilon_\Gamma 
\nabla_\Gamma  \cdot 
\alpha_\Gamma \rrbracket_\Gamma (V \cdot n)\, dS
\\
& = - \int_\Gamma \llbracket\varepsilon_\Gamma \nabla_\Gamma \cdot 
\alpha_\Gamma  \rrbracket_\Gamma (V \cdot n) \, dS. 
\end{align*}
This, together with \reff{Q2again}, \reff{Q21}, and the notation $\alpha_\Gamma
= \psi_{\rm r} \nabla_\Gamma \phi_{\Gamma, \infty}$, implies that 
\begin{equation}
\label{Q2final}
Q_2 = 
- \int_\Gamma \llbracket\varepsilon_\Gamma \nabla \cdot \alpha_\Gamma \rrbracket_\Gamma (V\cdot n)\, dS 
= - \int_\Gamma \llbracket\varepsilon_\Gamma \nabla \cdot 
(\psi_{\rm r} \nabla_\Gamma \phi_{\Gamma, \infty}) \rrbracket_\Gamma (V\cdot n)\, dS. 
\end{equation}

Now, let us calculate the term $Q_3$ in \reff{Q}. 
Since $ V = (V \cdot n) n$ (cf.\ \reff{VisVnn}), we have from both sides of $\Gamma$ that 
\begin{align*}
\nabla(\nabla \phi_{\Gamma, \infty} \cdot V)\cdot n
& = \nabla \left( ( \nabla \phi_{\Gamma, \infty} \cdot n) ( V\cdot n) \right) \cdot n
\\
& = \nabla ( \nabla \phi_{\Gamma, \infty} \cdot n) \cdot n  ( V\cdot n) 
+ (\nabla \phi_{\Gamma, \infty} \cdot n) \nabla ( V\cdot n) \cdot n. 
\end{align*}
Since $\llbracket \varepsilon_\Gamma \nabla\phi_{\Gamma, \infty} \cdot n \rrbracket_\Gamma =0$, 
we have by \reff{Q} and \reff{ddp} that 
\begin{align}
\label{Q3final}
Q_3 & = 
\int_\Gamma \llbracket  \varepsilon_\Gamma  \psi_{\rm r} \nabla 
(\nabla \phi_{\Gamma, \infty} \cdot V) \cdot n  \rrbracket_\Gamma  \, dS
= \int_\Gamma \llbracket  \varepsilon_\Gamma  \psi_{\rm r} 
 \nabla ( \nabla \phi_{\Gamma, \infty} \cdot n) \cdot n \rrbracket_\Gamma ( V \cdot n) \, dS. 
\end{align}

It now follows from \reff{Q}, \reff{Q1}, \reff{Q2final}, and \reff{Q3final} that  
\begin{align}
\label{Qsemifinal}
Q &=  \int_\Gamma \llbracket\varepsilon_\Gamma 
[  \partial_n \phi_{\Gamma, \infty} \partial_n \psi_{\rm r} 
- \nabla \cdot ( \psi_{\rm r} \nabla_\Gamma \phi_{\Gamma, \infty} ) 
- \psi_{\rm r} \nabla ( \nabla \phi_{\Gamma, \infty} \cdot n) \cdot n] \rrbracket_\Gamma ( V \cdot n) \, dS. 
\end{align}
By the definition of the tangential gradient, 
the fact that $\Delta \phi_{\Gamma, \infty} = 0$ on both sides of $\Gamma$ 
(cf.\ \reff{DphiGi0}), and $\nabla \cdot n = 2 H$ on $\Gamma$, 
we can simplify the terms  inside the pair of brackets in \reff{Qsemifinal}. On both
sides of $\Gamma$, we have  
\begin{align*}
& \partial_n \phi_{\Gamma, \infty} \partial_n \psi_{\rm r}
- \nabla \cdot ( \psi_{\rm r} \nabla_\Gamma \phi_{\Gamma, \infty}) 
- \psi_{\rm r} \nabla ( \nabla \phi_{\Gamma, \infty}  \cdot n) \cdot n
\\
&\qquad 
= \partial_n \phi_{\Gamma, \infty} \partial_n \psi_{\rm r}
- \nabla \cdot [  \psi_{\rm r} \nabla \phi_{\Gamma, \infty} - \psi_{\rm r} 
( \nabla \phi_{\Gamma, \infty} \cdot n) n ]
- \psi_{\rm r} \nabla ( \nabla \phi_{\Gamma, \infty}  \cdot n) \cdot n
\\
& \qquad 
= \partial_n \phi_{\Gamma, \infty} \partial_n \psi_{\rm r} 
- \nabla \psi_{\rm r} \cdot \nabla \phi_{\Gamma, \infty} - \psi_{\rm r}  \Delta \phi_{\Gamma, \infty}
\\
&\qquad \qquad 
+ \nabla ( \psi_{\rm r} ( \nabla \phi_{\Gamma, \infty} \cdot n) ) \cdot n 
+ \psi_{\rm r} ( \nabla \phi_{\Gamma, \infty} \cdot n) ( \nabla \cdot n)
- \psi_{\rm r}  \nabla ( \nabla \phi_{\Gamma, \infty}  \cdot n) \cdot n
\\
& \qquad 
= \partial_n \phi_{\Gamma, \infty} \partial_n \psi_{\rm r} 
- \nabla \psi_{\rm r} \cdot \nabla \phi_{\Gamma, \infty} 
+ (\nabla \phi_{\Gamma, \infty} \cdot n) ( \nabla \psi_{\rm r} \cdot n)  
+ \psi_{\rm r} ( \nabla \phi_{\Gamma, \infty} \cdot n) ( \nabla \cdot n)
\\
& \qquad 
= 2 \partial_n \phi_{\Gamma, \infty} \partial_n \psi_{\rm r} 
- \left[ (\nabla \psi_{\rm r} \cdot n) n + \nabla_\Gamma \psi_{\rm r} \right]
 \left[ (\nabla \phi_{\Gamma, \infty} \cdot n) n + \nabla_\Gamma \phi_{\Gamma, \infty} \right]
+ 2 H \psi_{\rm r} \partial_n \phi_{\Gamma, \infty}
\\
& \qquad 
=  \partial_n \phi_{\Gamma, \infty} \partial_n \psi_{\rm r} 
-  \nabla_\Gamma  \phi_{\Gamma, \infty} \cdot  \nabla_\Gamma \psi_{\rm r}
+ 2 H \psi_{\rm r} \partial_n \phi_{\Gamma, \infty}. 
\end{align*}
Plug this into \reff{Qsemifinal}. Noting that 
$\psi_{\rm r} = \psi_\Gamma - \hat{\phi}_{\Gamma, \infty}$ and that  
all $\nabla_\Gamma \psi_{\rm r}, $ $\nabla_\Gamma \phi_{\Gamma, \infty},$
$\ve_\Gamma \partial_n ( \psi_\Gamma - \hat{\phi}_{\Gamma, \infty} ), $
and $\ve_\Gamma \partial_n  \phi_{\Gamma, \infty}$ are 
continuous across the boundary $\Gamma$, we obtain that
\begin{align}
\label{Qfinal}
Q &= 
 \int_\Gamma \llbracket\varepsilon_\Gamma ( \partial_n \psi_{\rm r} 
\partial_n \phi_{\Gamma, \infty} 
- \nabla_\Gamma  \psi_{\rm r} \cdot  \nabla_\Gamma  \phi_{\Gamma, \infty} )
\rrbracket_\Gamma  (V\cdot n) \, dS 
\nonumber \\
& =  \int_\Gamma \llbracket\varepsilon_\Gamma 
[ \partial_n   ( \psi_\Gamma - \hat{\phi}_{\Gamma, \infty} ) 
\partial_n \phi_{\Gamma, \infty}
- \nabla_\Gamma  ( \psi_\Gamma - \hat{\phi}_{\Gamma, \infty} )
\cdot  \nabla_\Gamma  \phi_{\Gamma, \infty} ] \rrbracket_\Gamma  (V\cdot n) \, dS. 
\end{align}

Finally, we obtain by \reff{dGVE3}, \reff{Q}, and \reff{Qfinal} that some of the terms in 
$\delta_{\Gamma, V}E[\Gamma]$ \reff{dGVE3} are simplified into 
\begin{align*}
& \frac12 \int_\Gamma  \llbracket  \ve_\Gamma   
\partial_n \hat{\phi}_{\Gamma, \infty} \partial_n \phi_{\Gamma, \infty} 
\rrbracket_\Gamma (V \cdot n) \, dS
- \frac12 \int_\Gamma  \llbracket  \ve_\Gamma   
 \nabla_\Gamma \hat{\phi}_{\Gamma, \infty} \cdot \nabla_\Gamma \phi_{\Gamma, \infty} 
\rrbracket_\Gamma (V \cdot n) \, dS
\nonumber \\
& \quad 
+ \int_{\Omega_+} \frac{1}{2}\left( \zeta_{\Gamma, V} 
- \nabla \phi_{\Gamma, \infty} \cdot V  \right)
B'\left( \psi_\Gamma -\frac{\phi_{\Gamma,\infty} }{2} \right) dX 
\nonumber \\
& \quad 
= \frac12 \int_\Gamma  \llbracket  \ve_\Gamma   
\partial_n \hat{\phi}_{\Gamma, \infty} \partial_n \phi_{\Gamma, \infty} 
\rrbracket_\Gamma (V \cdot n) \, dS
- \frac12 \int_\Gamma  \llbracket  \ve_\Gamma   
 \nabla_\Gamma \hat{\phi}_{\Gamma, \infty} \cdot \nabla_\Gamma \phi_{\Gamma, \infty} 
\rrbracket_\Gamma (V \cdot n) \, dS
+ \frac12 Q
\\
& \quad 
=  
\frac12 \int_\Gamma \llbracket\varepsilon_\Gamma 
( \partial_n    \psi_\Gamma  \partial_n \phi_{\Gamma, \infty} - \nabla_\Gamma   \psi_\Gamma 
\cdot  \nabla_\Gamma  \phi_{\Gamma, \infty}) \rrbracket_\Gamma  (V\cdot n) \, dS 
\\
& \quad 
= \frac12 \int_\Gamma \varepsilon_+  
 \partial_n  \psi_\Gamma^+   \partial_n \phi_{\Gamma, \infty}^+ (V \cdot n) dS
- \frac12 \int_\Gamma \varepsilon_- 
 \partial_n  \psi_\Gamma^-   \partial_n \phi_{\Gamma, \infty}^-  (V \cdot n)  dS 
\\
& \quad \quad 
-\frac{\ve_+}{2}  \int_\Gamma \nabla_\Gamma \psi_\Gamma \cdot 
 \nabla_\Gamma  \phi_{\Gamma, \infty}  (V\cdot n) \, dS 
+\frac{\ve_-}{2}  \int_\Gamma \nabla_\Gamma \psi_\Gamma \cdot
\nabla_\Gamma  \phi_{\Gamma, \infty}  (V\cdot n) \, dS 
\\
& \quad 
= \frac12 \left( \frac{1}{\ve_+} - \frac{1}{\ve_-} \right)
\int_\Gamma 
\varepsilon_\Gamma   
 \partial_n  \psi_\Gamma  \ve_\Gamma  \partial_n \phi_{\Gamma, \infty} 
( V \cdot n) dS
- \frac{\ve_+ - \ve_-}{2}  
\int_\Gamma \nabla_\Gamma \psi_\Gamma \cdot
\nabla_\Gamma  \phi_{\Gamma, \infty}  (V\cdot n) \, dS. 
\end{align*}
This and \reff{dGVE3} imply the desired formula \reff{mainG1}. 
The proof is complete. 
\end{proof}

\medskip

\noindent
{\bf Acknowledgments.}
BL was supported in part by 
the US National Science Foundation through the grant DMS-1913144, 
the US National Institutes of Health through the grant R01GM132106, 
and a 2019--2020 Lattimer Research Fellowship, Division of Physical Sciences,
University of California, San Diego. 
ZZ was supported in part by the Natural Science Foundation of Zhejiang Province, China, 
through the grant LY17A010029. 
SZ was supported in part by the National Natural Science Foundation of China (NSFC)
through the grant NSFC 21773165, 
the Natural Science Foundation of Jiangsu Province, China, through the grant BK20160302,   
and the Young Elite Scientist Sponsorship Program, Jiangsu Association for Science and Technology, 
China.

}

\medskip

\bibliography{BdryVar,charge}
\bibliographystyle{plain}
\end{document}